\documentclass[12pt,a4paper,leqno]{article}

\usepackage[centertags]{amsmath}
\usepackage{amsthm}
\usepackage[all]{xy}
\usepackage{amssymb,exscale}
\usepackage{bm}
\usepackage{paralist}
\usepackage{latexsym}
\usepackage{graphicx}
\usepackage[perpage,multiple]{footmisc}
\usepackage{mathtools}
\usepackage{appendix}
\usepackage{hyperref}

\allowdisplaybreaks

\swapnumbers
\theoremstyle{definition}
\newtheorem{theorem}[equation]{Theorem}
\newtheorem{lemma}[equation]{Lemma}

\newtheorem{definition}[equation]{Definition}
\newtheorem{proposition}[equation]{Proposition}
\newtheorem{remark}[equation]{Remark}

\newcommand\fnsep{\textsuperscript{,}}

\renewcommand{\phi}{\varphi}

\newcommand{\I}{{\rm i}}
\newcommand{\D}{\mathrm{d}}
\newcommand{\E}{\mathrm{e}}

\newcommand{\ti}{\tilde}

\renewcommand{\(}{\bigl(}
\renewcommand{\)}{\bigr)\vphantom{)}}

\newcommand{\doT}{{\boldsymbol{\cdot}}}

\renewcommand{\equiv}{\;\>\Longleftrightarrow\;\>}
\renewcommand{\Re}{\operatorname{Re}}
\renewcommand{\Im}{\operatorname{Im}}

\newcommand{\gEuler}{\ga_{\text{Euler}}}
\newcommand{\fnz}{\operatorname{fnz}}

\newcommand{\Leak}{\operatorname{Lk}}

\newcommand{\sgn}{\operatorname{sgn}}
\newcommand{\const}{\operatorname{const}}

\newcommand{\One}{{1\hskip-2.5pt{\rm l}}}

\newcommand{\ccirc}{\mathbin{\kern-0.13em\circ}}

\newcommand{\bal}{{\boldsymbol{\alpha}}}

\newcommand{\eps}{\varepsilon}

\newcommand{\si}{\sigma}
\newcommand{\ga}{\gamma}

\newcommand{\om}{\omega}
\newcommand{\Om}{\Omega}
\newcommand{\de}{\delta}
\newcommand{\De}{\Delta}
\newcommand{\al}{\alpha}
\newcommand{\be}{\beta}
\newcommand{\cO}{\mathcal O}

\newcommand{\la}{\lambda}

\newcommand{\Ex}{\mathbb E\,}
\newcommand{\R}{\mathbb R}
\newcommand{\C}{\mathbb C}
\newcommand{\Z}{\mathbb Z}
\newcommand{\T}{\mathbb T}
\renewcommand{\Pr}[1]{\mathbb{P}\mskip1.5mu\(\mskip1.5mu#1\mskip1.5mu\)}
\newcommand{\PR}[1]{\mathbb{P}\mskip1.5mu\bigg(\mskip1.5mu#1\mskip1.5mu\bigg)}

\newcommand{\cE}[2]{\mathbb{E}\mskip1.5mu\(\mskip1.5mu#1\mskip1.5mu
 \big|\mskip1.6mu#2\mskip1.6mu\)}
\newcommand{\CE}[2]{\mathbb{E}\mskip1.5mu\Big(\mskip0.8mu#1\mskip1.5mu
 \Big|\mskip1.6mu#2\mskip1.5mu\Big)}

\newcommand{\VM}{\operatorname{VM}}
\newcommand{\mes}{\mathrm{mes}}

\newcommand{\wC}{w_{\scriptscriptstyle \C}}
\newcommand{\gC}{\gamma_{\scriptscriptstyle \C}}
\newcommand{\gR}{\gamma_{\scriptscriptstyle \R}}
\newcommand{\Fswap}{F_{\text{swap}}}

\newcommand{\sif}{$\sigma$\nobreakdash-\hspace{0pt}field}
\newcommand{\transformation}[1]{$(#1)$\nobreakdash-\hspace{0pt}transformation}
\newcommand{\component}[1]{$#1$\nobreakdash-\hspace{0pt}component}
\newcommand{\dimensional}[1]{$#1$\nobreakdash-\hspace{0pt}dimensional}
\newcommand{\splt}[1]{$#1$\nobreakdash-\hspace{0pt}split}
\newcommand{\splittable}[1]{$#1$\nobreakdash-\hspace{0pt}splittable}

\newcommand{\splittability}[1]{$#1$\nobreakdash-\hspace{0pt}splittability}
\newcommand{\leak}[1]{$#1$\nobreakdash-\hspace{0pt}leak}
\newcommand{\valued}[1]{$#1$\nobreakdash-\hspace{0pt}valued}
\newcommand{\noisy}[1]{$#1$\nobreakdash-\hspace{0pt}noisy}
\newcommand{\noisiness}[1]{$#1$\nobreakdash-\hspace{0pt}noisiness}

\newcommand{\myatop}[2]{{\genfrac{}{}{0pt}{}{#1}{#2}}}

\DeclareMathOperator*{\esssup}{ess\,sup}

\begin{document}

\title{Linear response and moderate deviations:\\ hierarchical
  approach. V}

\author{Boris Tsirelson}

\date{}
\maketitle

\begin{abstract}
The Moderate Deviations Principle (MDP) is well-understood for sums of
independent random variables, worse understood for stationary random
sequences, and scantily understood for random fields.
Here it is established for some planary random fields of the form $ X_t =
\psi(G_t) $ obtained from a Gaussian random field $ G_t $ via a function $ \psi
$, and consequently, for zeroes of the Gaussian Entire Function.
\end{abstract}

\setcounter{tocdepth}{2}
\tableofcontents

\pagebreak

\numberwithin{equation}{section}

\section[Introduction]
  {\raggedright Introduction}
\label{sect1}
We start with an application of splittability (defined in \cite{IV}) to random
complex zeroes.

The theory of random complex zeroes, surveyed by Nazarov and Sodin in Part 1 of
\cite{NS-ICM} (see also \cite{NS-AMS}), investigates, on one hand, asymptotic
probability distributions of relevant random variables (asymptotic normality
\cite[Sect.~1.3.1]{NS-ICM}, \cite[Sect.~1.3]{NS-IMRN} and its violation
\cite[Sect.~1.6]{NS-IMRN}, \cite[Sect.~1.3.2]{NS-ICM}), and on the other hand,
small probabilities of relevant rare events (\cite[Sect.~1.4]{NS-ICM}), first of
all, large and moderate deviations for relevant asymptotically normal random
variables.

The Gaussian Entire Function (G.E.F., for short)
\begin{equation}\label{1.1}
F(z) = \sum_{k=0}^\infty \zeta_k \frac{z^k}{\sqrt{k!}} \, ,
\end{equation}
where $ \zeta_0, \zeta_1, \dots $ are i.i.d.\ complex-valued random variables,
each having a two-dimensional rotation-invariant normal distribution with $ \Ex
|\zeta_k|^2 = 1 $, is distinguished by the distribution invariance of its
(random) set of zeroes $ Z_F = \{ z : F(z)=0 \} $ w.r.t.\ isometries of the
complex plane (see \cite[before
Sect.~1]{NS-ICM}, \cite[Introduction]{NS-IMRN}, \cite[p.~376]{NS-AMS}). The
counting measure $ n_F $ on $ Z_F $ satisfies $ n_F = \frac1{2\pi} \De \log |F|
$ (with the Laplacian taken in the sense of distributions). The centered (that
is, of mean zero) stationary random field $ X_z = \log |F(z)| - \frac12 |z|^2
+ \frac12 \gEuler $\,\footnote{
 See \cite[Lemma 2.1]{NS-IMRN}; here $ \gEuler = 0.577\dots $ is the
 Euler(-Mascheroni) constant.}
appears to be splittable (see Theorem \ref{5.25}). Thus, by \cite[Theorem 1.5
and Corollary 1.6]{IV}, for every continuous compactly supported function $ h
: \C \to \R $ holds
\begin{gather*}
\label{1.2}\stepcounter{equation}
\tag*{$\myatop{\displaystyle(1.2)\hfill}{\text{\footnotesize linear response}\hfill}$}
\lim_\myatop{ r\to\infty, \la\to0 }{ \la\log^2 r \to 0 } \frac1{ r^2 \la^2}
 \log \Ex \exp \la \int_\C \!\! h \Big( \frac1r z \Big) X_z \, \D z
 = \frac12 \|h\|^2 \si_X^2, \\
\label{1.3}\stepcounter{equation}
\tag*{$\myatop{\displaystyle(1.3)\hfill}{\text{\footnotesize moderate deviations}\hfill}$}
\lim_{\textstyle\myatop
 { r\to\infty, c\to\infty }
 { ( c\log^2 r ) / r \to 0 }
}
\frac1{c^2} \log \PR{ \! \int_\C \!\! h \Big( \frac1r z \Big) X_z \, \D
 z \ge c \|h\| \si_X r \!\! } = -\frac12,
\end{gather*}
where $ \|h\|^2 = \int_{\C} h^2(z) \, \D z $,\footnote{%
 Here $ \D z $ means $ \D(\Re z) \D(\Im z) $, that is, Lebesgue measure on $\C$
 rather than the usual complex-valued 1-form $dz$ on $\C$ (the 1-form is never
 used in this text).}
and
\begin{equation}
\si_X = \frac12 \sqrt{ \pi \zeta(3) }
\end{equation}
(where $ \zeta(\cdot) $ is Riemann's zeta-function), as we'll see soon.
Also, by \cite[Corollary 1.7]{IV}, the distribution of the random variable
$ \frac1r \int_\C h(\frac1r z) X_z \, \D z $ converges (as $ r\to\infty $) to
the normal distribution $ N(0,\|h\|^2 \si_X^2) $ (asymptotic normality).

Given a compactly supported $C^2$-smooth function $ h : \C \to \R $, we denote
the so-called linear statistic (see \cite[Sect.~1]{NS-ICM})
$ \sum_{z\in Z_F} h\( \frac1r z \) = \int_\C h\( \frac1r z \) n_F(\D z) $ by $
n_F(r,h) $, or just $ n(r,h) $. Taking into account that $ n_F
= \frac1{2\pi} \De \log |F| = \frac1{2\pi} ( 2 + \De X ) $ we have $ n(r,h)
= \frac1{2\pi} \int_\C h\( \frac1r z \) ( 2 + \De X_z ) \, \D z
= \frac{r^2}\pi \int_\C h(z) \, \D z + \frac1{2\pi r^2} \int_\C \De h ( \frac1r
z) X_z \, \D z $.\footnote{%
 Here $ \De h ( \frac1r z) $ denotes the Laplacian of $h$ at $ \frac1r z $; the
 Laplacian of the function $ z \mapsto h ( \frac1r z) $ at $z$ is rather
 $ \frac1{r^2} \De h ( \frac1r z) $.}
The first term is the expectation, $ \Ex n(r,h) $; the second term is the
deviation, $ n(r,h) - \Ex n(r,h) $.

Now we transfer the linear response, moderate deviations, and asymptotic
normality from $X$ to $ n_F - \Ex n_F $, getting
\begin{gather*}
\label{1.5}\stepcounter{equation}
\tag*{$\myatop{\displaystyle(1.5)\hfill}{\text{\footnotesize linear response}\hfill}$}
\lim_\myatop{ r\to\infty, \la\to0 }{ \la \log^2 r \to 0 }
\frac1{ r^2 \la^2 } \log \Ex \exp \la r^2 \( n(r,h) - \Ex n(r,h) \) =
\frac{ \si^2 }2 { \| \De h \|^2 } \, , \\
\label{1.6}\stepcounter{equation}
\tag*{$\myatop{\displaystyle(1.6)\hfill}{\text{\footnotesize moderate deviations}\hfill}$}
\lim_\myatop{ r\to\infty, c\to\infty }{ (c\log^2 r)/r \to0 }
\! \frac1{c^2} \log \PR{ \! n(r,h) - \Ex n(r,h) \ge \frac{ c\si }{ r } \| \De h \| \! }
= -\frac12 , \!\!\!
\end{gather*}
where $ \si = \frac1{2\pi} \si_X $.
Also, the distribution of the random variable $ r n(r,h) $ converges (as $
r\to\infty $) to the normal distribution $ N(0,\|\De h\|^2 \si^2) $ (asymptotic
normality). All moments also converge. Comparing it with \cite[(1.1)]{NS-ICM} we
get $ \si = \frac14 \sqrt{ \frac{\zeta(3)}{\pi} } $, and hence $ \si_X
= \frac12 \sqrt{ \pi \zeta(3) } $ as promised above.

The moderate deviations formula \hyperref[1.6]{(1.6)} and the asymptotic
normality above are combined in \cite[Th.~1.4]{NS-ICM}.\footnote{%
 There, however, $ n(r,h) $ is written, mistakenly, instead of $ n(r,h) - \Ex
 n(r,h) $.}
They both, and \hyperref[1.5]{(1.5)} as well, were published
in \cite[Introduction]{2008} without the value of $\si$, and with rather
sketchy proofs. Detailed (and hopefully, understandable) proofs are available
now.

The crucial point above is splittability of the random field $ X_z $. It is
ensured by sufficient conditions for splittability that are main results of this
paper. For $ z = t_1+\I t_2 $ we have $ X_z = \log |F(z)| - \frac12 |z|^2
+ \frac12 \gEuler = \psi(G_{t_1,t_2}) $ where $ G_{t_1,t_2} = F^*(z) =
F(z) \E^{-|z|^2/2} $ is a complex-valued Gaussian random field on $ \R^2 $, and
$ \psi(z) = \log |z| + \frac12 \gEuler $. This function $ \psi : \C \to \R $
satisfies the ``smeared Lipschitz condition'' (see Lemmas \ref{3.5}, \ref{3.15})
\begin{equation}\label{1.*}
\log \int_{\R^2} \exp | \psi(x+y) - \psi(x-y) | \, \ga_r(\D x) \le \const_\si
|y|
\end{equation}
for all $ y,r \in \R^2 $; here $ \ga_r(\D x) = (2\pi\si^2)^{-1} \exp \( -|x-r|^2
/ (2\si^2) \) \, \D x $ for a given $ \si \in (0,\infty) $; the constant depends
on $\si$ only.

The real-valued non-Gaussian random field $ \psi(G) $ is stationary, but the
complex-valued Gaussian random field $G$ is not. This difficulty is treated in
Sect.~\ref{5c}. For now we restrict ourselves to stationary real-valued $G$ of
the form $ G = \phi * w $ where $w$ is the white noise; that is,
$ G_t = \int_{\R^2} \phi(t-s) w_s \, \D s $,
or, in terms of the covariance function,
\[
\Ex G_t G_{t+r} = \Ex G_0 G_r = \int_{\R^2} \phi(-s) \phi(r-s) \, \D s
= \int_{\R^2} \phi(s) \phi(s+r) \, \D s
\]
for all $ t,r \in \R^2 $; and $ \Ex G_t = 0 $, of course. It appears that the
condition
\begin{equation}\label{1.**}
\esssup_{t\in\R^2} (1+|t|)^\bal \phi^2(t) < \infty \qquad \text{for some } \bal
> 6
\end{equation}
on $\phi$ (together with Condition \eqref{1.*} on $\psi$) is sufficient for
splittability of $ \psi(G) $ (see Theorem \ref{5.14}).

The first part, ``dimension one'' (Sections \ref{sect2}, \ref{sect3}), is
intended to be a helpful preparation to the main, second part ``dimension two''
(Sections \ref{sect4}, \ref{sect5}).

\section[Dimension one: basics]
  {\raggedright Dimension one: basics}
\label{sect2}
In this section, by a \emph{random process} we mean a special case, for $d=1$,
of a random field on $\R^d$ as defined in \cite[Sect.~1]{IV}; that is, a random
process is a random measure on $\R$; and by \emph{measure} we mean a locally
finite signed Borel measure (unless stated otherwise); by $\VM$ we denote the
space of all such measures\footnote{%
 Why ``$\VM$''? Since these are vector measures with values in
 the \dimensional{1} space $\R$; more general case will be needed in the next
 section.}
(a vector space endowed with a \sif, see \cite[Sect.~1]{IV}). Thus, a random
process $X$ is a measurable map  $X:\Om\to\VM$ denoted (if only for convenience)
by $ B \mapsto \int_B X_t \, \D t $ (rather than $ X(\om)(B) $).

By a \emph{transformation} we mean a measurable map $ F : \VM \to \VM $.

An example: convolution, $ \nu \mapsto \mu*\nu $, with a given compactly
supported $ \mu \in \VM $ is a transformation. A shift is a special case of the
convolution (when $\mu$ is the unit mass at a single point). On the other hand,
a shift is a special case of a transformation $F_\al$ that corresponds to a
homeomorphism $ \al : \R \to \R $ by $ (F_\al \nu)(B) = \nu(\al(B)) $.

These transformations are linear (on the space of measures). The variation is an
example of a nonlinear transformation. Here is a more interesting
example. Taking an absolutely continuous $ \mu \in \VM $ we get an absolutely
continuous convolution $ \mu * \nu $,
\[
\frac{\D(\mu*\nu)}{\D\,\mes} = \frac{\D\mu}{\D\,\mes} * \nu \qquad \text{ ($\mes$
being the Lebesgue measure)}
\]
and we may take
\begin{equation}\label{*}
\frac{\D F(\nu)}{\D\,\mes} (t) = \psi \Big( \frac{\D(\mu*\nu)}{\D\,\mes} (t) \Big)
\end{equation}
for a given bounded Borel measurable function $ \psi : \R \to \R $.

Given a random process $ X : \Om \to \VM $ and a transformation $F$, we get
another random process $ \Om \to \VM $; it could be denoted by $ F \circ X $,
but we prefer to denote it by $ F(X) $.

\begin{sloppypar}
Accordingly, the shift of $X=(X_t)_{t\in\R}$ by $s$ is denoted by
$(X_{t+s})_{t\in\R}$; the convolution of $X=(X_t)_{t\in\R}$ and $\mu$ is the process
$ \(\int X_{t-s} \mu(\D s) \)_{t\in\R} $; transformation \eqref{*} gives
$ \( \psi \( \int X_{t-s} \mu(\D s) \) \)_{t\in\R} $; the variation of $X$ may
be denoted by $(|X_t|)_{t\in\R}$; etc.
\end{sloppypar}

If $X$ and its shift $(X_{t+s})_{t\in\R}$ are identically distributed (for every
$ s \in \R $), we say that $X$ is stationary.

Recall that a split\footnote{%
 \cite[Def.~1.1]{IV}.}
of $X$ is a triple of random processes $ X^0, X^-, X^+ $ (on some probability
space $ \ti\Om $ possibly different from $ \Om $) such that the two random
processes $ X^-, X^+ $ are (mutually) independent and the four random processes
$ X, X^0, X^-, X^+ $ are identically distributed.

Note that only distributions matter, namely, the distribution of $X$ and the
distribution of a split, that is, the joint distribution of $ X^0, X^-, X^+
$. Thus, for a stationary $X$, a split of $X$ is also a split of
$(X_{t+s})_t$. Note also that a split of a stationary process need not be
stationary (and in fact, stationary splits are useless).

If $ X,Y $ are identically distributed, then $ F(X) $ and $ F(Y) $ are
identically distributed (evidently; for arbitrary transformation $F$).

If $ X,Y $ are independent, then $ F(X) $ and $ F(Y) $ are independent
(evidently).\footnote{%
 This generalizes readily to $ F_1(X_1), F_2(X_2), \dots $}

Therefore, if $ (X^0,X^-,X^+) $ is a split
of $X$, then
\begin{equation}\label{2.15}
\( F(X^0), F(X^-), F(X^+) \) \quad \text{is a split of} \quad F(X) \, .
\end{equation}
In particular, the shifted split $ \( (X^0_{t+s})_t, (X^-_{t+s})_t,
(X^+_{t+s})_t \) $ is a split of the shifted process $ (X_{t+s})_t $. 

Recall that the leak $Y$ of a split $ (X^0,X^-,X^+) $ of $X$ is defined
by\footnote{%
 \cite[Def.~1.2]{IV}.}\fnsep\footnote{%
 Here and throughout, $ \sgn t $ is $-1$ for $t<0$ and $+1$ for $t\ge0$.}
\[
Y_t = X_t^{\sgn t} - X_t^0 = \begin{cases}
  X^-_t - X^0_t &\text{when } t<0,\\
  X^+_t - X^0_t &\text{when } t\ge0;
\end{cases}
\]
more formally, $ \int_A Y_t \, \D t = \int_A (X^-_t-X^0_t) \, \D t $ for
every bounded Borel measurable set $ A \subset (-\infty,0) $, and $ \int_A
Y_t \, \D t = \int_A (X^+_t-X^0_t) \, \D t $ for every bounded Borel measurable
set $ A \subset [0,\infty) $.

In addition, we say that $Y$ is a \emph{leak for} $X$, if there exists a split
of $X$ such that the leak of this split is distributed like $Y$.\footnote{%
 That is, $Y$ and this leak are identically distributed.}

Informally, a split is useful when its leak is small.

Note that the leak $ ( X^{\sgn t}_{t+s} - X^0_{t+s} )_t $ of the shifted split
differs from the shifted leak $ (Y_{t+s})_t = ( X^{\sgn (t+s)}_{t+s} - X^0_{t+s}
)_t $ even if $X$ is stationary. Shifting back the former (the leak of the
shifted split) we get the process $ { ( X^{\sgn (t-s)}_t - X^0_t )_t } $; let us
call it the leak (of the split $ (X^0,X^-,X^+) $ of $X$) around $s$, or a leak
for $X$ around $s$.

Can we obtain the leak
$Y_1$ of $ \(F(X^0),F(X^-),F(X^+)\) $ as $ F(Y_0) $ where $ Y_0 $ is a leak of
$(X^0,X^-,X^+)$? The answer is negative in general (see below), but affirmative
in a useful special case, when $F$ is just a multiplication:
\[
F(X) = fX \, , \quad \text{that is, } F(X) = \( f(t) X_t \)_{t\in\R}
\]
where $ f : \R \to \R $ is locally bounded and Borel measurable. (More formally,
we multiply a measure $ \nu \in \VM $ by the function $f$ as follows: $ f\nu :
B \mapsto \int_B f \, \D\nu $.) The leak of $ ( fX^0, fX^-, fX^+ ) $ is equal to
$ fY_0 $ where $ Y_0 $ is a leak of $(X^0,X^-,X^+)$, see \cite[Lemma
2.2]{IV}. This happens due to linearity and locality of the transformation.

Beyond this special case it may happen that $Y_0$ vanishes (always, everywhere),
while $Y_1$ does not, even for a linear\footnote{%
 But not local.}
$F$ and stationary $X$. For example, let $X$ be the centered Poisson point
process,\footnote{%
 See \cite[page 2]{I}}
and $F$ the convolution with the Lebesgue measure on $(0,1)$. This $X$ has a
leakless split due to independent restrictions of $X$ to $(-\infty,0)$ and
$[0,\infty)$, while $F(X)$ has not (because its restrictions to $(-\infty,0)$
and $[0,\infty)$ fail to be independent).

Nevertheless, in dimension $1$ we can use the split $ \( F(X^0), F(X^-),
F(X^+) \) $ of $F(X)$ for proving that, under some conditions,
splittability of $X$ implies splittability of $ F(X) $, by checking the
inequality $ \Ex \exp \int_{\R} |Y_t| \, \D t \le 2 $ for the leak $Y$ of $ \(
F(X^0), F(X^-), F(X^+) \) $. Few results of this kind follow after some
preparation.

The definition of splittability \cite[Def.~1.3, 1.4]{IV}, formulated for random
fields, simplifies a lot in the special case of random processes (dimension
one). In this case we have (in that definition) $ d=1 $, $ k=1 $, and (in Item
(c) there) $ i=1 $, $\ti X_t = X_{t+r} $, $ Y_t = \ti Y_t $, and the
definition boils down to the following.

\begin{definition}\label{2.spli}
A random process $ (X_t)_{t\in\R} $ is \splittable{1}, if

(a) for every $ s \in \R $,
\[
\Ex \exp \int_{[s,s+1)} |X_t| \, \D t \le 2 \, ;
\]

(b) for all $ t \in \R $,
\[
\Ex X_t = 0
\]
(more formally, $ \Ex \int_A X_t \, \D t = 0 $ for every bounded Borel
measurable set $ A \subset \R $);

(c) for every $ r \in \R $ there exists a leak $Y$ for $(X_{t+r})_{t\in\R} $
such that
\[
\Ex \exp \int_\R |Y_t| \, \D t \le 2 \, .
\]

In addition, $X$ is \splittable{C} if $ \frac1C X = \(\frac1C X_t\)_{t\in\R} $
is \splittable{1}; and $X$ is splittable, if it is \splittable{C} for some $
C \in (0,\infty) $.
\end{definition}

\begin{remark}\label{2.rem}
For a stationary random process $X$ the definition simplifies further:

(a) $ \Ex \exp \int_{[0,1)} |X_t| \, \D t \le 2 $;

(b) $ \Ex \int_{[0,1)} X_t \, \D t = 0 $ (or just $ \Ex X_0 = 0 $, if $ X_0 $ is
well-defined);

(c) $ \Ex \exp \int_\R |Y_t| \, \D t \le 2 $ for some leak $ Y $ for $ X $.
\end{remark}

\begin{proposition}\label{2.1}
Let $ X = (X_t)_{t\in\R} $ be a splittable random process, and $ Y_t = |X_t|
- \Ex |X_t| $. Then $ Y = (Y_t)_{t\in\R} $ is a splittable random process.
\end{proposition}

Here, by $ Y_t = |X_t| - \Ex |X_t| $ we mean that for every bounded Borel
measurable set $ B \subset \R $ holds $ \int_B Y_t \, \D t = \int_B |X_t| \, \D
t - \Ex \int_B |X_t| \, \D t $ where $ \int_B |X_t| \, \D t $ is the variation
on $B$ of the given measure. Finiteness of $ \Ex \int_B |X_t| \, \D t $ is
ensured by splittability of $X$, see Def.~\ref{2.spli}(a).

\begin{proof} \let\qed\relax
We prove that, moreover, if $X$ is \splittable{C}, then $Y$
is \splittable{2C}. WLOG, $ C=1 $. Denote $ m_t = \Ex |X_t| $, then $ Y_t =
|X_t| - m_t $. We inspect Def.~\ref{2.spli}(a,b,c).

(a) We know that $ \Ex \exp \int_{[s,s+1)} |X_t| \, \D t \le 2 $, and check
that\linebreak
$ \Ex \exp \frac12 \int_{[s,s+1)} |Y_t| \, \D t \le 2 $. We have $
|Y_t| \le |X_t| + m_t $;\,\footnote{%
 More formally: the variation of the sum (or difference) of two signed measures
 does not exceed the sum of variations.}
\begin{multline*}
\exp \frac12 \int_{[s,s+1)} |Y_t| \, \D t \le \exp \frac12 \Big( \int_{[s,s+1)}
 |X_t| \, \D t + \int_{[s,s+1)} m_t \, \D t \Big) \le \\
\le \frac12 \exp \int_{[s,s+1)}
 |X_t| \, \D t + \frac12 \exp \int_{[s,s+1)} m_t \, \D t
\end{multline*}
by convexity of $\exp$; take the expectation and note that $ \exp \int_{[s,s+1)}
m_t \, \D t = \exp \Ex \int_{[s,s+1)} |X_t| \, \D t \le \Ex \exp \int_{[s,s+1)}
|X_t| \, \D t $.

(b) Trivial.

(c) The shifted process $ \ti X = (X_{t+r})_t $ has a split $ (\ti X^0, \ti
X^-, \ti X^+) $ such that $ \Ex \exp \( \int_{(-\infty,0)} | \ti X^-_t - \ti
X^0_t | \, \D t + \int_{[0,\infty)} | \ti X^+_t - \ti X^0_t | \, \D t \) \le 2
$; we need such a split of the process $ \frac12 \ti Y $ where $ \ti Y_t =
|\ti X_t| - m_{t+r} $. We take $ \ti Y^0_t = | \ti X^0_t| - m_{t+r} $, $ \ti Y^-_t =
| \ti X^-_t| - m_{t+r} $, $ \ti Y^+_t = | \ti X^+_t| - m_{t+r} $, then $ (\ti
Y^0, \ti Y^-, \ti Y^+) $ is a split of $\ti Y$. We have $ | \ti Y^-_t - \ti Y^0_t |
= \big| \( |\ti X^-_t|-m_{t+r} \) - \( |\ti X^0_t|-m_{t+r} \) \big| = \big| |\ti
X^-_t| - |\ti X^0_t| \big| \le | \ti X^-_t - \ti X^0_t | $, and similarly $
| \ti Y^+_t - \ti Y^0_t | \le | \ti X^+_t - \ti X^0_t | $, whence (waiving the
$\frac12$)
\begin{multline*}
\int_{(-\infty,0)} | \ti Y^-_t - \ti Y^0_t | \, \D t + \int_{[0,\infty)} | \ti
 Y^+_t - \ti Y^0_t | \, \D t \le \\
\le \int_{(-\infty,0)} | \ti X^-_t - \ti X^0_t | \, \D t + \int_{[0,\infty)} | \ti
 X^+_t - \ti X^0_t | \, \D t \, .  \qquad\qquad \rlap{$\qedsymbol$}
\end{multline*}
\end{proof}

\begin{proposition}\label{2.2}
Let $\mu$ be a finite signed measure on $\R$ such that $ \int_{\R} |t| \,
|\mu|(\D t) < \infty $ (here $ |\mu| $ is the variation of $\mu$), $ X =
(X_t)_{t\in\R} $ a splittable random process, and $ Y_t = \int_{\R}
X_{t-s} \, \mu(\D s) $. Then $ Y = (Y_t)_{t\in\R} $ is a splittable random
process.
\end{proposition}

\begin{proof} \let\qed\relax
We inspect Def.~\ref{2.spli}(a,b,c).

(a) We know that $ \Ex \exp \int_{[s,s+1)} |X_t| \, \D t \le 2 $, and check
that\linebreak
$ \Ex \exp \frac1C \int_{[s,s+1)} |Y_t| \, \D t \le 2 $ for $ C \ge
2|\mu|(\R) $. We have $ |Y_t| \le \int_{\R} |X_{t-s}| \, |\mu|(\D s) $;
\begin{multline*}
\int_{[s,s+1)} |Y_t| \, \D t \le \int_{\R} |X_r| \, |\mu| [s-r,s+1-r) \, \D r
 = \sum_{n=-\infty}^\infty \int_{[n,n+1)} \dots \le \\
\le \sum_n |\mu| [s-n-1,s+1-n) \int_{[n,n+1)} |X_r| \, \D r \, ;
\end{multline*}
taking into account that $ \sum_n |\mu| [s-n-1,s+1-n) = 2|\mu|(\R) \le C $ and
using the general inequality
\begin{equation}
\Ex \exp \frac1C \sum_k a_k \xi_k \le \sup_k \Ex \exp \xi_k \quad \text{whenever
} \sum_k a_k \le C
\end{equation}
for arbitrary $ a_k \ge 0 $ and random $ \xi_k \ge 0 $ (it holds for $ C
= \sum_k a_k $ by convexity of $\exp$), we get
\begin{gather*}
\Ex \exp \frac1C \int_{[s,s+1)} \!\!\!\!\! |Y_t| \, \D t \le
\sup_n \Ex \! \exp \int_{[n,n+1)} \!\!\!\!\! |X_r| \, \D r \le 2 \, .
\end{gather*}

(b) Trivial.

(c) The shifted process $ \ti X = (X_{t+r})_t $ has a split $ (\ti X^0, \ti
X^-, \ti X^+) $ such that $ \Ex \exp \( \int_{(-\infty,0)} | \ti X^-_t - \ti
X^0_t | \, \D t + \int_{[0,\infty)} | \ti X^+_t - \ti X^0_t | \, \D t \) \le 2
$; we'll find such a split of the process $ \frac1C \ti Y = (\frac1C Y_{t+r})_t $
for $ C \ge 8 \int_{\R} (|r|+1) \, |\mu|(\D r) $. We take
$ \ti Y^0 = \mu * \ti X^0 $, $ \ti Y^- = \mu * \ti X^- $, $ \ti Y^+ = \mu * \ti
X^+ $, then $ (\ti Y^0, \ti Y^-, \ti Y^+) $ is a split of $\ti Y$.
 
Generally,
$ \Ex \exp \frac12 ( \xi_1 + \xi_2 ) \le \max ( \Ex \exp \xi_1, \Ex \exp \xi_2 ) $.
Thus, the needed inequality $ \Ex \exp \frac1C \( \int_{(-\infty,0)} | \ti Y^-_t
- \ti Y^0_t | \, \D t + \int_{[0,\infty)} | \ti Y^+_t - \ti Y^0_t | \, \D
t \) \le 2 $ reduces to two inequalities $ \Ex \exp \frac2C \int_{(-\infty,0)}
| \ti Y^-_t - \ti Y^0_t | \, \D t \le 2 $, $ \Ex \exp \frac2C \int_{[0,\infty)}
| \ti Y^+_t - \ti Y^0_t | \, \D t \le 2 $. We prove the latter; the former is
similar.

We have $ \ti Y^+ - \ti Y^0 = \mu * (\ti X^+ - \ti X^0) $;
$ | \ti Y^+ - \ti Y^0 | \le |\mu| * |\ti X^+ - \ti X^0| $;
\[
\int_{[0,\infty)} | \ti Y^+_t - \ti Y^0_t | \, \D t \le \int_{\R} |\ti X^+_r - \ti
X^0_r| |\mu| [-r,\infty) \, \D r = \int_{(-\infty,0)} \dots \D r
+ \int_{[0,\infty)} \dots \D r \, ;
\]
we prove that $ \Ex \exp \frac4C \int_{(-\infty,0)} \dots \D
r \le 2 $ and $ \Ex \exp \frac4C \int_{[0,\infty)} \dots \D r \le 2 $.

The latter holds for $ C \ge 4|\mu|(\R) $, since $ \int_{[0,\infty)} |\ti X^+_r
- \ti X^0_r| |\mu| [-r,\infty) \, \D r \le |\mu|(\R) \int_{[0,\infty)} |\ti
X^+_r - \ti X^0_r| \, \D r $. It remains to prove the former for $ C \ge
8 \int_{\R} (|r|+1) \, |\mu|(\D r) $. We have
\begin{multline*}
\int_{(-\infty,0)} |\ti X^+_r - \ti X^0_r| |\mu| [-r,\infty) \, \D r \le
\int_{(-\infty,0)} ( |\ti X^+_r| + |\ti X^0_r| ) |\mu| [-r,\infty) \, \D r \le \\
 \sum_{n=1}^\infty |\mu|[n-1,\infty) \int_{[-n,-n+1)} ( |\ti X^+_r| + |\ti X^0_r| ) \, \D r \, ;
\end{multline*}
$ \sum_{n=1}^\infty |\mu|[n-1,\infty) \le \int_{[0,\infty)} (r+1) \, |\mu|(\D
r) \le \frac C 8 $;
\begin{multline*}
\Ex \exp \frac4C \int\limits_{(-\infty,0)} |\ti X^+_r - \ti X^0_r| |\mu| [-r,\infty) \, \D
 r \le \\
\Ex \exp \frac4C \sum_{n=1}^\infty |\mu|[n-1,\infty)
\int_{(-\infty,0)} ( |\ti X^+_r| + |\ti X^0_r| ) \, \D r \le \\
\sup_n \max \bigg( \Ex \exp \int_{[-n,-n+1)} |\ti X^+_r| \, \D
r, \, \Ex \exp \int_{[-n,-n+1)} |\ti X^0_r| \, \D r \bigg) \le 2 \, . \qquad \rlap{$\qedsymbol$}
\end{multline*}
\end{proof}

The general assumptions on the transformation $F$ may be relaxed. Given $X$, we
may use $F$ defined only on a subset of $ \VM $ provided that $X$ belongs
to this subset almost surely. For example, we may take $ \(F(X)\)_t = \psi(X_t) $
for a given bounded Borel measurable $ \psi : \R \to \R $ provided that random
variables $X_t$ are well-defined; that is, the measure $ \nu : B \mapsto \int_B
X_t \, \D t $ is absolutely continuous almost surely, and $ X_t
= \frac{\D\,\nu}{\D\,\mes} (t) $. Generally, such function $ t \mapsto X_t $ is
defined up to equivalence (that is, arbitrary change on a \emph{random} null
set), but still, the measure $ B \mapsto \int_B \psi(X_t) \, \D t $ is
well-defined (and absolutely continuous). Also, boundedness of $\psi$ may be
relaxed; $ \psi(x) = \cO(|x|) $ for $ |x| \to \infty $ is enough for local
integrability of $ \psi(X_t) $.

\begin{proposition}\label{2.3}
Let $ X = (X_t)_{t\in\R} $ be a splittable random process such that, almost
surely, the measure $ B \mapsto \int_B X_t \, \D t $ is absolutely continuous,
and $ Y_t = \psi(X_t) - \Ex \psi(X_t) $ for a given function $ \psi : \R \to \R
$ that satisfies the Lipschitz condition $ \exists L \; \forall x,y \;\>
|\psi(x)-\psi(y)| \le L |x-y| $. Then $ Y = (Y_t)_{t\in\R} $ is a splittable
random process.
\end{proposition}

\begin{proof}
The proof of Prop.~\ref{2.1} needs only trivial modifications. WLOG, $L=1$ and
$ \psi(0)=0 $. Denote $ m_t = \Ex \psi(X_t) $, then $ Y_t = \psi(X_t) - m_t $.

(a) $ |\psi(X_t)| \le |X_t| $;
\begin{multline*}
\exp \frac12 \int_{[s,s+1)} |Y_t| \, \D t \le \exp \frac12 \Big( \int_{[s,s+1)}
 |\psi(X_t)| \, \D t + \int_{[s,s+1)} |m_t| \, \D t \Big) \le \\
\le \frac12 \exp \int_{[s,s+1)}
 |\psi(X_t)| \, \D t + \frac12 \exp \int_{[s,s+1)} |m_t| \, \D t \, ;
\end{multline*}
take the expectation and note that $ \exp \int_{[s,s+1)} |m_t| \, \D t
\le \exp \int_{[s,s+1)} \Ex |\psi(X_t)| \, \D t = \exp \Ex \int_{[s,s+1)}
|\psi(X_t)| \, \D t \le \Ex \exp \int_{[s,s+1)} |\psi(X_t)| \, \D t $.

(c) $ \ti Y^0_t = \psi(\ti X^0_t) - m_{t+r} $, and so on.
\end{proof}

On the other hand, sometimes it is useful to extend the domain of $F$ beyond
$ \VM $. In particular, the white noise\footnote{%
 See \cite[page 2]{I}.}
$ w = (w_t)_{t\in\R} $ cannot be
thought of as a random (locally finite signed) measure, but still, for every $
\phi \in L_2(\R) $ the random variable $ \int_{\R} \phi(t) w_t \, \D t $ is a
well-defined element of $ L_2(\Om) $, and the convolution $ G = \phi * w $ (that
is, $ G_t = \int_{\R} \phi(t-s) w_s \, \D s $) is a stationary Gaussian process
on $\R$, continuous in probability (and in the mean square). Under mild
conditions on $\phi$ the process $G$ has a sample continuous modification, that
is, may be thought of as a random continuous function. But generally it is not,
even if $\phi$ is a compactly supported continuous function.\footnote{%
 This is closely related to Pisier algebra, see \cite[Sect.~15.1, p.~213;
 Sect.~15.3, Theorem 3]{Kahane}, \cite[p.~682]{BPZ}.}
Generally, $G$ may be thought of as a random equivalence class\footnote{%
 Equivalence being equality almost everywhere on $\R$.}
of locally integrable functions on $\R$, thus, also a random locally finite
signed measure on $\R$.

The splittability definition does not apply to the white
noise $w$, since $w$ fails to be a random measure. But informally $w$ is
splittable, and moreover, \splittable{C} for all $C$, just because the past $
(w_t)_{t\in(-\infty,0)} $ and the future $ (w_t)_{t\in[0,\infty)} $ are
independent. We obtain a leakless split $(w^0,w^-,w^+)$ of $w$ as follows: $
w^-, w^+ $ are independent copies of $w$, and
\[
w_t^0 = w_t^{\sgn t} = \begin{cases} 
  w_t^- &\text{when } t<0,\\
  w_t^+ &\text{when } t\ge0;
\end{cases}
\]
more formally, for every $ \phi \in L_2(\R) $ the equality $ { \int_{\R} \phi(t) w^0_t \,
\D t } = \linebreak { \int_{\R} \phi(t) \One_{(-\infty,0)}(t) w^-_t \, \D t }
+ { \int_{\R} \phi(t) \One_{[0,\infty)}(t) w^+_t \, \D t } $ holds almost surely. Here is
a counterpart of Prop.~\ref{2.2}.

\begin{proposition}\label{2.4}
Let $\phi\in L_2(\R)$ satisfy $ \esssup_{t\in\R} (1+|t|)^\bal \phi^2(t) < \infty
$ for some $ \bal>3 $, and $ G = \phi * w $. Then $ G $ is a splittable random
process.
\end{proposition}

\begin{lemma}\label{2.6}
The inequality
\[
\Ex \exp \int_A |G_t| \, \D t \le \E^{C^2/2} \frac2{\sqrt{2\pi}}
\int_{-\infty}^C \E^{-u^2/2} \, \D u \le \exp \bigg( \sqrt{\frac2\pi} C
+ \frac12 C^2 \bigg) \le \infty \, ,
\]
where $ C = \int_A \sqrt{ \Ex G_t^2 } \, \D t $, holds for every Lebesgue
measurable set $ A \subset \R $ and every centered\footnote{%
 That is, of mean zero.}
Gaussian process $G$ on $A$
whose covariation function $ (s,t) \mapsto \Ex G_s G_t $ is Lebesgue measurable
on $ A \times A $.
\end{lemma}

\begin{remark}\label{2.7}
As usual, we assume that the Hilbert space $ L_2(\Om) $ is separable. This lemma
generalizes readily to $ A \subset \R^d $, and moreover, to arbitrary measure
space $A$. The ``$ \le\infty $'' clarifies that the inequality is void when $
C=\infty $.
\end{remark}

\begin{proof}[Proof of Lemma \ref{2.6}]
Denoting $ \si_t = \sqrt{ \Ex G_t^2 } $ we have $ \int_A \si_t \, \D t = C $ and
\begin{multline*}
\Ex \exp C \frac{|G_t|}{\si_t}
 = \frac1{\sqrt{2\pi}} \int_{-\infty}^\infty \E^{C|u|} \E^{-u^2/2} \, \D u = \\
\frac2{\sqrt{2\pi}} \int_0^\infty \exp \Big( -\frac12 (u-C)^2 + \frac12
 C^2 \Big) \, \D u
= \E^{C^2/2} \frac2{\sqrt{2\pi}} \int_{-\infty}^C \E^{-u^2/2} \, \D u
\end{multline*}
for all $ t \in A $; by convexity of $ \exp $,
\[
\exp \int_A |G_t| \, \D t = \exp \bigg(
C \int_A \frac{|G_t|}{\si_t} \frac{\si_t}C \, \D t \bigg) \le \int_A \Big( \exp
C \frac{|G_t|}{\si_t} \Big) \frac{\si_t}C \, \D t \, ;
\]
thus,
\[
\Ex \exp \int_A |G_t| \, \D t \le \int_A \Big( \Ex \exp
C \frac{|G_t|}{\si_t} \Big) \frac{\si_t}C \, \D t = \E^{C^2/2} \frac2{\sqrt{2\pi}}
\int_{-\infty}^C \E^{-u^2/2} \, \D u \, .
\]
And finally,
\[
\frac2{\sqrt{2\pi}} \int_{-\infty}^C \E^{-u^2/2} \, \D u
= \frac2{\sqrt{2\pi}} \bigg( \int_{-\infty}^0 \!\!\! \dots + \int_0^C \!\!\! \dots \bigg) \le
1 + \frac2{\sqrt{2\pi}} C \le \exp \bigg( \sqrt{\frac2\pi} C \bigg) \, .
\]
\end{proof}

\begin{proof}[Proof of Prop.~\ref{2.4}] \let\qed\relax
We inspect \ref{2.rem}(a,b,c).

(a) WLOG, $ \int_{\R} \phi^2(t) \, \D t = 1 $. We have $ \Ex G_t^2 = 1 $, thus
$ \int_{[0,1)} \sqrt{ \Ex G_t^2 } \, \D t = 1 $. Applying Lemma \ref{2.6} to
$ \frac1C G $ we get $ \Ex \exp \frac1C \int_{[0,1)} |G_t| \, \D
t \le \exp \( \frac1C \sqrt{\frac2\pi} + \frac1{2C^2} \) \le 2 $ for $C$ large
enough.

(b) Trivial.

(c) We use $ w^0, w^-, w^+ $ (recall the text before the proposition), introduce
$ G^0 = \phi * w^0 $, $ G^- = \phi * w^- $, $ G^+ = \phi * w^+ $, observe that $
(G^0, G^-, G^+) $ is a split of $ G $, consider its leak $ Y : t \mapsto
G_t^{\sgn t} - G_t^0 $, and prove that $ \Ex \exp \frac1C \int_\R |Y_t| \, \D
t \le 2 $ for $C$ large enough.

We have $ G^+ - G^0 = \phi * ( w^+ - w^0 ) $, that is, $ G^+_t - G^0_t
= \int_{\R} \phi(t-s) ( w^+_s - w^0_s ) \, \D s = \int_{(-\infty,0)} \phi(t-s) (
w^+_s - w^-_s ) \, \D s $. Further, $ \Ex ( G^+_t - G^0_t )^2 =
2\int_{(-\infty,0)} \phi^2(t-s) \, \D s = 2 \int_{(t,\infty)} \phi^2(s) \, \D s
$, whence $ \int_{[0,\infty)} \sqrt{ \Ex ( G^+_t - G^0_t )^2 } \, \D t = \sqrt2
M $ where $ M = \int_0^\infty \sqrt{ \int_t^\infty \phi^2(s) \, \D s } \, \D t
$. Applying  Lemma \ref{2.6} to $ \frac2C ( G^+ - G^0 ) $ we get
\[
\Ex \exp \Big( \frac2C \int_{[0,\infty)} | G^+_t - G^0_t | \, \D
t \Big) \le \exp \Big( \sqrt{\frac2\pi} \cdot \frac2C \cdot \sqrt2 M
+ \frac12 \cdot \frac4{C^2} \cdot 2 M^2 \Big) \le 2
\]
for $C$ large enough, provided that $ M < \infty $. It remains to note that,
denoting $ \esssup_{t\in\R} (1+|t|)^\bal \phi^2(t) $ by $ C_1 $, we have
\begin{gather*}
\int_t^\infty \phi^2(s) \, \D s \le \int_t^\infty C_1 (1+|s|)^{-\bal} \, \D s =
 \frac{C_1}{(\bal-1)(1+t)^{\bal-1}} \, ; \\
\int_0^\infty \sqrt{ \frac{C_1}{(\bal-1)(1+t)^{\bal-1}} } \, \D t
= \sqrt{ \frac{C_1}{\bal-1} } \frac2{\bal-3} < \infty \, . \qquad\quad \rlap{$\qedsymbol$}
\end{gather*}
\end{proof}

\section[Dimension one: more]
  {\raggedright Dimension one: more}
\label{sect3}
\begin{sloppypar}
\begin{proposition}\label{3.1}
Assume that $ \phi \in L_2(\R) $ satisfies $ \int_\R \phi^2(t) \, \D t = 1 $, $
\esssup_{t\in\R} (1+|t|)^\bal \phi^2(t) < \infty $ for some $ \bal>3 $, and $
\psi : \R \to \R $ satisfies the Lipschitz condition, and $ \int_\R \psi(u)
\E^{-u^2/2} \, \D u = 0 $. Then the following random process is splittable: $ X
= \psi(\phi*w) $, that is, $ X_t = \psi\((\phi*w)_t\) $.
\end{proposition}
\end{sloppypar}

\begin{proof}
Just combine Propositions \ref{2.3} and \ref{2.4}.
\end{proof}

In this section we generalize Prop.~\ref{3.1} to complex-valued $\phi$ and not
necessarily Lipschitz functions $ \psi : \C \to \R $, including the function $
\psi : z \mapsto \log |z| $.

The complex-valued white noise is just $ \wC = \frac1{\sqrt2} (w_1+\I w_2) $,
where $ w_1,w_2 $ are two independent copies of the real-valued white noise (and
$ \I \in \C $ is the imaginary unit). Similarly to the real-valued case, given a
complex-valued $ \phi \in L_2(\R\to\C) $, the convolution $ \phi * \wC $ is a
complex-valued stationary Gaussian random process on $\R$;
\begin{multline*}
( \phi * \wC )_t = \int \phi(t-s) \wC(s) \, \D s = \int \( \Re \phi(t-s) w_1(s)
 - \Im \phi(t-s) w_2(s) \) \, \D s + \\
\I \int \( \Re \phi(t-s) w_2(s) + \Im \phi(t-s) w_1(s) \) \, \D s \, ;
\end{multline*}
it is a well-defined random equivalence class of locally integrable functions $
\R \to \C $.\,\footnote{%
 Recall page 10.}
Its real part $ \Re (\phi * \wC) $ and imaginary part $ \Im (\phi * \wC) $ are
identically distributed Gaussian processes.\footnote{
 Since $ (w_1,-w_2) $ and $ (w_2,w_1) $ are identically distributed.}
Their cross-covariation function $ (t_1,t_2) \mapsto \Ex (\Re X_{t_1}) (\Im
X_{t_2}) = \int ( \Re\phi(t_1-s) \Im\phi(t_2-s) - \Im\phi(t_1-s) \Re\phi(t_2-s)
\) \, \D s $ vanishes on the diagonal ($t_1=t_2$), but need not vanish outside
the diagonal.

The notions of split, leak, splittability, Definition \ref{2.spli} and Remark
\ref{2.rem} generalize readily to complex-valued random processes.

The function $ \psi : z \mapsto \log |z| $ needs a new approach, different from
the approach of Sect.~\ref{sect2} (used in the proof of Prop.~\ref{3.1}). Here
is why.

In the approach of Sect.~\ref{sect2}, when estimating (an exponential moment of)
a non-Gaussian leak integral $ \int_0^\infty | \psi(\phi*w^+) - \psi(\phi*w^0) |
\, \D t $, decay of $ \phi $ (such as $ \esssup_{t\in\R} (1+|t|)^\bal
\phi^2(t) < \infty $ for some $ \bal>3 $) is used for estimating the underlying
Gaussian leak integral $ \int_0^\infty | \phi*w^+ - \phi*w^0) | \, \D t $;
afterwards, Lipschitz continuity of $\psi$ matters, but decay of $ \phi $
does not.

Another approach is needed for $ \psi : z \mapsto \log |z| $, as shown by the
following counterexample. We choose a pair of continuous functions $ f,g :
[0,\infty) \to \C $ such that $ f(t) \to 1 $ and $ g(t) \to 0 $ as $ t
\to \infty $, take three independent complex-valued random variables $ \xi_0,
\xi_1, \xi_2 $ each having the standard normal distribution on $\C $ (that is,
like $ \frac1{\sqrt2} (\zeta_1+\I\zeta_2) $ where $ \zeta_1, \zeta_2 $ are
independent real-valued normal $N(0,1)$ random variables), and introduce a pair
of (correlated) identically distributed complex-valued Gaussian processes $ X_1,
X_2 $ on $[0,\infty)$ by $ X_1 = \xi_0 f + \xi_1 g $, $ X_2 = \xi_0 f + \xi_2 g
$. A calculation, sketched below, shows that
\[
\Ex \exp \frac1C \int_0^\infty \big| \log |X_1(t)| - \log |X_2(t)| \big| \, \D t
= \infty \qquad \text{for all } C \in (0,\infty)
\]
even if the function $ t \mapsto \Ex | X_1(t) - X_2(t) |^2 = 2|g(t)|^2 $ decays very
fast, namely, $ |g(t)| \sim \exp(-t^2) $ for large $t$, and moreover, when $ |g(t)|
\sim \exp(-t^3) $, or $ |g(t)| \sim \exp(-t^4) $, and so on.

Here is the calculation. For a small $ \eps $, the event
``$|\xi_0|\le\eps$ and $|\xi_1|\le\eps$ and $|\xi_2|\ge1$'' is of probability $
\sim \eps^4 $ (up to a coefficient); denote this event by $A$. Given $A$, we
have
\[
\frac{|X_1(t)|}{|X_2(t)|} = \frac{ |\xi_0 f(t) + \xi_1 g(t)| }{ |\xi_0 f(t) +
\xi_2 g(t)| } \le \frac{ \eps |f(t)| + \eps |g(t)| }{ |g(t)|-\eps|f(t)| } \, .
\]
Assuming in addition that $ |g(t)| \ge 2\eps |f(t)| $ we have
\[
\frac{|X_1(t)|}{|X_2(t)|} \le 2\eps \frac{ |f(t)| + |g(t)| }{ |g(t)| } \, ,
\]
whence
\[
\big| \log |X_1(t)| - \log |X_2(t)| \big| \ge \log \frac{|X_2(t)|}{|X_1(t)|} \ge
\log \frac{ |g(t)| }{ 2\eps (|f(t)| + |g(t)|) } \, .
\]
Now, let $ g(t) = \exp(-t^n) $. We take $ [a,b] \subset (0,\infty) $ such that $
t \in [a,b] $ if and only if $ \eps^{3/4} \le |g(t)| \le \eps^{1/4} $; that is,
$ a = \sqrt[n]{\frac14 \log \frac1\eps} $ and $ b = \sqrt[n]{\frac34 \log
\frac1\eps} $. If $\eps$ is small enough, we have $ |f(t)| \le 2 $ and $ |g(t)|
\ge \eps^{3/4} \ge 4\eps \ge 2\eps |f(t)| $ for all $ t \in [a,b] $, thus,
\begin{gather*}
\big| \log |X_1(t)| - \log |X_2(t)| \big| \ge \log \frac{ \eps^{3/4} }{ 2\eps
 (2+\eps^{1/4}) } \ge \frac14 \log \frac1\eps - \const \, ; \\
\! \int_a^b \! \big| \log |X_1(t)| - \log |X_2(t)| \big| \, \D t 
 \ge (b-a) \Big( \frac14 \log \frac1\eps - \const \!\! \Big) \ge \const \cdot
 \Big( \! \log \frac1\eps \Big)^{1+\frac1n} \, ; \\
\Ex \exp \frac1C \int_0^\infty \big| \dots - \dots \big| \, \D t
\ge \underbrace{\Pr{A}}_{\ge\const\cdot\eps^4} \!\! \cdot \exp \bigg( \frac1C \cdot
\const \Big( \log \frac1\eps \Big)^{1+\frac1n} \bigg) \to \infty
\end{gather*}
as $ \eps \to 0+ $.

The two cases, real-valued and complex-valued, may be treated as special cases
of \valued{\R^p} Gaussian processes. Thus, $ \psi : \R^p \to \R $ (for a given $
p \in \{1,2,3,\dots\} $), and first of all we need $ \Ex \exp \frac1C \psi(X) <
\infty $ for a given Gaussian random vector $X$, when $C$ is large enough.

Further, we need an appropriate condition on $\psi$, weaker than Lipschitz
continuity.

Here is a general fact, to be used in the proof of Lemma \ref{3.15} (and
\ref{3.16}).

\begin{lemma}\label{3.2}
If functions $ f,g : (0,\infty) \to (0,\infty) $ are decreasing, then
\[
\int_{\R^p} f(|x+a|) g(|x|) \, \D x \le \int_{\R^p} f(|x|) g(|x|) \, \D x \le
\infty
\]
for all $ a \in \R^p $.
\end{lemma}

\begin{proof}\let\qed\relax
WLOG, $ f(r) \to 0 $ and $ g(r) \to 0 $ as $ r \to \infty $. We have
\begin{multline*}
\int_{\R^p} f(|x+a|) g(|x|) \, \D x = \int_{\R^p} \Big( \int_{[|x+a|,\infty)}
 \(-\D f(u)\) \Big) \Big( \int_{[|x|,\infty)} \(-\D g(u)\) \Big) \, \D x = \\
\iint \Big( \int_{|x+a|\le u, |x|\le v} \D x \Big) \(-\D f(u)\) \(-\D g(v)\) \le
 \\
\iint \Big( \int_{|x|\le u, |x|\le v} \D x \Big) \(-\D f(u)\) \(-\D g(v)\) =
 \int_{\R^p} f(|x|) g(|x|) \, \D x \, . \qquad \rlap{$\qedsymbol$}
\end{multline*}
\end{proof}

\begin{definition}\label{noisy}
Let $ \mu $ be a probability measure on $ \R^p $, and $
(X_1,\dots,X_n) $ a sequence of \valued{\R^p} random
variables. We say that $ (X_1,\dots,X_n) $ is \emph{\noisy{\mu},} if there exist
(on some probability space) two independent sequences $ (Y_1,\dots,Y_n) $, $
(Z_1,\dots,Z_n) $ of \valued{\R^p} random variables such that
$ (Y_1+Z_1,\dots,Y_n+Z_n) $ is distributed like $ (X_1,\dots,X_n) $, and $
Z_1,\dots,Z_n $ are independent, each distributed $\mu$. In other words: the
distribution of $ (X_1,\dots,X_n) $ is the convolution of some probability
measure and the product measure $ \mu \times \dots \times \mu $.
\end{definition}

\begin{lemma}\label{3.12}
If $ (X_1,\dots,X_n) $ is \noisy{\mu}, then
\[
\Ex \prod_{k=1}^n f_k(X_k) \le \prod_{k=1}^n \sup_{y\in\R^p} \int_{\R^p}
f_k(y+z) \, \mu(\D z) \le \infty
\]
for all Borel measurable $ f_1,\dots,f_n : \R^p \to [0,\infty) $.\,\footnote{%
 See also \cite[Lemma 3.4]{2008}.}\fnsep\footnote{%
 If $\mu$ is absolutely continuous, then Lebesgue measurability is enough. And,
 as usual, whenever we deal with a Lebesgue measurable function, only its
 equivalence class matters.}
\end{lemma}

\begin{proof}\let\qed\relax
Using $Y_k,Z_k$ from Def.~\ref{noisy} we get
\[
\Ex \prod_{k=1}^n f_k(X_k) = \Ex \prod_{k=1}^n f_k(Y_k+Z_k) = \Ex \(
\cE{\textstyle \prod_{k=1}^n f_k(Y_k+Z_k) }{\textstyle Y_1,\dots,Y_n } \); \quad \text{and}
\]
\begin{multline*}
\cE{\textstyle \prod_{k=1}^n f_k(Y_k+Z_k) }{\textstyle Y_1,\dots,Y_n } = \prod_{k=1}^n \cE{
 f_k(Y_k+Z_k) }{ Y_1,\dots,Y_n } = \\
\prod_{k=1}^n \int_{\R^p} f_k(Y_k+z) \, \mu(\D z) \le \prod_{k=1}^n
 \sup_{y\in\R^p} \int_{\R^p} f_k(y+z) \, \mu(\D z) \, . \qquad \rlap{$\qedsymbol$}
\end{multline*}
\end{proof}

Here is a chain of rather trivial implications of Lemma \ref{3.12}. Introducing
shifted measures $ \mu_y $ by $ \int f(z) \, \mu_y(\D z) = \int f(y+z) \mu(\D z)
$ we get
\[
\Ex \prod_k f_k(X_k) \le \prod_k \sup_y \int f_k(z) \, \mu_y(\D z) \, .
\]
Replacing $ f_k(\cdot) $ with $ \exp f_k(\cdot) $ we get
\[
\Ex \exp \sum_k f_k(X_k) \le \exp \sum_k \sup_y \log \int \exp f_k(z) \,
\mu_y(\D z) \, .
\]
Replacing $ f_k(\cdot) $ with $ f(\cdot,x'_k) $ (with arbitrary parameters $
x'_1, \dots, x'_n \in \R^p $) we get
\[
\Ex \exp \sum_k f(X_k,x'_k) \le \exp \sum_k \sup_y \log \int \exp f(z,x'_k) \,
\mu_y(\D z) \, .
\]
Replacing this $ \sup \log \int \dots $ by $ g(x'_k) $ we get
\begin{gather*}
\text{if} \quad \forall x',y \;\> \log \int \exp f(z,x') \, \mu_y(\D z) \le
 g(x') \, , \quad \text{then} \\
\Ex \exp \sum_k f(X_k,x'_k) \le \exp \sum_k g(x'_k) \, .
\end{gather*}
Assuming Borel measurability of the function $ g $ and replacing the sequence of
parameters $ x' = (x'_1,\dots,x'_n) $ with a sequence $ X' = (X'_1,\dots,X'_n) $
of \valued{\R^p} random variables $ X'_k $ such that $ X' $ is independent of
the sequence $ X = (X_1,\dots,X_n) $ (while $X'_1,\dots,X'_n$ may be
interdependent) we get
\begin{gather*}
\text{if} \quad \forall x',y \;\> \log \int \exp f(z,x') \, \mu_y(\D z) \le
 g(x') \, , \quad \text{then} \\
\CE{ \exp \sum_k f(X_k,X'_k) }{ X' } \le \exp \sum_k g(X'_k) \, .
\end{gather*}
Taking unconditional expectation we get
\begin{gather*}
\text{if} \quad \forall x',y \;\> \log \int \exp f(z,x') \, \mu_y(\D z) \le
 g(x') \, , \quad \text{then} \\
\Ex \exp \sum_k f(X_k,X'_k) \le \Ex \exp \sum_k g(X'_k) \, .
\end{gather*}
Taking $ f $ of the form $ f(x,x') = | \psi(x+x') - \psi(x-x') | $ we get
\begin{gather*}
\text{if} \quad \forall x',y \;\> \log \int \exp | \psi(z+x') - \psi(z-x') | \,
\mu_y(\D z) \le g(x') \, , \quad \text{then} \\
\Ex \exp \sum_k | \psi(X_k+X'_k) - \psi(X_k-X'_k) | \le \Ex \exp \sum_k g(X'_k)
\, .
\end{gather*}
We summarize.

\begin{lemma}\label{3.5}
Let Borel measurable functions $ \psi : \R^p \to \R $, $ g : \R^p \to \R $ and
a probability measure $ \mu $ on $ \R^p $ satisfy\footnote{%
 As before, $ \mu_h $ is $ \mu $ shifted by $ h $.}
\begin{equation}\label{3.6}
\log \int_{\R^p} \exp | \psi(x+y) - \psi(x-y) | \, \mu_h(\D x) \le g(y)
\end{equation}
for all $ y,h \in \R^p $. Then
\[
\Ex \exp \sum_{k=1}^n | \psi(X_k+Y_k) - \psi(X_k-Y_k) | \le \Ex \exp
\sum_{k=1}^n g(Y_k)
\]
for every pair $X,Y$ of independent sequences $ X = (X_1,\dots,X_n) $ and $ Y =
(Y_1,\dots,Y_n) $ of \valued{\R^p} random variables such that $ X $ is
\noisy{\mu}.\footnote{%
 Clarification: the whole $X$ is independent of the whole $Y$, while among
 $ X_1,\dots,X_n $ dependency is allowed, as well as among $ Y_1,\dots,Y_n $.}%
\fnsep\footnote{%
 If $\mu$ is absolutely continuous, then Lebesgue measurability is enough for
 $\psi$. If also each $Y_k$ has an absolutely continuous distribution, then
 the same holds for $g$, too.}
\end{lemma}

Condition \eqref{3.6} may be thought of as a generalization of H\"older
continuity. Indeed, if $\mu$ is a single atom at $0$, then $ \log \int \exp |
\psi(x+y) - \psi(x-y) | \, \mu_h(\D x) = | \psi(h+y) - \psi(h-y) | $; if also $
g(y) = |2y|^\de $, then \eqref{3.6} boils down to $ \forall a,b \;\>
|\psi(a)-\psi(b)| \le |a-b|^\de $.

\begin{lemma}\label{3.15}
Let the function $ \psi : \R^p \to \R $ be given by $ \psi(x) = \log |x|
$, and $ \mu $ be a nontrivial rotation invariant Gaussian measure on $ \R^p $,
that is, $ \mu(\D x) = (2\pi\si^2)^{-p/2} \exp \( - |x|^2 / (2\si^2) \) \D x $ for
some $ \si \in (0,\infty) $. Then condition \eqref{3.6} is satisfied by (this
$\psi$ and) $ g : \R^p \to \R $ given by $ g(y) = C|y| $, provided that the
constant $C$ (dependent only on $p$ and $\si$) is large enough, and $ p \ge 2 $.
\end{lemma}

\begin{sloppypar}
\begin{proof}
WLOG, $ \si=1 $, since $ \log |x+y| - \log |x-y| $ is insensitive to the
homothety $ x \mapsto \si x $.
We have $ \int_{\R^p} \exp \big| \log |x+y| - \log |x-y| \big| \, \mu_h(\D x) =
\int_{\R^p} \exp \max \( \log \frac{|x+y|}{|x-y|}, \log \frac{|x-y|}{|x+y|} \)
\, \mu_h(\D x) = \int_{\R^p} \max \( \frac{|x+y|}{|x-y|}, \frac{|x-y|}{|x+y|} \)
\, \mu_h(\D x) \le \int_{\R^p} \max \( 1+\frac{2|y|}{|x-y|},
1+\frac{2|y|}{|x+y|} \) \, \mu_h(\D x) \le 1 + 2|y| \int_{\R^p} \( \frac1{|x-y|}
+ \frac1{|x+y|} \) \, \mu_h(\D x) = 1 + 2|y| \int_{\R^p} \( \frac1{|x+h-y|}
+ \frac1{|x+h+y|} \) \, \mu(\D x) $. By Lemma \ref{3.2}, $ \int_{\R^p}
\frac{\mu(\D x)}{|x+h-y|} \le \linebreak[1] \int_{\R^p} \frac{\mu(\D x)}{|x|} = \const(p)
\int_0^\infty \E^{-r^2/2} r^{p-2} \, \D r < \infty $, and the same holds for $
\int_{\R^p} \frac{\mu(\D x)}{|x+h+y|} $. Finally, $ \log(1+C|y|) \le C|y| $.
\end{proof}
\end{sloppypar}

The case $ p=1 $ is different. Here \eqref{3.6} with $ g(y) = C|y| $ (for
small $y$) implies that $\psi$ is locally of bounded variation, therefore,
locally bounded (unlike $ \log|\cdot|$), since $ \int_{\R} \exp | \psi(x+y) -
\psi(x-y) | \, \mu_h(\D x) = 1 + 2|y| \int_{\R} |\psi'(x)| \, \mu_h(\D x) + o(y)
$ as $ y \to 0 $, whenever $ \psi $ is continuously differentiable and compactly
supported (otherwise approximate).

\begin{lemma}\label{3.16}
Let the function $ \psi : \R \to \R $ be given by $ \psi(x) = \frac12 \log |x|
$, and $ \mu(\D x) = (2\pi)^{-1/2} \exp (-x^2/2) \, \D x $. Then
condition \eqref{3.6} is satisfied by (this $\psi$ and) $ g : \R \to \R $ given
by $ g(y) = C\sqrt{|y|} $, provided that the constant $C$ is large enough.
\end{lemma}

{\raggedright
\begin{proof}
We modify the proof of Lemma \ref{3.15} as follows:
$ \int_\R \exp \frac12 \big| {\log |x+y|} - \log |x-y| \big| \, \mu_h(\D x)
\le \int_\R \max \( \sqrt{ 1+\frac{2|y|}{|x-y|} },
\sqrt{ 1+\frac{2|y|}{|x+y|} } \) \, \mu_h(\D x) \le
\int_\R \sqrt{ 1+\frac{2|y|}{|x-y|}+\frac{2|y|}{|x+y|} } \, \mu_h(\D x) \le
\int_\R \( 1 + \sqrt{\frac{2|y|}{|x-y|}} + \sqrt{\frac{2|y|}{|x+y|}} \) \,
\mu_h(\D x) \le 1 + \sqrt{2|y|} \cdot 2 \int_\R \frac{ \mu(\D x) }{ \sqrt{|x|} }
$. 
\end{proof}
}

Now we abandon the generality of values in $\R^p$ and return to real values
($p=1$) and complex values ($p=2$). We introduce two standard Gaussian measures,
$ \gR $ on $ \R $ and $ \gC $ on $ \C $, by
\[
\gR(\D x) = (2\pi)^{-1/2} \E^{-x^2/2} \, \D x \, , \quad
\gC(\D z) = \pi^{-1} \E^{-|z|^2} \, \D z
\]
(about $\D z$ see footnote on page 2). We use the (real-valued)
white noise $(w_t)_{t\in\R}$, and the complex-valued
white noise $ \wC = \frac1{\sqrt2} ( w_1 + \I w_2 ) $ (recall the text after
Prop.~\ref{3.1}).
For every $ \phi \in L_2(\R) $ such that $ \int_\R \phi^2(t) \, \D t = 1 $ the
random variable $ \int_\R \phi(t) w_t \, \D t $ is distributed $ \gR $. For
every $ \phi \in L_2(\R\mapsto\C) $ such that $ \int_\R |\phi(t)|^2 \, \D t = 1
$ the random variable $ \int_\R \phi(t) (\wC)_t \, \D t $ is distributed $ \gC
$. We consider real-valued processes of the form $ \phi * w $ where $ \phi \in
L_2 $ (recall the text before Prop.~\ref{2.4}); and similarly, complex-valued
processes $ \phi * \wC $ where $ \phi \in L_2(\R\to\C) $. Given $ G = \phi * w
$, we consider its split $ (G^0,G^-,G^+) $ where
\begin{equation}\label{3.*}
G^0 = \phi * w^0 \, , \quad G^- = \phi * w^- \, , \quad G^+ = \phi * w^+ \, ;
\end{equation}
here $ w^-, w^+ $ are two
independent copies of $w$; $ w_t^0 = w_t^- $ for $ t<0 $, and $ w_t^0 = w_t^+ $
for $ t\ge0 $ (recall the proof of Prop.~\ref{2.4}). The same applies in the
complex-valued case (with $ \wC $ instead of $w$, and accordingly, $ \wC^0,
\wC^-, \wC^+ $). We know that the processes $ G^+, G^0 $ (and $ G^- $ as well)
are identically distributed. Here is a stronger claim.

\begin{lemma}\label{3.9}
The processes $ G^+ + G^0 $ and $ G^+ - G^0 $ are independent (in both cases,
real-valued and complex-valued).
\end{lemma}

\begin{proof}
We do it in the real-valued case; the complex-valued case is similar. We have $
G_t^+ - G_t^0 = \int \phi(t-s) w_s^+ \, \D s - \int \phi(t-s) w_s^0 \, \D s =
\int_{s<0} \phi(t-s) w_s^+ \, \D s + \int_{s\ge0} \phi(t-s) w_s^+ \, \D s -
\int_{s<0} \phi(t-s) w_s^- \, \D s - \int_{s\ge0} \phi(t-s) w_s^+ \, \D s =
\int_{s<0} \phi(t-s) (w_s^+ - w_s^-) \, \D s $. Similarly, $ G_t^+ + G_t^0 =
\int_{s<0} \phi(t-s) (w_s^+ + w_s^-) \, \D s + \int_{s\ge0} \phi(t-s) \cdot
2w_s^+ \, \D s $. It remains to note that $ (w_s^+)_{s\ge0} $, $ (w_s^+ +
w_s^-)_{s<0} $, and $ (w_s^+ - w_s^-)_{s<0} $ are independent.
\end{proof}

In the spirit of Def.~\ref{noisy}, given two random processes $ X,Y $, we say
that $X$ is \emph{noisier} than $Y$, if there exist (on some probability space)
two independent processes $ U,V $ such that $ U+V $ is distributed like $X$, and
$V$ is distributed like $Y$.\,\footnote{%
 This applies in both cases, real-valued and complex-valued.}

\begin{lemma}\label{3.10}
The process $ G^+ + G^0 $ is noisier than $ \sqrt2 G^0 $ (in both cases,
real-valued and complex-valued).
\end{lemma}

\begin{sloppypar}
\begin{proof}
\quad
We do it in the real-valued case; the complex-valued case is similar.
In addition to $ w^-, w^+ $ we introduce an independent copy $ \ti w^+ $ of $
w^+ $ (so that $ w^-, w^+, \ti w^+ $ are independent), define $ U_t = \sqrt2
\int_{s\ge0} \phi(t-s) \ti w_s^+ \, \D s $, $ V_t = \sqrt2 G_t^0 $, and observe
that $ G_t^+ + G_t^0 = { \int_{s<0} \phi(t-s) ( w_s^+ + w_s^- ) \, \D s } + { 2
\int_{s\ge0} \phi(t-s) w_s^+ \, \D s } $ is distributed like $ { \int_{s<0}
\phi(t-s) \cdot \sqrt2 w_s^- \, \D s } + { \sqrt2 \int_{s\ge0} \phi(t-s) ( w_s^+ +
\ti w_s^+ ) \, \D s } = { \sqrt2 \int_\R \phi(t-s) w_s^0 \, \D s } + \sqrt2
\int_{s\ge0} \phi(t-s) \ti w_s^+ \, \D s = \sqrt2 G_t^0 + U_t $.
\end{proof}
\end{sloppypar}

\begin{lemma}\label{3.11}
(a)
Let $ G = \phi * w $ where $ \phi \in L_2(\R) $ satisfies $ \int_\R \phi^2(t) \,
\D t = 1 $ and $ \esssup_{t\in\R} (1+|t|)^\bal \phi^2(t) < \infty $ for some $
\bal > 2 $. Then for every $ \eps > 0 $ there exists $ C $ (dependent only on $
\eps $, $ \bal $ and the $ \esssup $) such that for every $ M \ge C $ and every
$n$ the sequence $ ((1+\eps)G_M, (1+\eps)G_{2M}, \dots, (1+\eps)G_{nM}) $ is
\noisy{\gR}.

(b)
Let $ G = \phi * \wC $ where $ \phi \in L_2(\R\to\C) $ satisfies $ \int_\R
|\phi(t)|^2 \, \D t = 1 $ and $
\esssup_{t\in\R} (1+|t|)^\bal |\phi(t)|^2 < \infty $ for some $ \bal > 2 $.
Then for every $ \eps > 0 $ there exists $ C $ (dependent only on $
\eps $, $ \bal $ and the $ \esssup $) such that for every $ M \ge C $ and every
$n$ the sequence $ ((1+\eps)G_M, (1+\eps)G_{2M}, \dots, (1+\eps)G_{nM}) $ is
\noisy{\gC}.
\end{lemma}

\begin{sloppypar}
\begin{proof}
\textsc{The real-valued case} (a). WLOG, $ \eps \le 1/2 $ and $ |\phi(t)| \le
A (1+|t|)^{-\bal/2} $ for all $t$ (and a given $A$).
For arbitrary $h>0$ we have $ | \Ex G_{-h} G_h | = | \int_\R \phi(-h-s)
\phi(h-s) \, \D s | \le { 2 A^2 \int_0^\infty (1+|s+h|)^{-\bal/2} (1+|s-h|)^{-\bal/2} \,
\D s } \le 2 A^2 (1+h)^{-\bal/2} \int_\R (1+|s|)^{-\bal/2} \, \D s $. Thus (using
stationarity),
\[
\sum_{k=1}^\infty | \Ex G_0 G_{kM} | \le 2 A^2 \bigg( \int_\R \frac{\D
s}{(1+|s|)^{\bal/2}} \bigg) \sum_{k=1}^\infty \Big( 1 + \frac{kM}2
\Big)^{-\bal/2} \to 0 \quad \text{as } M \to \infty \, .
\]
We take $C$ such that $ \sum_{k=1}^\infty | \Ex G_0 G_{kM} | \le \eps/2 $ for
all $ M \ge C $. For such $M$, for arbitrary $ n $ and $ a_1,\dots,a_n \in \R $,
using the inequality $ { |a_k a_\ell| \le \frac12(a_k^2+a_\ell^2) } $, we have
$ \Ex ( a_1 G_M + \dots + a_n G_{nM} )^2 \ge { \sum_{k=1}^n a_k^2 \( \Ex G_{kM}^2
- \sum_{\ell\ne k} |\Ex G_{kM} G_{\ell M}| \) } \ge \sum_{k=1}^n a_k^2 (1-\eps)
\ge \frac1{(1+\eps)^2} (a_1^2+\dots+a_n^2) $ (since $\eps\le1/2$). It follows
that the Gaussian sequence $ ((1+\eps)G_M, (1+\eps)G_{2M}, \dots,
(1+\eps)G_{nM}) $ is \noisy{\gR}, since we may take a Gaussian sequence
$(Y_1,\dots,Y_n)$, independent of $ (G_M, G_{2M}, \dots, G_{nM}) $, such that $
\Ex (a_1 Y_1 + \dots + a_n Y_n)^2 = (1+\eps)^2 \Ex ( a_1 G_M + \dots + a_n
G_{nM} )^2 - (a_1^2+\dots+a_n^2) $ for all $ a_1,\dots,a_n \in \R $.

\textsc{The complex-valued case} (b) needs only trivial modifications: $ \Ex
G_{-h} \overline{G_h} $ instead of $ \Ex G_{-h} G_h $; similarly, $ \int_\R
\phi(-h-s) \overline{\phi(h-s)} \, \D s $ and $ \Ex G_0 \overline{G_{kM}} $;
also, $ a_1,\dots,a_n \in \C $, $ a_k \overline{a_\ell} $, $ \Ex | a_1 G_M +
\dots + a_n G_{nM} |^2 $, $ \Ex |G_{kM}|^2 $, $ \Ex G_{kM} \overline{G_{\ell M}}
$, and complex-valued Gaussian sequences;\footnote{%
 A centered complex-valued Gaussian sequence is the image, under a
 complex-linear transformation, of a sequence of independent random variables
 distributed $ \gC $ each.}
$ \Ex |a_1 Y_1 + \dots + a_n Y_n|^2 = (1+\eps)^2 \Ex | a_1 G_M + \dots + a_n
G_{nM} |^2 - (|a_1|^2+\dots+|a_n|^2) $ for all $ a_1,\dots,a_n \in \C $.
\end{proof}
\end{sloppypar}

\begin{sloppypar}
Combining Lemma \ref{3.11} and Lemma \ref{3.10} we see that the sequence $
\((1+\eps)\frac{G^+_{kM} + G^0_{kM}}{\sqrt2}\)_{k=1}^n $ is \noisy{\ga} (for $ M $ large
enough), where $\ga$ is $ \gR $ in the real-valued case, and $ \gC $ in the
complex-valued case. Thus, Lemma \ref{3.5} applies to the pair of sequences
(independent by Lemma \ref{3.9})
\end{sloppypar}
\[
\Big( \rho\frac{G^+_{kM} + G^0_{kM}}{2} \Big)_{k=1}^n \, , \quad \Big(
\rho\frac{G^+_{kM} - G^0_{kM}}{2} \Big)_{k=1}^n
\]
whenever $ \rho > \sqrt2 $, and gives
\[
\Ex \exp \sum_{k=1}^n | \psi(\rho G^+_{kM}) - \psi(\rho G^0_{kM}) | \le \Ex \exp
\sum_{k=1}^n g\Big( \rho \frac{G^+_{kM} - G^0_{kM}}{2} \Big)
\]
provided that $ \psi, g $ satisfy \eqref{3.6} with the Gaussian measure $ \mu =
\ga $.
We note that $M$ does not depend on $n$, thus,
\[
\Ex \exp \sum_{k=1}^\infty | \psi(\rho G^+_{kM}) - \psi(\rho G^0_{kM}) | \le \Ex \exp
\sum_{k=1}^\infty g\Big( \rho \frac{G^+_{kM} - G^0_{kM}}{2} \Big) \, .
\]
Similarly,
\begin{multline*}
\Ex \exp \sum_{k=0}^\infty | \psi(\rho G^+_{(\theta+k)M}) -
 \psi(\rho G^0_{(\theta+k)M}) | \le \\
\Ex \exp \sum_{k=0}^\infty g\Big( \frac\rho2 \( G^+_{(\theta+k)M} -
 G^0_{(\theta+k)M} \) \Big)
\end{multline*}
for every $ \theta \in [0,1] $ (the shift by $ (\theta-1) M $, trivial for the
stationary process $G$, transfers to the nonstationary process $ G^+ + G^0 $ by
Lemma \ref{3.10} as before). We also note that generally
\[
\int_0^1 \sum_{k=0}^\infty f \( (\theta+k)M \) \, \D\theta = \sum_{k=0}^\infty
\int_0^1 \dots = \sum_{k=0}^\infty \frac1M
\int_{kM}^{(k+1)M} f(t) \, \D t = \frac1M \int_0^\infty f(t) \, \D t
\]
for measurable $ f : [0,\infty) \to [0,\infty) $; and by convexity of $ \exp $,
\[
\exp \frac1M \int_0^\infty \!\!\!\! f(t) \, \D t = \exp \int_0^1 \sum_{k=0}^\infty f \(
(\theta+k)M \) \, \D\theta \le \int_0^1 \!\! \exp \sum_{k=0}^\infty f \(
(\theta+k)M \) \, \D\theta \, ;
\]
thus
\begin{multline}\label{3.13}
\Ex \exp \frac1M \int_0^\infty | \psi(\rho G^+_t) - \psi(\rho G^0_t) | \,
 \D t \le \\
\int_0^1 \Ex \exp \sum_{k=0}^\infty | \psi(\rho G^+_{(\theta+k)M}) -
 \psi(\rho G^0_{(\theta+k)M}) | \, \D\theta \le \\
\int_0^1 \Ex \exp \sum_{k=0}^\infty g\Big( \frac \rho 2 \( G^+_{(\theta+k)M} -
 G^0_{(\theta+k)M} \) \Big) \, \D\theta
\end{multline}
for all $ \rho > \sqrt2 $ and $M$ large enough (according to Lemma
\ref{3.11} for $ \eps = \sqrt2 - 1 $); this applies in both cases, real-valued
and complex-valued. We summarize.
\enlargethispage{1pt}

\begin{sloppypar}
\begin{lemma}\label{3.14}
(a)
Let $ \phi \in L_2(\R) $ satisfy $ \int_\R \phi^2(t) \, \D t = 1 $ and $
\esssup_{t\in\R} (1+|t|)^\bal \phi^2(t) < \infty $ for some $ \bal > 2 $; and $
\psi : \R \to \R $ be Lebesgue measurable, satisfying $ \int_\R \psi(u) \gR(\D u) =
0 $. Then there exists $ C_0 $ (dependent only on $\bal$ and the $\esssup$) such
that for every $ C \ge C_0 $ the following two conditions are sufficient for
\splittability{C} of the process $ \psi(\phi*w) $:
\begin{gather*}
\tag{1}  \int_{-\infty}^\infty \exp \frac1C |\psi(u)| \, \gR(\D u) \le 2 ; \\
\tag{2}  \int_0^1 \Ex \exp \sum_{k=0}^\infty g ( G^\al_{\al(\theta+k)C}
 - G^0_{\al(\theta+k)C} ) \, \D\theta \le 2 \quad \text{for } \al=\pm1 \, ;
\end{gather*}
here $ G^0,G^-, G^+ $ are as in \eqref{3.*}, and $ g : \R \to \R $ is defined
by\footnote{%
 Here $ \gR\(\D(x-h)\) $ means $ (2\pi)^{-1/2} \E^{-(x-h)^2/2} \, \D x $; that
 is, $ \ga_{\scriptscriptstyle\R,h}(\D x) $, where $
 \ga_{\scriptscriptstyle\R,h} $ is $ \gR $ shifted by $ h $.}
\[
\tag{3} g(y) = \sup_{h \in \R} \, \log \int_{-\infty}^\infty \!\!\! \exp 2 \bigg|
\psi \Big( \frac{x+y}{2} \Big) - \psi \Big( \frac{x-y}{2} \Big) \bigg| \,
\gR\(\D(x-h)\) \, .
\]

(b)
The same holds in the complex-valued setup, with the following trivial
modifications: $ \phi \in L_2(\R\to\C) $; $|\phi(t)|^2$ (rather than
$\phi^2(t)$); $ \psi : \C \to \R $; $ \int_\C \dots \gC(\dots) $ (rather than $
\int_\R \dots \gR(\dots) $); $ \wC $ (rather than $w$); $ g : \C \to \R $; and $
h,y \in \C $.
\end{lemma}
\end{sloppypar}

\begin{proof}
\textsc{The real-valued case} (a). We inspect \ref{2.rem}(a,b,c) for the
process $ X = \frac1C \psi(G) $ where $ G = \phi*w $.

\begin{sloppypar}
(a) $ \exp \int_{[0,1)} |X_t| \, \D t \le \int_{[0,1)} \exp |X_t| \, \D t $ by
convexity of $ \exp $, and $ \Ex \exp |X_t| = \Ex \exp \frac1C |\psi(G_t)| =
\int_\R \exp \frac1C |\psi(u)| \, \gR(\D u) $, thus, $ \Ex \exp \int_{[0,1)}
|X_t| \, \D t \le 2 $ by (1).
\end{sloppypar}

(b) $ \Ex X_0 = \frac1C \Ex \psi(G_0) = \frac1C \int_\R \psi(u) \, \gR(\D u) = 0
$ (by assumption).

(c) It is sufficient to prove that $ \Ex \exp \int_\R |Y_t| \, \D t \le 2 $ where $ Y $ is
the leak of the split $ (X^0,X^-,X^+) = \( \frac1C \psi(G^0), \frac1C \psi(G^-),
\frac1C \psi(G^+) \) $ of $ X $, and the split $(G^0,G^-,G^+)$ of $G$ is given
by \eqref{3.*}. We introduce $ \psi_2 : \R \to \R $ by $ \psi_2(x) = 2
\psi(\frac12 x) $; by (3) the pair $ \psi_2, g $ satisfies \eqref{3.6} with the
Gaussian measure $ \mu = \gR $. Thus \eqref{3.13} holds for $\psi_2$,
$\rho=2$ and $ M \ge C_0 $, where $ C_0 $ is given by Lemma \ref{3.11} for $
\eps = \sqrt2-1 $. That is,
$ \Ex \exp \frac2C \int_0^\infty | \psi(G^+_t) - \psi(G^0_t) | \, \D t \le
\int_0^1 \Ex \exp \sum_{k=0}^\infty g( G^+_{(\theta+k)C} -
G^0_{(\theta+k)C} ) \, \D\theta $. In combination with (2) it gives $ \Ex \exp 2
\int_0^\infty | X_t^+ - X_t^0 | \, \D t \le 2 $. Similarly, $ \Ex \exp 2
\int_{-\infty}^0 | X_t^- - X_t^0 | \, \D t \le 2 $. Finally, $ \Ex \exp
\int_\R |Y_t| \, \D t = \Ex \exp \( \frac12 \cdot 2 \int_{-\infty}^0 | X_t^- - X_t^0 | \,
\D t + \frac12 \cdot 2 \int_0^\infty | X_t^+ - X_t^0 | \, \D t \) \le \Ex \frac12 \( \exp
2 \int_{-\infty}^0 | X_t^- - X_t^0 | \, \D t + \exp 2 \int_0^\infty | X_t^+ - X_t^0 |
\, \D t \) \le 2 $.

\textsc{The complex-valued case} (b) needs only few trivial modifications: $ G =
\phi * \wC $; $ \gC $ instead of $ \gR $; and $ \psi_2 : \C \to \R $.
\end{proof}

\begin{remark}\label{3.15a}
If the left-hand side of (1) is finite for some $C$, then it converges to $1$
as $ C \to \infty $ by the dominated convergence theorem. Moreover, denoting it
by $f(C)$ we have $ f(MC) \le \( f(C) \)^{1/M} $, since generally $ \Ex \exp
\frac1M X \le (\Ex \exp X)^{1/M} $ for all $ M \in [1,\infty) $ and arbitrary
random variable $X$.\footnote{%
 Apply the inequality $ \Ex Y \le (\Ex Y^M)^\frac1M $ to $ Y = \E^{X/M} $.}
Similarly, denoting the right-hand side of (3) by $ g_\psi(y) $ we have $
g_{\frac1M \psi}(y) \le \frac1M g_\psi(y) $. Also, denoting the left-hand side
of (2) by $ f_g(C) $ we have $ f_{\frac1M g} (C) \le \(f_g(C)\)^{1/M} $ for all $
M \in [1,\infty) $, and $ f_g(MC) \le f_g(C) $ for all $ M \in \{1,2,3,\dots\}
$,\footnote{%
 Since $ \{ (\theta+k)M : k=0,1,\dots \} \subset \{ (\theta M \mod 1) + k :
 k=0,1,\dots \} $.}
whence $ f_{\frac1M g}(MC) \le \( f_g(C) \)^{1/M} $ for all $ M \in
\{1,2,3,\dots\} $.

Thus, replacing ``$\le2$'' with the weaker ``$<\infty$'' in (1), (2), and
replacing $ \exp 2 |\dots| $ with $ \exp |\dots| $ in (3), we still
get splittability (rather than \splittability{C} with the given $C$, of
course). This applies to both cases, (a) and (b).
\end{remark}

\begin{proposition}\label{3.155}
(a)
Let positive numbers $ C, \bal, \de $ satisfy $ \de \le 1 $ and $ \bal >
1+\frac2\de $. Let $ \phi \in L_2(\R) $ satisfy $ \int_\R \phi^2(t) \, \D
t = 1 $ and $ \esssup_{t\in\R} (1+|t|)^\bal \phi^2(t) < \infty $. Let a Lebesgue
measurable function $ \psi : \R \to \R $ satisfy $ \int_\R \exp \frac1C
|\psi(u)| \, \gR(\D u) < \infty $, $ \int_\R \psi(u) \gR(\D u) = 0 $, and
\[
\log \int_\R \exp \bigg| \psi \Big( \frac{x+h+y}{2} \Big) - \psi \Big(
\frac{x+h-y}{2} \Big) \bigg| \, \gR(\D x) \le C|y|^\de
\]
for all $ y,h \in \R $. Then the random process $ \psi(\phi*w) $ is splittable.

(b)
Let positive numbers $ C, \bal, \de $ satisfy $ \de \le 1 $ and $ \bal >
1+\frac2\de $. Let $ \phi \in L_2(\R\to\C) $ satisfy $ \int_\R |\phi(t)|^2 \, \D
t = 1 $ and $ \esssup_{t\in\R} (1+|t|)^\bal |\phi(t)|^2 < \infty $. Let a Lebesgue
measurable function $ \psi : \C \to \R $ satisfy $ \int_\C \exp \frac1C
|\psi(z)| \, \gC(\D z) < \infty $, $ \int_\C \psi(z) \gC(\D z) = 0 $, and
\[
\log \int_\C \exp \bigg| \psi \Big( \frac{x+h+y}{2} \Big) - \psi \Big(
\frac{x+h-y}{2} \Big) \bigg| \, \gC(\D x) \le C|y|^\de
\]
for all $ y,h \in \C $. Then the random process $ \psi(\phi*\wC) $ is splittable.
\end{proposition}

We need the following generalization of Lemma~\ref{2.6}.

\begin{lemma}\label{3.2.6}
The inequality
\[
\Ex \exp \int_A |G_t|^\de \, \D t \le \int_\R \exp(( C|u|^\de ) \, \gR(\D u) \le
\infty\, ,
\]
where $ C = \int_A ( \Ex G_t^2 )^{\de/2} \, \D t $, holds for every $ \de \in
(0,2) $, every Lebesgue measurable set $ A \subset \R $ and every centered
Gaussian process $G$ on $A$ whose covariation function $ (s,t) \mapsto \Ex G_s
G_t $ is Lebesgue measurable on $ A \times A $.
\end{lemma}

\begin{remark}\label{3.2.7}
Remark \ref{2.7} applies, still. That is, we assume that the Hilbert space $
L_2(\Om) $ is separable. This lemma generalizes readily to $ A \subset \R^d $,
and moreover, to arbitrary measure space $A$.
\end{remark}

\begin{proof}[Proof of Lemma \ref{3.2.6}]\let\qed\relax
Similarly to the proof of Lemma \ref{2.6}, denoting $ \si_t = \sqrt{ \Ex G_t^2 }
$ we have $ \int_A \si_t^\de \, \D t = C $ and
\[
\exp \int_A |G_t|^\de \, \D t = \exp \bigg(
C \int_A \frac{|G_t|^\de}{\si_t^\de} \frac{\si_t^\de}C \, \D t \bigg) \le \int_A
\Big( \exp C \frac{|G_t|^\de}{\si_t^\de} \Big) \frac{\si_t^\de}C \, \D t \, ;
\]
thus,
\[
\Ex \exp \int_A |G_t|^\de \, \D t \le \int_A \Big( \Ex \exp
C \frac{|G_t|^\de}{\si_t^\de} \Big) \frac{\si_t^\de}C \, \D t = \int_\R
\exp(C|u|^\de) \, \gR(\D u) \, . \qquad \rlap{$\qedsymbol$}
\]
\end{proof}

\begin{sloppypar}
\begin{proof}[Proof of proposition~\ref{3.155}]
\textsc{The real-valued case} (a) According to Lemma~\ref{3.14} and
Remark~\ref{3.15a}, it is sufficient to prove Condition (2) of Lemma~\ref{3.14}
with ``$<\infty$'' instead of ``$\le2$''. Having $ g(y) \le C|y|^\de $, we check
that $ \int_0^1 \Ex \exp C \sum_{k=0}^\infty | G^+_{(\theta+k)C} -
G^0_{(\theta+k)C} |^\de \, \D\theta < \infty $ (the other case, $ G^- - G^0 $,
being similar). Applying Lemma~\ref{3.2.6} to the countable
measure space $ \{0,1,2,\dots\} $ with the counting measure (according to
Remark~\ref{3.2.7}) we reduce it to $ \esssup_{\theta \in (0,1)} C_\theta <
\infty $, where $
C_\theta = \sum_{k=0}^\infty \( (\Ex(G^+_{(\theta+k)C} - G^0_{(\theta+k)C})^2
\)^{\de/2} $. Similarly to the proof of Prop.~\ref{2.4}(c) we have $ \Ex ( G^+_t
- G^0_t )^2 = 2\int_t^\infty \phi^2(s) \, \D s \le \const \cdot
(1+t)^{-(\bal-1)} $ with a constant dependent only on $ \bal $ and the given $
\esssup $. Thus, $ C_\theta \le \sqrt{\const} \sum_{k=0}^\infty
(1+kC)^{-\frac12(\bal-1)\de} < \infty $, since $ \bal > 1+\frac2\de $.

\textsc{The complex-valued case} (b) needs only few trivial modifications: $ G =
\phi * \wC $; $ \gC $ instead of $ \gR $; $ \psi_2 : \C \to \R $; and $
|\dots|^2 $ rather than $ (\dots)^2 $.
\end{proof}
\end{sloppypar}

\begin{proposition}\label{3.19}
Let $ \phi \in L_2(\R\to\C) $ satisfy $ \int_\R |\phi(t)|^2 \, \D t = 1 $ and $
\esssup_{t\in\R} (1+|t|)^\bal |\phi(t)|^2 < \infty $ for some $ \bal > 3 $. Then
the following random process $X$ is splittable:
\[
X = \log | \phi * \wC | + \gEuler/2 \, , \quad \text{that is,} \quad X_t = \log
\bigg| \int_\R \phi(t-s) \wC(s) \, \D s \bigg| + \frac12 \gEuler \, .
\]
\end{proposition}

\begin{sloppypar}
\begin{proof}
We use Prop.~\ref{3.155}(b) for $\de=1$ and $ \psi(z) = \log |z| + \frac12
\gEuler $.

The condition $ \int_\C \exp \frac1C |\psi(z)| \, \gC(\D z) < \infty $ is
satisfied for $ C \ge 1 $; indeed, $ \int_\C \exp |\psi(z)| \, \gC(\D z) =
\int_\C \exp \big| \log|z| + \frac12 \gEuler \big| \, \gC(\D z) = \int_0^\infty
\exp \big| \log r + \frac12 \gEuler \big| \cdot \frac1\pi \E^{-r^2} \cdot 2\pi r
\, \D r = \int_0^\infty \exp \big| \frac12 \log s + \frac12 \gEuler \big| \cdot
\E^{-s} \, \D s < \infty $.

The condition $ \int_\C \psi(z) \, \gC(\D z) = 0 $ is satisfied, since $
\int_\C \psi(z) \, \gC(\D z) = \int_\C \( \log |z| + \frac12 \gEuler
\)  \, \gC(\D z) = \int_0^\infty \log r \cdot \frac1\pi \E^{-r^2} \cdot
2\pi r \, \D r + \frac12 \gEuler = \int_0^\infty \frac12 \log s \cdot
\E^{-s} \, \D s + \frac12 \gEuler = 0 $.

The condition $ \log \int_\C \dots \le C|y| $ is ensured for large $C$ by Lemma
\ref{3.15} (for $ p=2 $ and $ \si=1/\sqrt2 $).
\end{proof}
\end{sloppypar}

\begin{remark}\label{3.17}
Let $ \phi \in L_2(\R) $ satisfy $ \int_\R \phi^2(t) \, \D t = 1 $ and $
\esssup_{t\in\R} (1+|t|)^\bal \phi^2(t) < \infty $ for some $ \bal > 5 $. Then
the following random process $X$ is splittable:
\[
X = \log | \phi * w | + \be \, , \quad \text{that is,} \quad X_t = \log
\bigg| \int_\R \phi(t-s) w_s \, \D s \bigg| + \be
\]
where $ \be = -\int_\R \log|u| \, \gR(\D u) $.

This is proved similarly to Prop.~\ref{3.19}, using Lemma~\ref{3.16}, and
Prop.~\ref{3.155}(a) for $\de=1/2$. 
\end{remark}

\section[Dimension two: basics]
  {\raggedright Dimension two: basics}
\label{sect4}
\subsection{Split of split, leak for leak}
\label{4a}

In dimension $2$, regretfully, our technique fails to give results of the form
``if $X$ is splittable, then $F(X)$ is splittable''. Here is why. Splittability
of $ X = (X_{t_1,t_2})_{t_1,t_2\in\R} $ is defined \cite[Def.~1.3]{IV} in two
stages. First, \splittability{(1,1)}; roughly, it requires existence of a split,
good (that is, with small leak) along the line $t_1=0$,
\[
X^0_{t_1,t_2} = X^-_{t_1,t_2} \! - \text{\small leak} \;\;\; \text{for } t_1 < 0
\, , \quad
X^0_{t_1,t_2} = X^+_{t_1,t_2} \! - \text{\small leak} \;\;\; \text{for } t_1 \ge
0 \, ,
\]
and another split, good along the line $ t_2=0 $.
Second, \splittability{1} is \splittability{(1,2)}; roughly, it requires a split
good along the line $t_1=0$ whose leak, being \splittable{(1,1)}, has a split
good along the line $t_2=0$.
As noted in Sect.~\ref{sect2}, a leak of the transformed process $ F(X) $ is not
a transformed leak of $X$. Thus, we cannot deduce \splittability{1} of the
former leak from \splittability{1} of the latter leak.

\begin{remark}
The leak along the line $t_1=0$ is the leak defined
by \cite[Def.~1.2]{IV}. The leak along the line $t_2=0$ is defined similarly
(replace $ t_1<0 $, $ t_1\ge0 $ with $ t_2<0 $, $ t_2\ge0 $ in that definition),
and may be obtained by swapping the two coordinates, taking the leak defined
by \cite[Def.~1.2]{IV}, and swapping the coordinates again.
\end{remark}

In order to handle a split of the leak of another split, we need a split of
another split, introduced below.

A split of a random field $X$, being a triple $(X^0,X^-,X^+)$ of random fields
(rather than a single random field), is not a random \emph{signed} measure, but may be
thought of as a random \emph{vector} measure with values in $\R^3$. Thus we introduce
the space $\VM_2^p$ of all locally finite vector Borel measures on $\R^2$ with
values in $\R^p$. Clearly, $ \VM_2^{p+q} = \VM_2^p \oplus \VM_2^q $. By
a \component{p} random field (on $\R^2$) we mean a measurable map
$ \Om \to \VM_2^p $. The notions ``independent'' and ``identically
distributed'', hence also ``split'', apply to \component{p} random fields. A
split of a \component{1} random field is a \component{3} random
field. Similarly, a split of a \component{p} random field, being a random
measure with values in $ \R^p \oplus \R^p \oplus \R^p = \R^{3p} $, is
a \component{3p} random field.

If $ (X^0,X^-,X^+) $ is a split of a \component{(p+q)} $ X = (Y,Z) $ (where $Y$
is \component{p} and  $Z$ is \component{q}), then $ X^0 = (Y^0,Z^0) $, $ X^- =
(Y^-,Z^-) $, $ X^+ = (Y^+,Z^+) $; necessarily,\footnote{%
 This necessary condition is not sufficient (see the next footnote).}
$ (Y^0,Y^-,Y^+) $ is a split of $Y$, and $ (Z^0,Z^-,Z^+) $ is a split of $Z$.

A split $ (X^0,X^-,X^+) = (X^\al)_{\al\in\{0,-,+\}} = X^\doT $ of
a \component{1} $X$ is a \component{3} random field, and we may consider
its \component{9} split $ (X^{\doT,0},X^{\doT,-},X^{\doT,+}) =
(X^{\doT,\be})_{\be\in\{0,-,+\}} = (X^{\al,\be})_{\al,\be\in\{0,-,+\}} =
X^{\doT,\doT} $.
\[
\begin{gathered}\includegraphics[scale=1]{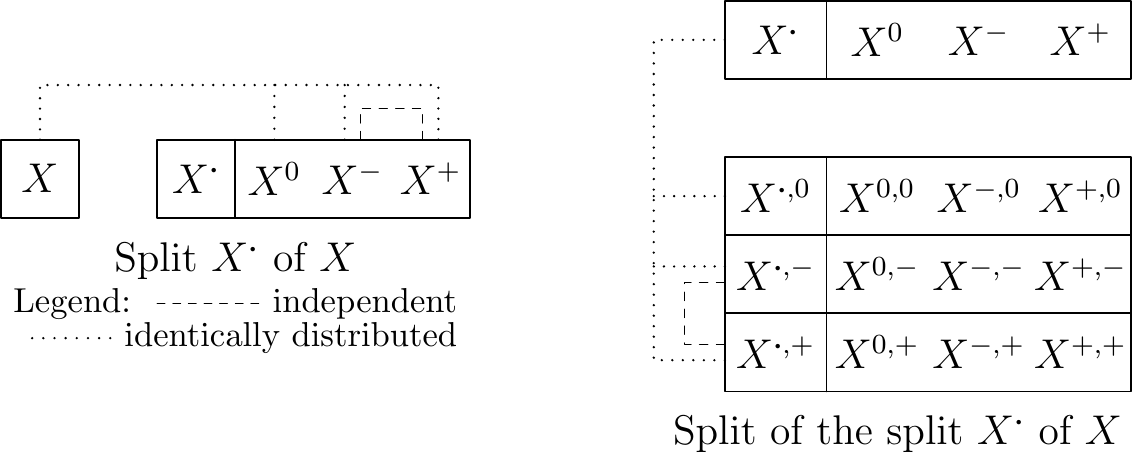}\end{gathered}
\vspace{-2.5pt}
\]
In terms of $ X^\doT $ we may say that $X$ is just something (denote it $X^*$)
distributed like $ X^\al $ for some $\al$, therefore, for all $\al$. Likewise,
in terms of $ X^{\doT,\doT} $ we may say that $X^\doT$ is just something (denote
it $X^{\doT,*}$) distributed like $ X^{\doT,\be} $ for some $\be$, therefore,
for all $\be$.
\enlargethispage{5pt}
\begin{sloppypar}
Necessarily,
$ X^{0,\doT} = (X^{0,0},X^{0,-},X^{0,+}) $ is a split of $ X^0 $,
$ X^{-,\doT} = (X^{-,0},X^{-,-},X^{-,+}) $ is a split of $ X^- $, and
$ X^{+,\doT} = (X^{+,0},X^{+,-},X^{+,+}) $ is a split of $ X^+ $.
Nevertheless, $ (X^{0,\doT}, X^{-,\doT}, X^{+,\doT}) $ need not be a
split.\footnote{%
 For a counterexample take four independent copies $ X^{-,-}, X^{-,+}, X^{+,-},
 X^{+,+} $ of $X$, and let $ X^{0,0} = X^{-,0} = X^{0,-} = X^{-,-} $, $ X^{+,0}
 = X^{0,+} =  X^{-,+} $. Then $ X^{-,\doT} = (X^{-,0},X^{-,-},X^{-,+}) =
 (X^{-,-},X^{-,-},X^{-,+}) $ and $ X^{+,\doT} = (X^{+,0},X^{+,-},X^{+,+}) =
 (X^{-,+},X^{+,-},X^{+,+}) $ are neither independent nor identically
 distributed.}
If it is a split, then we call $ X^{\doT,\doT} $ a \emph{\splt{2}} of
$X$. Every \splt{2} $ X^{\doT,\doT} $ of $X$ gives both the split $
(X^{\doT,0},X^{\doT,-},X^{\doT,+}) $ of the split $X^{\doT,*}$ of $X$, and
another split $ (X^{0,\doT},X^{-,\doT},X^{+,\doT}) $ of another split
$X^{*,\doT}$ of $X=X^{*,*}$.
\[
\begin{gathered}\includegraphics[scale=1]{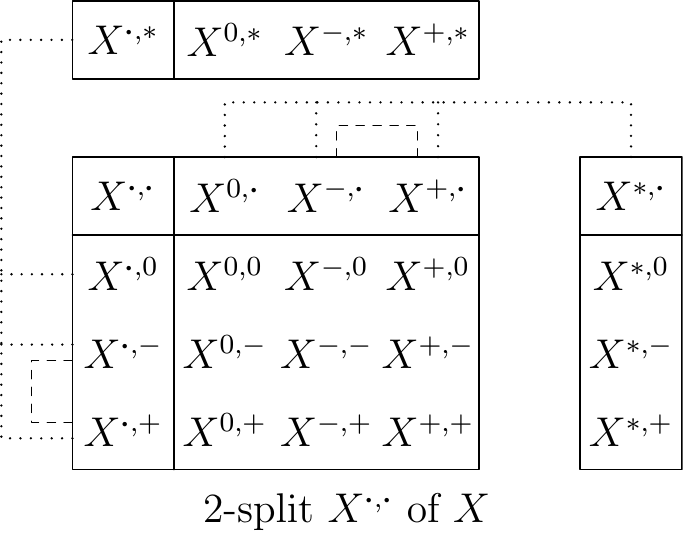}\end{gathered}
\vspace{-2.5pt}
\]
\end{sloppypar}
By a \transformation{p,q} we mean a measurable map $ F : \VM_2^p \to \VM_2^q
$. Given a \component{p} random field $X$ and a \transformation{p,q} $F$, we get
a \component{q} random field $ F(X) $.

If $X$ is a \component{p} random field and $ (X^\al)_{\al\in\{0,-,+\}} $ a split
of $X$, then
\begin{equation}\label{4.2}
\(F(X^\al)\)_{\al\in\{0,-,+\}} \quad \text{is a split of} \quad F(X)
\end{equation}
for arbitrary \transformation{p,q} $F$. (Recall \eqref{2.15}.) Similarly,\footnote{%
 Use the \transformation{3p,3q}  $F^3 : \mu_1 \oplus \mu_2 \oplus \mu_3 \mapsto
 F(\mu_1) \oplus F(\mu_2) \oplus F(\mu_3) $.}
if $ (X^{\al,\be})_{\al,\be\in\{0,-,+\}} $ is a \splt{2}
of $X$, then
\begin{equation}\label{4.25}
\(F(X^{\al,\be})\)_{\al,\be\in\{0,-,+\}} \quad \text{is a \splt{2} of} \quad
F(X) \, .
\end{equation}

The leak (along the line $t_i=0$ for a given $ i \in \{1,2\} $) of a split
$(X^0,X^-,X^+)$ of a \component{1} $X$ is $ \Leak_i(X^0,X^-,X^+) $, where
$ \Leak_1 $ is such a \transformation{3,1}:
\[
\nu = \Leak_1 (\mu^0, \mu^-, \mu^+) \equiv \left\{
\begin{gathered}
\forall A\subset(-\infty,0)\times\R \quad \nu(A)=\mu^-(A) - \mu^0(A) \,
; \\
\forall A\subset[0,\infty)\times\R \quad \nu(A)=\mu^+(A) - \mu^0(A)
\end{gathered}
\right.
\]
(recall \cite[Def.~1.2]{IV}), and $ \Leak_2 $ is defined similarly (using
$\R\times(-\infty,0)$ and $\R\times[0,\infty)$). Less formally, $ \Leak_i (X^0,
X^-, X^+) : (t_1,t_2) \mapsto X^{\sgn t_i}_{t_1,t_2} - X^0_{t_1,t_2} $.

Generally, given a split $(X^0,X^-,X^+)$ of $X$ and a transformation $F$, we get
a split $\(F(X^0),F(X^-),F(X^+)\)$ of $ F(X) $; its leak
$\Leak_i\(F(X^0),F(X^-),F(X^+)\)$ need not be equal to $
F\(\Leak_i(X^0,X^-,X^+)\) $ (unless $F$ is just multiplication by a given
function).

\begin{sloppypar}
The same applies to a \component{p} $X$ and a \transformation{p,q} $F$. In
particular, it applies to the \transformation{3,1} $ \Leak_1 $ and the split $
(X^{\doT,0},X^{\doT,-},X^{\doT,+}) = (X^{\doT,\be})_\be $ of the \component{3}
random field $ X^{\doT,*} $:
\[
\( \Leak_1 ( X^{\doT,\be} ) \)_\be \quad \text{is a split of} \quad X^{(1)}
\]
whenever $ X^{\doT,\doT} $ is a \splt{2} of $X$; here $ X^{(1)} =
\Leak_1(X^{\doT,*}) $, that is, $ X^{(1)}_{t_1,t_2} = X^{\sgn t_1,*}_{t_1,t_2} -
X^{0,*}_{t_1,t_2} $.
\end{sloppypar}

According to \cite[Def.~1.2(b)]{IV}, the leak of a split of $X$ is called a leak
\emph{for} $X$. In terms of this ``leak for'' relation we see that, first,
$ X^{(1)} $ is a leak for $X$, and second,
\[
X^{(1,2)} = \Leak_2 \Big( \( \Leak_1 ( X^{\doT,\be} ) \)_\be \Big) \quad
\text{is a leak for the leak } X^{(1)} \text{ for } X \, .
\]
Less formally, $ X^{(1,2)}_{t_1,t_2} = X^{\sgn t_1,\sgn t_2}_{t_1,t_2} -
X^{0,\sgn t_2}_{t_1,t_2} - X^{\sgn t_1,0}_{t_1,t_2} + X^{0,0}_{t_1,t_2} $.
Similarly, we may apply $ \Leak_2 $ to the split $ (X^{\al,\doT})_\al $ of the
split $ X^{*,\doT} $ of $X$, and get
\[
X^{(2,1)} = \Leak_1 \Big( \( \Leak_2 ( X^{\al,\doT} ) \)_\al \Big) \quad
\text{is a leak for the leak } X^{(2)} \text{ for } X,
\]
where $ X^{(2)} = \Leak_2(X^{*,\doT}) $. We note that $ X^{(2,1)} = X^{(1,2)} $,
and call this random field the \leak{2} of the given \splt{2}. So, we arrive at
a leak for a leak, provided that a split of a split is given.

\subsection{Sufficient conditions}
\label{sect3b}

\begin{proposition}\label{4.4}
Let $ X^{\doT,\doT} $ be a \splt{2} of a random field $ X $ on $ \R^2 $. Then
the following $8=3+4+1$ conditions (together; on $X$ and $X^{(1)}$, $X^{(2)}$,
$X^{(1,2)}$ introduced above) are sufficient for \splittability{C} of $X$:
\begin{description}
\item[\rm(s1)] $X$ is stationary; that is, the distribution of its shift, the random field
$ (X_{t_1+s_1,t_2+s_2})_{t_1,t_2} $, does not depend on $ s_1, s_2 \in  \R $;
\item[\rm(s2)] the distribution of the random field $ (X^{(1)}_{t_1,t_2+s})_{t_1,t_2} $
does not depend on $ s \in  \R $;
\item[\rm(s3)] the distribution of the random field $ (X^{(2)}_{t_1+s,t_2})_{t_1,t_2} $
does not depend on $ s \in  \R $;
\end{description}
\begin{gather*}
\tag{a1} \Ex \exp \frac1C \int_{[0,1)\times[0,1)} |X_t| \, \D t \le 2 \, ; \\
\tag{a2} \Ex \exp \frac1C \int_{\R\times[0,1)} |X_t^{(1)}| \, \D t \le 2 \, ; \\
\tag{a3} \Ex \exp \frac1C \int_{[0,1)\times\R} |X_t^{(2)}| \, \D t \le 2 \, ; \\
\tag{a4} \Ex \exp \frac1C \int_{\R\times\R} |X_t^{(1,2)}| \, \D t \le 2 \, ; \\
\tag{b} \Ex X_t = 0 \quad \text{for all } t \in \R^2 \, .
\end{gather*}
\end{proposition}

\begin{proof}
WLOG, $ C = 1 $. First we check that the random field
$(X^{(1)}_{t_2,t_1})_{t_1,t_2} $ is \splittable{(1,1)}. We inspect
\cite[Def.~1.3(a,b,c)]{IV} for $d=2$, $k=1$. Item (a) there follows from (a2)
and (s2) here. Item (b) there follows from (b) here, since a leak of a centered
random field is centered (evidently). Item (c) there requires a leak $ \ti Y $ for
$(X^{(1)}_{t_2,t_1+r})_{t_1,t_2} $ such that $ \Ex \exp \int_{\R^2} |\ti Y_t| \,
\D t \le 2 $. WLOG, $r=0$ due to (s2). We take $ \ti Y_{t_1,t_2} =
X_{t_2,t_1}^{(1,2)} $ and apply (a4).

Second, similarly, $ X^{(2)} $ is \splittable{(1,1)}.\footnote{%
 Similarly but a bit simpler; this time we do not need to swap $t_1,t_2$; $ \ti
 Y = X^{(2,1)} $ is a leak for $ (X_{t_1+r,t_2}^{(2)})_{t_1,t_2} $ irrespective
 of $r$.}

Finally, we check that $X$ is \splittable{(1,2)}, inspecting
\cite[Def.~1.3(a,b,c)]{IV} for $d=2$, $k=2$. Item (a) there follows from (a1)
and (s1) here. Item (b) there follows from (b) here. Item (c) there stipulates
two cases, $i=1$ and $i=2$.

Case $i=1$ requires a leak $ \ti Y $ for $(X_{t_1+r,t_2})_{t_1,t_2} $ such that
$ (\ti Y_{t_2,t_1})_{t_1,t_2} $ is \splittable{(1,1)}. WLOG, $r=0$ due to
(s1); we take $ \ti Y = X^{(1)} $.

Case $i=2$ requires a leak $ \ti Y $ for $(X_{t_2,t_1+r})_{t_1,t_2} $ such that
$ (\ti Y_{t_2,t_1})_{t_1,t_2} $ is \splittable{(1,1)}. WLOG, $r=0$ due to (s1);
we take $ \ti Y_{t_1,t_2} = X_{t_2,t_1}^{(2)} $.
\end{proof}

Similarly to Prop.~\ref{2.3} we have the following.

\begin{proposition}\label{4.5a}
Let $ X = (X_t)_{t\in\R^2} $ be a random field such that, almost
surely, the measure $ B \mapsto \int_B X_t \, \D t $ is absolutely continuous,
and $ Y_t = \psi(X_t) - \Ex \psi(X_t) $ for a given function $ \psi : \R \to \R
$ that satisfies the Lipschitz condition with a constant $ C_1 $.

(a) If $ X^\doT = (X^0,X^-,X^+) $ is a split of $X$ whose leak $ X^{(1)} =
\Leak_1(X^\doT) $ satisfies Condition (a2) of Prop.~\ref{4.4} with a constant $
C_2 $, then the leak $ Y^{(1)} = \Leak_1(Y^\doT) $, where $ Y^\al_t =
\psi(X^\al_t) - \Ex \psi(X^\al_t) $, satisfies Condition (a2) of Prop.~\ref{4.4}
with the constant $ C_1 C_2 $.

(b) The same holds for $ Y^{(2)} = \Leak_2(Y^\doT) $ and Condition (a3) of
Prop.~\ref{4.4}.

(c) If $ X^{\doT,\doT} $ is a \splt{2} of $X$ whose leaks $ X^{(1)}, X^{(2)} $
satisfy Conditions (a2), (a3) of Prop.~\ref{4.4} with a constant $ C_2 $, and $
Y^{\doT,\doT} $ is the corresponding \splt{2} of $ \psi(X) - \Ex \psi(X) $,
that is, $ Y^{\al,\be}_t = \psi(X^{\al,\be}_t) - \Ex \psi(X^{\al,\be}_t) $, then
the leaks $ Y^{(1)}, Y^{(2)} $ satisfy Conditions (a2), (a3) of Prop.~\ref{4.4}
with the constant $ C_1 C_2 $.
\end{proposition}

\begin{sloppypar}
\begin{proof}
Item (c) follows immediately from (a), (b); and (b) is similar to (a). We have
to prove (a).

We have $ |Y^{(1)}_{t_1,t_2}| = | Y^{\sgn t_1}_{t_1,t_2} - Y^0_{t_1,t_2} | = |
\psi(X^{\sgn t_1}_{t_1,t_2}) - \psi(X^0_{t_1,t_2}) | \le C_1 | X^{\sgn
t_1}_{t_1,t_2} - X^0_{t_1,t_2} | = C_1 |X^{(1)}_{t_1,t_2}| $ (since $ \Ex
\psi(X^\al_t) $ does not depend on $\al$); thus, $ \Ex \exp \frac1{C_1 C_2}
\int_{\R\times[0,1)} |Y^{(1)}_t| \, \D t \le \Ex \exp \frac1{C_2}
\int_{\R\times[0,1)} |X^{(1)}_t| \, \D t \le 2 $.
\end{proof}
\end{sloppypar}

\begin{remark}\label{4.6a}
This approach does not scale to $ Y^{(1,2)} $ and Condition (a4) of
Prop.~\ref{4.4}. The obstacle is that $ | \psi(a)-\psi(b)-\psi(c)+\psi(d) | $
need not be less than $ C_1 |a-b-c+d| $ unless $\psi$ is linear (affine). Rather,
$ | \psi(a)-\psi(b)-\psi(c)+\psi(d) | \le C_1 \min ( |a-b| + |c-d|, \, |a-c| +
|b-d| ) $, which will be used in Section \ref{sect5}. 
\end{remark}

\subsection{Poisson and Gaussian random fields}
\label{4c}

As was noted in Sect.~\ref{sect2} (before Def.~\ref{2.spli}), the centered Poisson
point process has a leakless split, due to independent restrictions to
$(-\infty,0)$ (the past) and $[0,\infty)$ (the future). Similarly, the centered
Poisson random field $X$ on $\R^2$, having independent restrictions to
$(-\infty,0)\times\R$ and $[0,\infty)\times\R$, has a leakless (along the line
$t_1=0$) split
\begin{equation}\label{***}
X^0_{t_1,t_2} = X^-_{t_1,t_2} \text{ for } t_1<0 \, , \qquad X^0_{t_1,t_2} =
X^+_{t_1,t_2} \text{ for } t_1 \ge 0 \, ,
\end{equation}
where $ X^-, X^+ $ are independent copies of $X$. That is, $ X^0_{t_1,t_2} =
X^{\sgn t_1}_{t_1,t_2} $.
Further, having independent restrictions to $\R\times(-\infty,0)$ and
$\R\times[0,\infty)$, the \component{3} random field $ X^\doT = (X^0,X^-,X^+) $
has a split leakless along the line $t_2=0$:
\[
X^{\doT,-}, X^{\doT,+} \text{ are independent copies of } X^\doT, \text{ and
} X^{\doT,0}_{t_1,t_2} = X^{\doT,\sgn t_2}_{t_1,t_2} \, .
\]
That is, $ X^{\doT,-} = (X^{0,-},X^{-,-},X^{+,-}) $, $ X^{\doT,+} =
(X^{0,+},X^{-,+},X^{+,+}) $; and if, for example, $ t_1\ge0 $, $ t_2<0 $, then
$ X^{\doT,0}_{t_1,t_2} = X^{\doT,-}_{t_1,t_2} =
(X^{0,-}_{t_1,t_2},X^{-,-}_{t_1,t_2},X^{+,-}_{t_1,t_2}) =
(X^{+,-}_{t_1,t_2},X^{-,-}_{t_1,t_2},X^{+,-}_{t_1,t_2}) $. 
More generally, introducing the ``first non-zero'' function
\[
\fnz (a,b) = a \text{ when }  a\ne0, \text{ and } b \text{ when } a=0 \, ,
\]
we have $ X^{\doT,\be} = (X^{0,\be},X^{-,\be},X^{+,\be}) $ where
\begin{equation}\label{4.3}
X^{\al,\be}_{t_1,t_2} = X^{\fnz(\al,\sgn t_1),\fnz(\be,\sgn
t_2)}_{t_1,t_2} \qquad \text{for } \al,\be \in \{-1,0,+1\} \, .
\end{equation}
Note that the four random fields $ X^{\al,\be} $ for $ \al,\be \in \{-1,1\} $
are independent; and all the nine random fields $ X^{\al,\be} $ for
$ \al,\be \in \{-1,0,1\} $ are functions of these four, each of the nine being
distributed as the centered Poisson random field. This way we get a split of a
split.

Alternatively, we may use $t_2$ when splitting $X$ and $t_1$ when splitting the
split of $X$; this leads to a different split of splits $ Y^{\al,\be}_{t_1,t_2}
= X^{\fnz(\al,\sgn t_2),\fnz(\be,\sgn t_1)}_{t_1,t_2} $. Swap of $ X^{-,+} $
and $ X^{+,-} $ preserves joint distributions and sends $ Y^{\al,\be}_{t_1,t_2}
$ to $ X^{\fnz(\be,\sgn t_1),\fnz(\al,\sgn t_2)}_{t_1,t_2} =
X^{\be,\al}_{t_1,t_2} $. We see that the $3\times3$ matrix $
(X^{\be,\al})_{\al,\be\in\{0,-,+\}} $, transposed to $ (X^{\al,\be})_{\al,\be} $,
being distributed like $ (Y^{\al,\be})_{\al,\be} $, is a split of splits. Thus,
$ X^{\doT,\doT} $ is a \splt{2} of $X$.

Arbitrary transformation may be applied by \eqref{4.2}, \eqref{4.25}. We summarize.

\begin{lemma}\label{4.5}
Let $ X^{\al,\be} $ for $ \al,\be \in \{-1,1\} $ be four independent copies of
the centered Poisson random field $X$ on $\R^2$, and five more random fields $
X^{\al,\be} $ for $ (\al,\be) \in \{-1,0,1\}\times\{-1,0,1\} \,\setminus\,
\{-1,1\}\times\{-1,1\} $ be defined by \eqref{4.3}. Let $F$ be
a \transformation{1,1}, and
\[
Y^{\al,\be} = F(X^{\al,\be}) \quad \text{for }
(\al,\be) \in \{-1,0,1\}\times\{-1,0,1\} \, .
\]
Then $ Y^{\doT,\doT} $ is a \splt{2} of $F(X)$.
\end{lemma}

More generally, instead of the centered Poisson random field one may use a
L\'evy random field, provided that it is a centered random measure whose
variation on a bounded set has some exponential moments. Still more generally,
waiving stationarity, one may use decomposable processes on $\R^2$ (see
Feldman \cite{Fe}). However, Gaussian component does not fit into this
framework because of locally infinite variation. We are especially interested in
the white noise $w$ on $\R^2$. Similarly to Prop.~\ref{2.4} we may adapt $w$ via
the convolution $ \phi * w $ with a given $ \phi \in L_2(\R^2) $; $ \phi * w $
is a stationary Gaussian random field, and we may apply to it a transformation
$F$. But, rather unexpectedly, in some notable cases we obtain a stationary
non-Gaussian random field $ F(X) $ via a nonstationary Gaussian random field
$X$. For this reason we replace $ \phi * w : t \mapsto \int_{\R^2} \phi(t-s)
w_s \, \D s $ with more general $ \phi \star w : t \mapsto \int_{\R^2} \phi(t,s)
w_s \, \D s $. Assuming that
\begin{equation}\label{4.6}
\begin{gathered}
\phi \text{ is a Lebesgue measurable function on } \R^2 \times \R^2
\text{, and}\\
 \textstyle \text{the function } \, t \mapsto \sqrt{ \int_{\R^2}
 |\phi(t,s)|^2 \, \D s } \, \text{ is locally integrable on } \R^2 \, ,
\end{gathered}
\end{equation}
we may treat the Gaussian random field $ \phi \star w $ as a random element of
$ \VM(2,1) $. Here is a counterpart of Lemma \ref{4.5}.

\begin{lemma}\label{4.8}
Let $ w^{\al,\be} $ for $ \al,\be \in \{-1,1\} $ be four independent copies of
the white noise $w$ on $\R^2$, and five more copies\footnote{%
 Not necessarily independent.}
$ w^{\al,\be} $ for $ (\al,\be) \in \{-1,0,1\}\times\{-1,0,1\} \,\setminus\,
\{-1,1\}\times\{-1,1\} $ be defined by
\begin{equation*}
w^{\al,\be}_{t_1,t_2} = w^{\fnz(\al,\sgn t_1),\fnz(\be,\sgn
t_2)}_{t_1,t_2} \qquad \text{for } \al,\be \in \{-1,0,+1\} \, .
\end{equation*}
Let $ \phi : \R^2 \times \R^2 \to \R $ satisfy \eqref{4.6}, and
\[
G^{\al,\be} = \phi \star w^{\al,\be} \quad \text{for }
 (\al,\be) \in \{-1,0,1\}\times\{-1,0,1\} \, .
\]
Then $ G^{\doT,\doT} $ is a \splt{2} of $ G = \phi \star w $.
\end{lemma}

\begin{proof}
The proof of Lemma \ref{4.5} needs only trivial modification. The notions
``independent'', ``identically distributed'' and ``split'' generalize readily to
random generalized functions such as the white noise; and the counterparts
of \eqref{4.2}, \eqref{4.25} for $ F(\dots) = \phi \star \dots $ hold evidently.
\end{proof}

\begin{remark}\label{4.9}
If in addition $F$ is a \transformation{1,1} and
\[
X^{\al,\be} = F(G^{\al,\be}) \quad \text{for }
(\al,\be) \in \{-1,0,1\}\times\{-1,0,1\} \, ,
\]
then $ X^{\doT,\doT} $ is a \splt{2} of $ F(\phi \star w) $.
Proof: \eqref{4.25} and Lemma \ref{4.8}.
\end{remark}

\begin{lemma}\label{4.10}
Let $\phi$ and $ G^{\doT,\doT} $ be as in Lemma \ref{4.8}, and $ G^{(1)} $, $
G^{(2)} $, $ G^{(1,2)} $ as before, then

\noindent(a) for $ k=1,2 $, $ G^{(k)} $ is distributed like $ \phi^{(k)} \star w $ where
\[
\phi^{(k)} (t_1,t_2;s_1,s_2) = \begin{cases}
  \sqrt2 (\sgn t_k) \phi(t_1,t_2;s_1,s_2), &\text{when }\, (\sgn s_k)(\sgn t_k)=-1,\\
  0, &\text{otherwise}.
\end{cases}
\]

\noindent(b) $ G^{(1,2)} $ is distributed like $ \phi^{(1,2)} \star w $ where $
\phi^{(1,2)}(t_1,t_2;s_1,s_2) = $
\[
\begin{cases}
  2 (\sgn t_1) (\sgn t_2) \phi((t_1,t_2;s_1,s_2), &\text{when }\,
   (\sgn s_1)(\sgn t_1)=-1=(\sgn s_2)(\sgn t_2),\\
  0, &\text{otherwise}.
\end{cases}
\]
\end{lemma}

\begin{proof}
(a) Case $k=1$: $ G^{(1)} = \Leak_1(G^{\doT,*}) $ is distributed like $ \Leak_1
(G^{\doT,\be}) $ (irrespective of $\be$); and
\begin{multline*}
\( \Leak_1 ( G^{\doT,\be} ) \)_{t_1,t_2} = G^{\sgn t_1,\be}_{t_1,t_2} -
G^{0,\be}_{t_1,t_2} = \( \phi \star (w^{\sgn t_1,\be} - w^{0,\be}) \)_{t_1,t_2}
= \\
= \iint \phi(t_1,t_2;s_1,s_2) ( w^{\sgn t_1,\be}_{s_1.s_2} - w^{\sgn
s_1,\be}_{s_1.s_2} ) \, \D s_1 \D s_2
\end{multline*}
since $ w^{0,\be}_{s_1,s_2} = w^{\sgn s_1,\be}_{s_1,s_2} $; when $ \sgn s_1 =
\sgn t_1 $, the integrand vanishes; otherwise the integrand is $ (\sgn t_1) (
w^{+,\be}_{s_1,s_2} - w^{-,\be}_{s_1,s_2} ) $; thus
\[
\Leak_1 ( G^{\doT,\be} ) = \frac1{\sqrt2} \phi^{(1)} \star ( w^{+,\be} -
w^{-,\be} ),
\]
and $ \frac1{\sqrt2} ( w^{+,\be} - w^{-,\be} ) $ is distributed like $w$. Case $k=2$:
similar.

(b) $ G^{(1,2)} = \Leak_2 \Big( \( \Leak_1 ( G^{\doT,\be} ) \)_\be \Big)
= \Leak_2 \Big( \( \frac1{\sqrt2} \phi^{(1)} \star (w^{+,\be}-w^{-,\be}) \)_\be
\Big) $,  whence
\begin{multline*}
G^{(1,2)}_{t_1,t_2} = (\sgn t_1) \iint\limits_{\sgn s_1\ne\sgn t_1}
\phi(t_1,t_2;s_1,s_2) \( ( w^{+,\sgn t_2}_{s_1,s_2} - w^{-,\sgn t_2}_{s_1,s_2} )
- \\
- ( w^{+,\sgn s_2}_{s_1,s_2} - w^{-,\sgn s_2}_{s_1,s_2} ) \) \, \D s_1 \D s_2
\end{multline*}
since $ w^{\al,0}_{s_1,s_2} = w^{\al,\sgn s_2}_{s_1,s_2} $;
when $ \sgn s_2 = \sgn t_2 $, the integrand vanishes; otherwise the integrand is
$ (\sgn t_2) \( ( w^{+,+}_{s_1,s_2} - w^{-,+}_{s_1,s_2} ) -
( w^{+,-}_{s_1,s_2} - w^{-,-}_{s_1,s_2} ) \) $; thus
\[
G^{(1,2)} = \frac12 \phi^{(1,2)} \star (w^{+,+}-w^{+,-}-w^{-,+}+w^{-,-}) \, .
\]
Finally, $ \frac12 (w^{+,+}-w^{+,-}-w^{-,+}+w^{-,-}) $ is distributed like $w$.
\end{proof}

In order to satisfy the stationarity conditions (s1), (s2), (s3) of
Prop.~\ref{4.4} we assume (from now on, till the end of Subsection \ref{4c})
that $ \phi(t,s) $ depends only on $ t-s $; we write $ \phi(t-s) $
instead of $ \phi(t,s) $, and $ \phi * w $ instead of $ \phi \star w
$. Accordingly, $ \phi^{(1)} (t_1,t_2; s_1,s_2) $ turns into $ \phi^{(1)}
(t_1,s_1, t_2-s_2) = \sqrt2 (\sgn t_1) \phi(t_1-s_1,t_2-s_2) $ when $ \sgn s_1
\ne \sgn t_1 $ (and $0$ otherwise). Similarly, $ \phi^{(2)} (t_1,t_2; s_1,s_2) $
turns into $ \phi^{(2)} (t_1-s_1, t_2,s_2) = \sqrt2 (\sgn t_2)
\phi(t_1-s_1,t_2-s_2) $ when $ \sgn s_2 \ne \sgn t_2 $ (and $0$ otherwise).
And $ \phi^{(1,2)} (t_1,t_2;s_1,s_2) $ turns into $ \phi^{(1,2)}
(t_1,t_2;s_1,s_2) = 2 (\sgn t_1) (\sgn t_2) \phi(t_1-s_1,t_2-s_2) $ when $ \sgn
s_1 \ne \sgn t_1 $, $ \sgn s_2 \ne \sgn t_2 $ (and $0$ otherwise). Now,
Lemma \ref{4.10} ensures (s1), (s2), (s3) of Prop.~\ref{4.4}. Thus, (a1)--(a4)
of Prop.~\ref{4.4} are sufficient for \splittability{C} of $ X = \phi * w $.

Here is a counterpart of Prop.~\ref{2.4}.

\begin{proposition}\label{4.12}
If $ \phi \in L_2(\R^2) $ satisfies $ \esssup_{t\in\R^2} ( 1+|t| )^{\bal}
\phi^2(t) < \infty $ for some $ \bal > 6 $, then the Gaussian random field $ \phi
* w $ is splittable.
\end{proposition}

\begin{lemma}\label{4.13}
Let $ K_1, K_2 \subset \R^2 $ be closed sets such that $ K_1 \cap K_2 = \{0\} $,
and $ \la x \in K_i $ whenever $ x \in K_i $, $ \la \in [0,\infty) $. Then
\[
\int_{K_1} \sqrt{ \int_{K_2} \frac{ \D s }{ (1+|t-s|)^\bal } } \> \D t < \infty
\]
for all $ \bal>6 $.
\end{lemma}

\begin{proof} \let\qed\relax
WLOG, $ K_1 \ne \{0\} $, $ K_2 \ne \{0\} $.
First, $ \inf \{ |x_1-x_2| : x_1 \in K_1, x_2 \in K_2, |x_1|=1 \} = \eps \in
(0,1] $ (since $ \{ x_1 \in K_1 : |x_1|=1 \} $ is compact, $ K_2 $ is closed, and
their intersection is empty). Second, $ \inf \{ |x_1-x_2| : x_1 \in K_1, x_2 \in
K_2, |x_1|\ge1 \} = \eps $ (since $ \big| \frac{x_1}{|x_1|} - \frac{x_2}{|x_1|}
\big| \ge \eps $). Third, $ \inf \{ |x_1-x_2| : x_1 \in K_1, x_2 \in
K_2, \max(|x_1|,|x_2|) \ge1 \} = \eps $ (since the ``Second'' claim applies also
to $K_2,K_1$). Thus,
\[
|x_1-x_2| \ge \max \( \eps |x_1|, \eps |x_2|, \big| |x_1| - |x_2| \big| \) \quad
\text{for all } x_1 \in K_1, x_2 \in K_2 \, .
\]
We estimate the internal integral $ I(t) = \int_{K_2} \frac{ \D s }{
(1+|t-s|)^\bal } $ using polar coordinates;
\[
I(t) \le \int_0^\infty \frac{ 2\pi r \, \D r}{(1+\max\( \eps r, \eps |t|, \big|
r-|t| \big| \)^\bal } \, .
\]
\textsc{Case $ |t| \ge \frac1{\eps} $. }
\begin{multline*}
I(t) \le \int_0^\infty \frac{ 2\pi r \, \D r}{(1+\max\( \eps |t|, \big| r-|t|
 \big| \)^\bal } \le \\
\int_0^{(1+\eps)|t|} \frac{ 2\pi r \, \D r}{(1+\eps|t|)^\bal} +
 \int_{(1+\eps)|t|}^\infty  \frac{ 2\pi r \, \D r}{(1+r-|t|)^\bal} = \\
\frac{ 2\pi }{(1+\eps|t|)^\bal} \cdot \frac12 (1+\eps)^2 |t|^2 + 2\pi
 \int_{1+\eps|t|}^\infty \frac{ u + |t| - 1 }{ u^\bal } \, \D u = \\
\frac{ \pi (1+\eps)^2 }{ \eps^2 } \frac{ (\eps|t|)^2 }{ (1+\eps|t|)^\bal } +
 2\pi \bigg(
 \frac1{(\bal-2)(1+\eps|t|)^{\bal-2}} +  \frac{ |t|-1 }{(\bal-1)(1+\eps|t|)^{\bal-1}}
 \bigg) \le \\
\frac{ \pi }{ (1+\eps|t|)^{\bal-2} } \bigg( \frac{(1+\eps)^2 }{ \eps^2 } +
 \frac2{\bal-2} + \frac{ 2 (|t|-1) }{(\bal-1)(1+\eps|t|) } \bigg) \le \\
\frac{\pi}{ \eps^{\bal-2} (1+|t|)^{\bal-2} } \bigg( \frac4{\eps^2} + \frac12 +
\frac2{5\eps} \bigg) = \frac{ \const(\eps) }{ (1+|t|)^{\bal-2} } \, ;
\end{multline*}
here ``$\const(\eps)$'' means a constant dependent on $\eps$ only.

\textsc{Case $ |t| \le \frac1{\eps} $. }
\begin{multline*}
I(t) \le \int_0^\infty \frac{ 2\pi r \, \D r}{\(1+\max( 0,
 r-\frac1\eps)\)^\bal } =
 \int_0^{1/\eps} 2\pi r \, \D r + \int_{1/\eps}^\infty \frac{ 2\pi r \, \D
 r}{\(1+r-\frac1\eps\)^\bal } = \\
\frac\pi{\eps^2} + 2\pi \int_1^\infty \frac{ u + \frac1\eps - 1 }{ u^\bal } \,
\D u \le \pi \Big( \frac1{\eps^2} + \frac2{\bal-2} + \frac2{(\bal-1)\eps} \Big)
\le \pi \Big( \frac1{\eps^2} + \frac12 + \frac2{5\eps} \Big) \, ;
\end{multline*}
just $\const(\eps)$. Finally,
\begin{multline*}
\int_{K_1} \sqrt{ I(t) } \, \D t \le
 \int_0^\infty 2\pi r \sqrt{ \frac{ \const(\eps) }{ (1+r)^{\bal-2} } } \, \D r
 \le \\
\const(\eps) \int_0^\infty \frac{\D r}{(1+r)^{\frac12\bal-2}} \le \frac{
 \const(\eps) }{ \bal - 6 } \, . \qquad \rlap{\qedsymbol}
\end{multline*}
\end{proof}

\begin{remark}\label{4.13a}
For $ \bal\le6 $ the integral diverges, unless at least one of $K_1$, $K_2$ is of
measure $0$. For the proof, consider $ K_{1,n} = \{ x \in K_1 : 2^n < |x| <
2^{n+1} \} $, $ K_{2,n} = \{ x \in K_2 : 2^n < |x| < 2^{n+1} \} $ and note that
$ \int_{K_{1,n}} \sqrt{ \int_{K_{2,n}} \frac{ \D s }{ (1+|t-s|)^\bal } } \> \D t
\ge \const \cdot (2^n)^2 \cdot \sqrt{ (2^n)^2 \cdot \frac1{2^{\bal n}} } = \const
\cdot 2^{\frac12(6-\bal)n} $.
\end{remark}

\begin{remark}\label{4.13b}
More generally, given $ \de>0 $, the integral $ \int_{K_1} \( \int_{K_2} \frac{
\D s }{ (1+|t-s|)^\bal } \)^{\de/2} \, \D t $ converges for all $ \bal >
2\(1+\frac2\de\) $. The proof is the same, except for the last phrase: finally,
\begin{multline*}
\int_{K_1} ( I(t) )^{\de/2} \, \D t \le
 \int_0^\infty 2\pi r \Big( \frac{ \const(\eps) }{ (1+r)^{\bal-2} } \Big)^{\de/2}
 \, \D r  \le \\
\const(\eps) \int_0^\infty \frac{\D r}{(1+r)^{\frac12(\bal-2)\de-1}} \le \frac{
 \const(\eps) }{ (\bal-2)\de - 4 } \, .
\end{multline*}
\end{remark}

\begin{remark}\label{4.13c}
Instead of integrating over $K_1$, we may take the sum over the integer points
in $K_1$:
\[
\sum_{t\in K_1\cap\Z^2} \Big( \int_{K_2} \frac{ \D s }{ (1+|t-s|)^\bal }
\Big)^{\de/2} \> \D t < \infty
\]
for all $ \bal>2\(1+\frac2\de\) $.

\begin{sloppypar}
Proof: $ \sum_{t\in K_1\cap\Z^2} \( I(t) \)^{\de/2} \le \const(\eps) \sum_{t\in
  K_1\cap\Z^2} (1+|t|)^{-(\bal-2)\de/2} \le \const(\eps) \sum_{t\in K_1\cap\Z^2}
\int_{|t|}^\infty \frac12(\bal-2)\de (1+r)^{-\frac{(\bal-2)\de+2}2} \, \D r =
\const(\eps) \int_0^\infty \frac12(\bal-2)\de (1+r)^{-\frac{(\bal-2)\de+2}2} N(r)
\, \D r $, where $N(r)$ is the number of $ t\in K_1\cap\Z^2 $ such that $ |t|<r
$; clearly, $ N(r) \le \pi \( r + \frac{\sqrt2}2 \)^2 $.

The same applies to a shifted lattice, and the bound is uniform over shifts:
\[
\sup_{\theta\in\R^2} \sum_{t\in K_1\cap(\Z^2+\theta)} \Big( \int_{K_2} \frac{ \D
  s }{ (1+|t-s|)^\bal } \Big)^{\de/2} \> \D t < \infty \, .
\]
\end{sloppypar}
\end{remark}

\begin{sloppypar}
\begin{proof}[Proof of Prop.~\ref{4.12}]
We inspect Conditions (a1)--(a4) of Prop.~\ref{4.4}, using Lemma \ref{4.10}, for $ G
= \phi * w $.

(a1) According to Lemma \ref{2.6} (used as in Item (a) of the proof of
Prop.~\ref{2.4}), in order to get (a1) for large $C$, it is sufficient to prove
that $ \int_{[0,1)\times[0,1)} \sqrt{ \Ex G_t^2 } \, \D t < \infty $. This is
immediate, since $ \Ex G_t^2 = \int_{\R^2} \phi^2(t-s) \, \D s = \int_{\R^2}
\phi^2(s) \, \D s $ is finite and does not depend on $t$.

(a2) By Lemma \ref{2.6} again, it is sufficient to prove that $
\int_{\R\times[0,1)} \sqrt{ \Ex (G^{(1)}_t)^2 } \, \D t < \infty $. Denoting by
  $M$ the given $\esssup$, we have $
\Ex (G^{(1)}_t)^2 = \iint \( \phi^{(1)} (t_1,s_1,t_2-s_2) \)^2 \, \D s_1 \D s_2
= 2 \iint_{\sgn s_1 \ne \sgn t_1} \phi^2(t_1-s_1,t_2-s_2) \, \D s_1 \D s_2 \le
2M \iint_{\sgn s_1 \ne \sgn t_1} \( 1 + \sqrt{ (t_1-s_1)^2+(t_2-s_2)^2 }
\)^{-\bal} \, \D s_1 \D s_2 = 2M \iint_{s_1>|t_1|} (1+\sqrt{s_1^2+s_2^2})^{-\bal} \,
\D s_1 \D s_2 \le 2M \iint_{s_1>0, s_1^2+s_2^2>t_1^2} (1+\sqrt{s_1^2+s_2^2})^{-\bal} \,
\D s_1 \D s_2 = 2M \int_{|t_1|}^\infty (1+r)^{-\bal} \pi r \, \D r \le 2\pi M
\int_{|t_1|}^\infty (1+r)^{-(\bal-1)} \, \D r = \frac{2\pi M}{\bal-2}
(1+|t_1|)^{-(\bal-2)} $, thus,
$ \int_{\R\times[0,1)} \sqrt{ \Ex (G^{(1)}_t)^2 } \, \D t \le \sqrt{ \frac{2\pi
M}{\bal-2} } \iint_{0<t_2<1} (1+|t_1|)^{-(\frac\bal2-1)} \, \D t_1 \D t_2 = \sqrt{
\frac{2\pi M}{\bal-2} } \cdot 2 \int_0^\infty (1+|t_1|)^{-(\frac\bal2-1)} \, \D
t_1 < \infty $ for $ \bal > 4 $, the more so, for $ \bal > 6 $.

(a3) Similar to (a2); just swap the two coordinates on $\R^2$.

\raggedright{
$\quad\,$ (a4) Again, it is sufficient to prove that $ \int_{\R^2} \sqrt{ \Ex
(G^{(1,2)}_t)^2 } \, \D t < \infty $. We have $ \Ex (G^{(1,2)}_{t_1,t_2})^2 =
\iint \( \phi^{(1,2)} (t_1,t_2;s_1,s_2) \)^2 \, \D s_1 \D s_2 = 4 \iint_{\sgn
s_1 \ne \sgn t_1,\sgn s_2 \ne \sgn t_2} \phi^2(t_1-s_1,t_2-s_2) \, \D s_1 \D s_2
\le 4M \iint_{\sgn s_1 \ne \sgn t_1,\sgn s_2 \ne \sgn t_2} \( 1 + \sqrt{
(t_1-s_1)^2+(t_2-s_2)^2 } \)^{-\bal} \, \D s_1 \D s_2 = 4M \iint_{s_1>0,s_2>0} \(
1 + \sqrt{ (|t_1|+s_1)^2+(|t_2|+s_2)^2 } \)^{-\bal} \, \D s_1 \D s_2 $.
We apply Lemma \ref{4.13} to $ K_1 = (-\infty,0] \times (-\infty,0] $, $ K_2 =
[0,\infty) \times [0,\infty) $, and note that $ \Ex (G^{(1,2)}_t)^2 \le 4M
\int_{K_2} (1+|t-s|)^{-\bal} \, \D s $ for $ t \in K_1 $, thus $ \int_{\R^2}
\sqrt{ \Ex (G^{(1,2)}_t)^2 } \, \D t \le 4 \int_{K_1} \sqrt{ 4M \int_{K_2}
(1+|t-s|)^{-\bal} \, \D s } \, \D t < \infty $ for $ \bal>6 $.}
\end{proof}
\end{sloppypar}

\section[Dimension two: more]
  {\raggedright Dimension two: more}
\label{sect5}
\subsection{Lipschitz functions of stationary Gaussian random fields}

Here is a counterpart of Prop.~\ref{3.1}.

\begin{sloppypar}
\begin{proposition}\label{4.16}
Assume that $ \phi \in L_2(\R^2) $ satisfies $ \int_{\R^2} \phi^2(t) \, \D t = 1
$, $ \esssup_{t\in\R} (1+|t|)^\bal \phi^2(t) < \infty $ for some $ \bal>6 $, and $
\psi : \R \to \R $ satisfies the Lipschitz condition, and $ \int_\R \psi(u)
\E^{-u^2/2} \, \D u = 0 $. Then the following random field is
splittable: $ X = \psi(\phi*w) $, that is, $ X_t = \psi\((\phi*w)_t\) $.
\end{proposition}
\end{sloppypar}

First, a rather trivial generalization of Lemma \ref{2.6}.

\begin{lemma}\label{4.17}
The inequality
\[
\Ex \exp \int_A (|\xi_t|+|\eta_t|) \, \D t \le \E^{C^2/2} \frac2{\sqrt{2\pi}}
\int_{-\infty}^C \E^{-u^2/2} \, \D u \le \exp \bigg( \sqrt{\frac2\pi} C
+ \frac12 C^2 \bigg) \, ,
\]
where $ C = \int_A ( \sqrt{ \Ex \xi_t^2 } + \sqrt{ \Ex \eta_t^2 } ) \, \D t $, holds
for every Lebesgue measurable set $ A \subset \R $ and every pair of centered
Gaussian process $\xi,\eta$ on $A$
whose covariation and cross-covariation functions are Lebesgue measurable
on $ A \times A $.
\end{lemma}

\begin{proof}\let\qed\relax
Similarly to the proof of Lemma \ref{2.6},
denoting $ \si_t = \sqrt{ \Ex \xi_t^2 } $ and $ \de_t = \sqrt{ \Ex \eta_t^2 } $ we
have $ \int_A ( \si_t + \de_t ) \, \D t = C $ and
\[
\Ex \exp C \frac{|\xi_t|}{\si_t} = \Ex \exp C \frac{|\eta_t|}{\de_t} =
\E^{C^2/2} \frac2{\sqrt{2\pi}} \int_{-\infty}^C \E^{-u^2/2} \, \D u
\]
for all $ t \in A $; by convexity of $ \exp $,
\begin{multline*}
\exp \int_A ( |\xi_t| + |\eta_t| ) \, \D t = \exp \bigg(
 C \int_A \frac{|\xi_t|}{\si_t} \frac{\si_t}C \, \D t + C \int_A
 \frac{|\eta_t|}{\de_t} \frac{\de_t}C \, \D t \bigg) \le \\
\int_A \Big( \exp C \frac{|\xi_t|}{\si_t} \Big) \frac{\si_t}C \, \D t +
 \int_A \Big( \exp C \frac{|\eta_t|}{\de_t} \Big) \frac{\de_t}C \, \D t \, ;
\end{multline*}
thus,
\[
\qquad\qquad
\Ex \exp \int_A ( |\xi_t| + |\eta_t| ) \, \D t \le
\E^{C^2/2} \frac2{\sqrt{2\pi}} \int_{-\infty}^C \E^{-u^2/2} \, \D u \, .
\qquad\qquad\qquad \rlap{$\qedsymbol$}
\]
\end{proof}

\begin{sloppypar}
\begin{proof}[Proof of Prop.~\ref{4.16}]
We use the \splt{2} $ G^{\doT,\doT} $ of $ G = \phi * w $ given by Lemma
\ref{4.8}, and the corresponding \splt{2} $ X^{\doT,\doT} $ of $ X = \psi(G) $,
that is, $ X^{\al,\be} = \psi(G^{\al,\be}) $ (recall Remark \ref{4.9}). The
corresponding leaks are
\begin{gather*}
X^{(1)}_{t_1,t_2} = X^{\sgn t_1,*}_{t_1,t_2} - X^{0,*}_{t_1,t_2} \, ; \quad
 X^{(2)}_{t_1,t_2} = X^{*,\sgn t_2}_{t_1,t_2} - X^{*,0}_{t_1,t_2} \, ; \\
X^{(1,2)}_{t_1,t_2} = X^{\sgn t_1,\sgn t_2}_{t_1,t_2} - X^{0,\sgn t_2}_{t_1,t_2}
 - X^{\sgn t_1,0}_{t_1,t_2} + X^{0,0}_{t_1,t_2} \, .
\end{gather*}
We inspect the conditions of Prop.~\ref{4.4}. Condition (s1) holds for $X$,
since it holds for $G$. Condition (s2) holds for $X^{(1)}$, not just since it
holds for $G^{(1)}$ (because $X^{(1)}$ is not $\psi(G^{(1)})$), but rather,
since the joint distribution of three fields $ G^{0,*}, G^{-,*}, G^{+,*} $ is
invariant under shifts $ (t_1,t_2) \mapsto (t_1,t_2+s) $. Similarly, Condition
(s3) holds for $X^{(2)}$. Condition (b) holds, since $ \int \psi \, \D\gR = 0
$.

It remains to check Conditions (a1)--(a4). For $ G, G^{(1)}, G^{(2)}, G^{(1,2)} $
they hold (as was seen in the proof of Prop.~\ref{4.12}). It follows by
Prop.~\ref{4.5a}(c) that (a2), (a3) hold for $ X^{(1)}, X^{(2)} $. Also, (a1)
holds for $X$ by the argument used in Item (a) of the proof of Prop.~\ref{2.3}
(just replace $ \psi(\cdot) $ with $ \psi(\cdot) - \psi(0) $, and integrate over
$ [0,1)\times[0,1) $ instead of $ [s,s+1) $).

Condition (a4) needs more effort, as noted in Remark \ref{4.6a}; we have $
X^{(1,2)}_{t_1,t_2} = \psi(a) - \psi(b) - \psi(c) + \psi(d) $, where $ a =
G^{\sgn t_1,\sgn t_2}_{t_1,t_2} $, $ b = G^{0,\sgn t_2}_{t_1,t_2} $, $ c =
G^{\sgn t_1,0}_{t_1,t_2} $, $ d = G^{0,0}_{t_1,t_2} $. Thus, $
|X^{(1,2)}_{t_1,t_2}| \le \min(|a-b|+|c-d|,|a-c|+|b-d|) $, assuming WLOG that
the Lipschitz constant of $\psi$ is $1$ (or less). We use $ |a-b|+|c-d| $ when $
|t_1| > |t_2| $, and $ |a-c|+|b-d| $ when $ |t_1| < |t_2| $.
Lemma \ref{4.17} reduces (a4) (for large $C$) to
$ \int_A ( \sqrt{ \Ex \xi_t^2 } + \sqrt{ \Ex \eta_t^2 } ) \, \D t < \infty $
where $ A = \{ (t_1,t_2) \in \R^2 : |t_1| > |t_2| \} $, $ \xi_{t_1,t_2} = a-b =
G^{\sgn t_1,\sgn t_2}_{t_1,t_2} - G^{0,\sgn t_2}_{t_1,t_2} $ and $
\eta_{t_1,t_2} = c-d = G^{\sgn   t_1,0}_{t_1,t_2} - G^{0,0}_{t_1,t_2} $. The
other case, $ |t_1| > |t_2| $, follows by swapping the two coordinates.

We note that $ \eta = \Leak_1(G^{\doT,0}) $ is distributed like $ G^{(1)} $,
while $ \xi $ is distributed differently, being glued out of two independent
fields, $ \Leak_1(G^{\doT,-}) $ on $ \{t_2<0\} $ and $ \Leak_1(G^{\doT,+}) $
on $ \{t_2\ge0\} $, each distributed like $ G^{(1)} $. Anyway, $ \Ex \xi_t^2 =
\Ex \eta_t^2 = \Ex (G_t^{(1)})^2 $ for all $ t \in \R^2 $; it is sufficient to
prove that $ \iint_{|t_1|>|t_2|} \sqrt{ \Ex (G_{t_1,t_2}^{(1)})^2 } \, \D t_1 \D
t_2 < \infty $. By Lemma \ref{4.10}(a), $ G^{(1)} $ is distributed like $
\phi^{(1)} \star w $; thus $ \Ex (G_{t_1,t_2}^{(1)})^2 = \iint_{\R^2} \(
\phi^{(1)} (t_1,s_1,t_2-s_2) \)^2 \, \D s_1 \D s_2 = 2 \iint_{\sgn s_1 \ne \sgn
t_1} \phi^2(t_1-s_1,t_2-s_2) \, \D s_1 \D s_2 $. We have to prove that
\[
\iint_{t_1>|t_2|} \sqrt{ \iint_{s_1<0} \( 1 + \sqrt{ (t_1-s_1)^2 + (t_2-s_2)^2 }
\)^{-\bal} \, \D s_1 \D s_2 } \> \D t_1 \D t_2 < \infty
\]
(the other case, $ t_1 < -|t_2| $ and $ s_1>0 $, gives the same). It remains to
apply Lemma \ref{4.13} to $ K_1 = \{ (t_1,t_2) : t_1 \ge |t_2| \} $, $ K_2 = \{
(s_1,s_2) : s_1 \le 0 \} $.
\end{proof}
\end{sloppypar}

\subsection{More general than Lipschitz condition}

The notion of a \noisy{\mu} sequence $(X_1,\dots,X_n)$ of (\valued{\R^p}) random
variables (recall Def.~\ref{noisy}), being insensitive to arbitrary permutations
of these random variables, may be thought of as a notion of a \noisy{\mu} family
$(X_i)_{i\in I}$ of random variables, indexed by a finite set $I$. By
Lemma \ref{3.12}, $ \Ex \prod_{i\in I} f_i(X_i) \le \prod_{i\in
I} \sup_{y\in\R^p} \int_{\R^p} f_i(y+z) \, \mu(\D z) $ for all Borel
measurable\footnote{%
 Or Lebesgue measurable, if $\mu$ is absolutely continuous.}
$ f_i : \R^p \to [0,\infty) $. Lemma \ref{3.5} gives
\begin{equation}\label{5.*}
\Ex \exp \sum_{i\in I} | \psi(X_i+Y_i) - \psi(X_i-Y_i) | \le \Ex \exp \sum_{i\in I} g(Y_i) \, ,
\end{equation}
while \eqref{3.6} remains intact.

Now we consider $ G = \phi * w $ as in Prop.~\ref{4.16}, $ \psi : \R \to \R
$ Lebesgue measurable (but generally not Lipschitz),
$ \int_\R \psi(u) \E^{-u^2/2} \, \D u = 0 $, $ X = \psi(G) $; \splt{2}s $
G^{\doT,\doT} $, $ X^{\doT,\doT} $ and leaks $ X^{(1)}, X^{(2)}, X^{(1,2)} $ as
in the proof of Prop.~\ref{4.16}. There, when checking the conditions of
Prop.~\ref{4.4}, the Lipschitz continuity of $\psi$ was needed only for
(a1)--(a4); these have to be checked anew via the ``weaken Lipschitz
condition'' \eqref{3.6}. For ensuring (a1) we just require that
\[
\int \exp \frac1C |\psi(u)| \E^{-u^2/2} \, \D u < \infty \quad \text{for }
C \text{ large enough.}
\]

\begin{lemma}\label{5.3}
For every $ \be \in \{-1,0,1\} $ the two fields $ G^{+,\be} + G^{0,\be} $ and $
G^{+,\be} - G^{0,\be} $ are independent (in both cases, real-valued and
complex-valued).
\end{lemma}

\begin{proof}
Similarly to the proof of Lemma \ref{3.9}, $ G^{+,\be} - G^{0,\be} = \phi \star
( w^{+,\be} - w^{0,\be} ) $ is made from $ (w^{+,\be}_{s_1,s_2} -
w^{-,\be}_{s_1,s_2})_{s_1<0} $, while $ G^{+,\be} + G^{0,\be} $ is made from $
(w^{+,\be}_{s_1,s_2} + w^{-,\be}_{s_1,s_2})_{s_1<0} $ and $
(w^{+,\be}_{s_1,s_2})_{s_1\ge0} $.
\end{proof}

The relation ``noisier than'' introduced before Lemma \ref{3.10} for random
processes, generalizes readily to random fields.

\begin{lemma}\label{5.4}
For every $ \be \in \{-1,0,1\} $ the field $ G^{+,\be} + G^{0,\be} $ is noisier
than $ \sqrt2 G^{0,\be} $ (in both cases, real-valued and complex-valued).
\end{lemma}

\begin{proof}
The proof of Lemma \ref{3.10} needs only trivial modifications: independent copy
$ \ti w^{+,\be} $ of $ w^{+,\be} $; $ U_{t_1,t_2} = \sqrt2 \iint_{s_1\ge0}
\phi(t_1-s_1,t_2-s_2) \ti w^{+,\be}_{s_1,s_2} \, \D s_1 \D s_2 $; and so on.
\end{proof}

\begin{lemma}\label{5.5}
(a)
Let $ G = \phi * w $ where $ \phi \in L_2(\R^2) $ satisfies
$ \int_{\R^2} \phi^2(t) \, \D t = 1 $ and $ \esssup_{t\in\R^2}
(1+|t|)^\bal \phi^2(t) < \infty $ for some $ \bal > 4 $. Then for every $ \eps > 0
$ there exists $ C $ (dependent only on $ \eps $, $ \bal $ and the $ \esssup $)
such that for every $ M \ge C $ and every $n$ the array $ \( (1+\eps)G_{kM,\ell
M} \)_{k,\ell\in\{1,\dots,n\}} $ is \noisy{\gR}.
\enlargethispage{14pt}

(b)
Let $ G = \phi * \wC $ where $ \phi \in L_2(\R^2\to\C) $ satisfies
$ \int_{\R^2} |\phi(t)|^2 \, \D t = 1 $ and $ \esssup_{t\in\R^2}
(1+|t|)^\bal |\phi(t)|^2 < \infty $ for some $ \bal > 4 $. Then for every $ \eps > 0
$ there exists $ C $ (dependent only on $ \eps $, $ \bal $ and the $ \esssup $)
such that for every $ M \ge C $ and every $n$ the array $ \( (1+\eps)G_{kM,\ell
M} \)_{k,\ell\in\{1,\dots,n\}} $ is \noisy{\gC}.
\end{lemma}

\begin{sloppypar}
\begin{proof}
Similarly to the proof of Lemma \ref{3.11}, $ | \Ex G_{-h} G_h | = | \int_{\R^2} \phi(-h-s)
\phi(h-s) \, \D s | \le A^2 \int_{\R^2} (1+|s+h|)^{-\bal/2} (1+|s-h|)^{-\bal/2} \,
\D s = 2 A^2 \iint_{s_1>0} \( 1+\sqrt{(s_1+|h|)^2+s_2^2} \)^{-\bal/2} \(
1+\sqrt{(s_1-|h|)^2+s_2^2} \)^{-\bal/2}  \, \D s_1 \D s_2 \le 2A^2
(1+|h|)^{-\bal/2} \iint_{\R^2} \( 1+\sqrt{s_1^2+s_2^2} \)^{-\bal/2} \, \D s_1 \D
s_2 = 2A^2 (1+|h|)^{-\bal/2} \int_0^\infty (1+r)^{-\bal/2} 2\pi r \, \D r \le 2A^2
\cdot \frac{2\pi}{\frac\bal2-2} (1+|h|)^{-\bal/2} $. Thus,
\[
\sum_{k,\ell=1}^\infty | \Ex G_{0,0} G_{kM,\ell M} | \le \frac{8\pi
A^2}{\bal-4} \sum_{k,\ell=1}^\infty \Big( 1
+ \frac{M}2 \sqrt{k^2+\ell^2} \Big)^{-\bal/2} \to 0 \quad \text{as }
M \to \infty \, .
\]
The rest of the proof of Lemma \ref{3.11} needs only trivial modifications.
\end{proof}
\end{sloppypar}

Here is a two-dimensional counterpart of \eqref{3.13}:
\begin{multline}\label{5.7}
\Ex \exp \frac1{M^2} \int_{[0,\infty)\times[0,\infty)} | \psi(\rho G^{\al,\be}_t)
 - \psi(\rho G^{0,\be}_t) | \, \D t \le \\
\int_0^1 \int_0^1 \Ex \exp \sum_{k,\ell=0}^\infty
g \Big( \frac \rho 2 \( G^{\al,\be} -
G^{0,\be} \)_{(\theta_1+k)M,(\theta_2+\ell)M} \Big) \, \D\theta_1 \D\theta_2
\end{multline}
for all $ \rho > \sqrt2 $ and all $M$ large enough (according to Lemma \ref{5.5}
for $ \eps = \sqrt2 - 1 $), all $ \al\in\{-1,1\} $ and all $ \be\in\{-1,0,+1\} $; this applies in both
cases, real-valued and complex-valued. Inequality \eqref{5.7} follows
from \eqref{5.*} and Lemmas \ref{5.3}, \ref{5.4}, \ref{5.5} in the same way
as \eqref{3.13} from Lemmas \ref{3.5}, \ref{3.9}, \ref{3.10}, \ref{3.11}; just
replace $ G^+_{kM} $ with $ G^{+,\be}_{kM,\ell M} $; $ \sum_k $ with
$ \sum_{k,\ell} $; $ G^+_{(\theta+k)M} $ with $
G^{+,\be}_{(\theta_1+k)M,(\theta_2+\ell)M} $; $ \int_0^\infty \dots \D t $ with
$ \int_0^\infty \int_0^\infty \dots \D t_1 \D t_2 $; and
$ \int_0^1 \dots \D\theta $ with $ \int_0^1 \int_0^1 \dots \D\theta_1 \D\theta_2
$.

However, the proof of Prop.~\ref{4.16} suggests that the random fields
$G^{+,\beta}-G^{0,\beta}$ should be estimated on the set
$ \{(t_1,t_2):t_1>|t_2|\} $ (see $ \xi_{t_1,t_2}, \eta_{t_1,t_2} $ and $ K_1 $
in the proof of Prop.~\ref{4.16}) rather than $ [0,\infty)\times[0,\infty)
$. Thus, we adapt \eqref{5.7} to a given Lebesgue measurable subset $ A $ of
$ \R^2 $ as follows:
\begin{multline}\label{5.8}
\Ex \exp \frac1{M^2} \int_A | \psi(\rho G^{\al,\be}_t) - \psi(\rho G^{0,\be}_t)
 | \, \D t \le \\
\int_{(0,1)^2} \Ex \exp \sum_{t\in A_\theta} g \Big( \frac \rho 2 \(
 G^{\al,\be}_t - G^{0,\be}_t \) \Big) \, \D\theta \, ,
\end{multline}
where
\[
A_\theta = A \cap M ( \Z^2+\theta ) \quad \text{for } \theta \in (0,1)^2
\]
and $ M(\Z^2+\theta) = \{ (M(k_1+\theta_1),M(k_2+\theta_2)) : k_1,k_2\in\Z \} $
for $ \theta=(\theta_1,\theta_2) $ (and $\Z$ is the set of all integers).

The proof is nearly the same; most of it deals with sums over (finite subsets
of) $ A_\theta $ for a given $\theta$, treating $ A_\theta $ as just a (finite
or) countable subset of $\R^2$; and finally (recall the two displays before \eqref{3.13}) we
note that generally
\begin{multline*}
\int_{(0,1)^2} \sum_{t\in A_\theta} f (t) \, \D\theta =
 \int_{(0,1)^2} \sum_{k\in\Z^2} f (M(k+\theta)) \One_A (M(k+\theta)) \, \D\theta
 = \\
\sum_{k\in\Z^2} \int_{(0,1)^2} \!\!\!\! f (M(k+\theta)) \One_A
 (M(k+\theta)) \, \D\theta = \!
\int_{\R^2} \!\! f(Mt) \One_A(Mt) \, \D t = \frac1{M^2} \! \int_A f(t) \, \D t
\end{multline*}
for measurable $ f : [0,\infty) \to [0,\infty) $; and by convexity of $ \exp $,
\[
\exp \frac1{M^2} \int_A f(t) \, \D t = \exp \int_{(0,1)^2} \sum_{t\in A_\theta}
f (t) \, \D\theta \le \int_{(0,1)^2} \exp \sum_{t\in A_\theta} f
(t) \, \D\theta \, .
\]

Here is a two-dimensional counterpart of Lemma \ref{3.14}.

\begin{lemma}\label{5.9}
(a)
Let $ \phi \in  L_2(\R^2) $ satisfy $ \int_{R^2} \phi^2(t) \, \D t = 1 $ and $
\esssup_{t\in\R^2} (1+|t|)^\bal \phi^2(t) < \infty $ for some $ \bal > 4 $; and $
\psi : \R \to \R $ be Lebesgue measurable, satisfying $ \int_\R \psi(u) \gR(\D u) =
0 $. Then there exists $ C_0 $ (dependent only on $\bal$ and the $\esssup$) such
that for every $ M \ge C_0 $ the following $5=1+4$ conditions are sufficient for
\splittability{M^2} of the random field $ \psi(\phi*w) $:
\[
\tag{1}  \int_{-\infty}^\infty \exp \frac1{M^2} |\psi(u)| \, \gR(\D u) \le 2 \, ;
\]
\[
\tag{$2_1$}  \int_0^1 \Ex \exp \sum_{k=0}^\infty g ( G^{(1)}_{\al(\theta+k)M,0}
) \, \D\theta \le 2 \quad \text{for } \al=\pm1 \, ;
\]
\[
\tag{$2_2$}  \int_0^1 \Ex \exp \sum_{k=0}^\infty g ( G^{(2)}_{0,\be(\theta+k)M}
) \, \D\theta \le 2 \quad \text{for } \be=\pm1 \, ;
\]
\[
\tag{$2_3$} \int_{(0,1)^2} \Ex \exp \sum_{t\in A^{\al,\be}_\theta}
g ( G^{(1)}_t ) \, \D\theta \le  2 \quad \text{for } \al,\be=\pm1 \, ,
\]
where\footnote{%
 $ A^{\al,\be}_\theta = A^{\al,\be} \cap M ( \Z^2+\theta ) $, of course; we do
 not repeat this convention.}
$ A^{\al,\be} = \{ (t_1,t_2) \in \R^2 : |t_1|>|t_2|, \al t_1>0, \be t_2>0 \} $;
\[
\tag{$2_4$} \int_{(0,1)^2} \Ex \exp \sum_{t\in B^{\al,\be}_\theta}
 g ( G^{(2)}_t ) \, \D\theta \le 2 \quad \text{for } \al,\be=\pm1 \, ,
\]
where $ B^{\al,\be} = \{ (t_1,t_2) \in \R^2 : |t_1|<|t_2|, \al t_1>0, \be t_2>0 \} $.

Above, $ G^{(1)}, G^{(2)} $ are as in Lemma \ref{4.10},
and $ g : \R \to \R $ is defined by\footnote{%
 Here $ \gR\(\D(x-h)\) $ means $ (2\pi)^{-1/2} \E^{-(x-h)^2/2} \, \D x $; that
 is, $ \ga_{\scriptscriptstyle\R,h}(\D x) $, where $
 \ga_{\scriptscriptstyle\R,h} $ is $ \gR $ shifted by $ h $.}
\[
\tag{3} g(y) = \sup_{h \in \R} \, \log \int_{-\infty}^\infty \exp 16 \bigg|
\psi \Big( \frac{x+y}{2} \Big) - \psi \Big( \frac{x-y}{2} \Big) \bigg| \,
\gR\(\D(x-h)\) \, .
\]

(b)
The same holds in the complex-valued setup, with the following trivial
modifications: $ \phi \in L_2(\R^2\to\C) $; $|\phi(t)|^2$ (rather than
$\phi^2(t)$); $ \psi : \C \to \R $; $ \int_\C \dots \gC(\dots) $ (rather than $
\int_\R \dots \gR(\dots) $); $ \wC $ (rather
than $w$); $ g : \C \to \R $; and $ h,y \in \C $.
\end{lemma}

\begin{sloppypar}
\begin{proof}
\textsc{The real-valued case} (a).
Given $\bal$ and the $\esssup$, Lemma \ref{5.5} for $ \eps = \sqrt2 - 1 $ gives
$C_0$ such that the two conditions on $\phi$ ensure \noisiness{\gR} of the
two-dimensional array $ ( \sqrt2 G_{kM,\ell M} )_{k,\ell\in\{1,\dots,n\}} $ for
all $ M \ge C_0 $ and all $n$. WLOG, $ C_0 \ge 1 $.

Then \noisiness{\gR} of $ ( \sqrt2 G_{(k_1+\theta_1)M,(k_2+\theta_2)M}
)_{k_1,k_2\in\{-n,-n+1,\dots,n\}} $
is ensured as well (due to stationarity of $ G = \phi * w $ and arbitrariness on
$n$). Also one-dimensional arrays $ ( \sqrt2 G_{(k_1+\theta_1)M,t_2}
)_{k_1\in\{-n,\dots,n\}} $ and $ ( \sqrt2 G_{t_1,(k_2+\theta_2)M}
)_{k_2\in\{-n,\dots,n\}} $ are \noisy{\gR} for all $t_1,t_2$ (each being a part
of a noisy two-dimensional array).

By Lemma \ref{5.4} we may replace $ \sqrt2 G $ with $ G^{\al,\be} + G^{0,\be}
$ in all these arrays. Similarly, we may use $ G^{\al,\be} + G^{\al,0} $.

We introduce $ \psi_{16} : \R \to \R $ by $ \psi_{16}(x) = 16\psi\(\frac{x}2\) $ and
use it instead of $\psi$ in \eqref{5.8}, with $ \rho=2 $, taking into account
that $ (G^{\al,\be}_t-G^{0,\be}_t)_{t\in A^{\al,\be}} $ is distributed like $
(G^{(1)}_t)_{t\in A^{\al,\be}} $ irrespective of $ \be \in \{-1,0,1\} $; we get
\begin{align}\label{5.10a}
& \Ex \exp \frac{16}{M^2} \int_{A^{\al,\be}} | \psi(G^{\al,\be}_t)
 - \psi(G^{0,\be}_t) | \, \D t \le \text{LHS}(2_3) \le 2 \, , \\
\label{5.10b}
& \Ex \exp \frac{16}{M^2} \int_{A^{\al,\be}} | \psi(G^{\al,0}_t)
 - \psi(G^{0,0}_t) | \, \D t \le \text{LHS}(2_3) \le 2
\end{align}
(here ``$ \text{LHS}(2_3) $'' means the left-hand side of $(2_3)$),
since \eqref{3.6} is satisfied by $ \psi_{16} $ and $g$ due to (3). The
inequality above results from \noisiness{\gR} of the array $
(G^{\al,\be}+G^{0,\be})|_I $ for every finite $ I \subset A^{\al,\be}_\theta
$.

Using one-dimensional arrays and \eqref{3.13} (with $\psi_{16}$ instead of
$\psi$, and $ \rho=2 $) we get
\begin{equation}\label{5.11}
\Ex \exp \frac{16}M \int_{\al t>0} | \psi(G^{\al,\be}_{t_1,0})
 - \psi(G^{0,\be}_{t_1,0}) | \, \D t_1 \le \text{LHS}(2_1) \le 2 \, .
\end{equation}
This inequality is obtained for all $ \be \in \{-1,0,1\} $, but note that the
left-hand side does not depend on $\be$, since the distribution of
the \component{3} random field $ (G^{0,\be},G^{-,\be},G^{+,\be}) = G^{\doT,\be}
$ does not depend on $\be$, and $G^{\doT,*}$ is just something distributed like
$ G^{\doT,\be} $ for some $\be$, therefore, for all $\be$ (recall
Sect.~\ref{4a}, near the pictures); thus, we may write ``$*$'' instead of
``$\be$'' in \eqref{5.11}.

We inspect Conditions (a1)--(a4) of
Prop.~\ref{4.4} (for $ C = M^2 \ge 1 $), since the other conditions, (s1)--(s3)
and (b), are ensured similarly to the proof of Prop.~\ref{4.16}, as noted before
Lemma~\ref{5.3}. Recall that $ X^{\al,\be} = \psi(G^{\al,\be}) $ and
\begin{gather*}
X^{(1)}_{t_1,t_2} = X^{\sgn t_1,*}_{t_1,t_2} - X^{0,*}_{t_1,t_2} \, ; \quad
 X^{(2)}_{t_1,t_2} = X^{*,\sgn t_2}_{t_1,t_2} - X^{*,0}_{t_1,t_2} \, ; \\
X^{(1,2)}_{t_1,t_2} = X^{\sgn t_1,\sgn t_2}_{t_1,t_2} - X^{0,\sgn t_2}_{t_1,t_2}
 - X^{\sgn t_1,0}_{t_1,t_2} + X^{0,0}_{t_1,t_2} \, ,
\end{gather*}
as in the proof of Prop.~\ref{4.16} (this is irrelevant to (a1), but relevant to
(a2), (a3), (a4)).

(a1) $ \exp \frac1{M^2} \int_{[0,1)\times[0,1)} |X_t| \, \D t \le
\int_{[0,1)\times[0,1)} \exp \frac1{M^2} |X_t| \, \D t $, and for every $t$,
$ \Ex \exp \frac1{M^2} |X_t| = \int_\R \exp \frac1{M^2} |\psi(u)| \, \gR(\D
u) \le 2 $ by (1).

(a2) First, $ \exp \frac1{M^2} \int_0^1 \( \int_\R |X^{(1)}_{t_1,t_2}| \, \D
t_1 \) \, \D t_2 \le \int_0^1 \exp \( \frac1{M^2} \int_\R
|X^{(1)}_{t_1,t_2}| \, \D t_1 \) \, \D t_2 $;
second, $ \exp \frac1{M^2} \int_\R |X^{(1)}_{t_1,t_2}| \, \D t_1
= \exp \frac12 \( \frac2{M^2} \int_{-\infty}^0 | X^{-,*}_{t_1,t_2} -
X^{0,*}_{t_1,t_2} | \, \D t_1 + \frac2{M^2} \int_0^\infty | X^{+,*}_{t_1,t_2} -
X^{0,*}_{t_1,t_2} | \, \D t_1 \) \le \frac12 \sum_\al \exp \frac2{M^2} \int_{\al
t>0} | X^{\al,*}_{t_1,t_2} - X^{0,*}_{t_1,t_2} | \, \D t_1 $;
take the expectation (it does not depend on $t_2$ due to \ref{4.4}(s2)),
use \eqref{5.11} and note that $ \frac2{M^2} \le \frac{16}M $.

(a3) is similar to (a2); swap the coordinates $t_1,t_2$ (and the corresponding
upper indices), and modify \eqref{5.11} using $(2_2)$ instead of $(2_1)$. 

(a4) Similarly to the proof of Prop.~\ref{4.16},
\[
|X^{(1,2)}_{t_1,t_2}| \le
\begin{cases}
 |\psi(a)-\psi(b)| + |\psi(c)-\psi(d)| &\text{when } |t_1|>|t_2|,\\
 |\psi(a)-\psi(c)| + |\psi(b)-\psi(d)| &\text{when } |t_1|<|t_2|,
\end{cases}
\]
where $ a = G^{\sgn t_1,\sgn t_2}_{t_1,t_2} $, $ b = G^{0,\sgn t_2}_{t_1,t_2} $,
$ c = G^{\sgn t_1,0}_{t_1,t_2} $, $ d = G^{0,0}_{t_1,t_2} $. We have
$ \exp \frac4{M^2} \iint_{|t_1|>|t_2|} |\psi(c)-\psi(d)|  \, \D t_1 \D
t_2 = \exp \frac14 \sum_{\al,\be} \frac{16}{M^2} \int_{A^{\al,\be}}
| \psi(G_t^{\al,0}) - \psi(G_t^{0,0}) | \, \D
t \le \frac14 \sum_{\al,\be} \exp \frac{16}{M^2} \int_{A^{\al,\be}} \dots $;
by \eqref{5.10b},
\[
\Ex \exp \frac4{M^2} \iint_{|t_1|>|t_2|} |\psi(c)-\psi(d)| \, \D t_1 \D
t_2 \le 2 \, .
\]
Similarly, using \eqref{5.10a},
\[
\Ex \exp \frac4{M^2} \iint_{|t_1|>|t_2|} |\psi(a)-\psi(b)| \, \D t_1 \D
t_2 \le 2
\]
(just consider $ \int_{A^{\al,\be}} | \psi(G_t^{\al,\be}) - \psi(G_t^{0,\be})
| \, \D t $).

Further,
$ \exp \frac2{M^2} \iint_{|t_1|>|t_2|} |X^{(1,2)}_{t_1,t_2}| \, \D t_1 \D
t_2 \le
\exp \frac12 \( \frac4{M^2} \iint_{|t_1|>|t_2|} |\psi(a)-\psi(b)| \, \D t_1 \D
t_2 +
\frac4{M^2} \iint_{|t_1|>|t_2|} |\psi(c)-\psi(d)| \, \D t_1 \D t_2 \) \le
\frac12 \exp \frac4{M^2} \iint_{|t_1|>|t_2|} |\psi(a)-\psi(b)| \, \D t_1 \D
t_2 +
\frac12 \exp \frac4{M^2} \iint_{|t_1|>|t_2|} |\psi(c)-\psi(d)| \, \D t_1 \D
t_2 $, thus,
\[
\Ex \exp \frac2{M^2} \iint_{|t_1|>|t_2|} |X^{(1,2)}_{t_1,t_2}| \, \D t_1 \D
t_2 \le 2 \, .
\]
Similarly, using ($2_4$) instead of ($2_3$),
\[
\Ex \exp \frac2{M^2} \iint_{|t_1|<|t_2|} |X^{(1,2)}_{t_1,t_2}| \, \D t_1 \D
t_2 \le 2 \, .
\]
Finally,
$ \exp \frac1{M^2} \int_{\R^2} |X^{(1,2)}_t| \, \D t
= \exp \frac12 \( \frac2{M^2} \iint_{|t_1|>|t_2|} \dots
+ \frac2{M^2} \iint_{|t_1|<|t_2|} \dots \) \le \frac12 \exp \frac2{M^2} \iint_{|t_1|>|t_2|} \dots
+ \frac12 \exp \frac2{M^2} \iint_{|t_1|<|t_2|} \dots $, thus,
\[
\Ex \exp \frac1{M^2} \int_{\R^2} |X^{(1,2)}_t| \, \D t \le 2 \, .
\]

\textsc{The complex-valued case} (b) is quite similar.
\end{proof}
\end{sloppypar}

\begin{remark}\label{5.13}
Replacing ``$\le2$'' with the weaker ``$<\infty$'' in ($1$), ($2_1$), \dots,
($2_4$) and waiving the coefficient $16$ in (3), we still get splittability
(rather than \splittability{C} with the given $C$, of course). This applies to
both cases, (a) and (b). (The proof given in Remark \ref{3.15a} needs only
trivial modifications.)
\end{remark}

Here is a counterpart of Prop.~\ref{3.155}.

\begin{theorem}\label{5.14}
(a)
Let positive numbers $ C, \bal, \de $ satisfy $ \de \le 1 $ and $ \bal >
2(1+\frac2\de) $. Let $ \phi \in L_2(\R^2) $ satisfy $ \int_{\R^2} \phi^2(t) \, \D
t = 1 $ and $ \esssup_{t\in\R^2} (1+|t|)^\bal \phi^2(t) < \infty $. Let a Lebesgue
measurable function $ \psi : \R \to \R $ satisfy $ \int_\R \exp \( \frac1C
|\psi(u)| \) \E^{-u^2/2} \, \D u < \infty $, $ \int_\R \psi(u) \E^{-u^2/2} \, \D u = 0 $, and
\[
\log \int_\R \exp \bigg| \psi \Big( \frac{x+h+y}{2} \Big) - \psi \Big(
\frac{x+h-y}{2} \Big) \bigg| \cdot \frac1{\sqrt{2\pi}} \E^{-x^2/2} \, \D x \le C|y|^\de
\]
for all $ y,h \in \R $. Then the following random field is splittable:
\[
X = \psi(\phi*w) \, , \quad \text{that is,} \quad X_t = \psi \bigg( \int_{\R^2}
\phi(t-s) w(s) \, \D s \bigg) \, .
\]

(b)
Let positive numbers $ C, \bal, \de $ satisfy $ \de \le 1 $ and $ \bal >
2(1+\frac2\de) $. Let $ \phi \in L_2(\R^2\to\C) $ satisfy $ \int_\R |\phi(t)|^2 \, \D
t = 1 $ and $ \esssup_{t\in\R^2} (1+|t|)^\bal |\phi(t)|^2 < \infty $. Let a Lebesgue
measurable function $ \psi : \C \to \R $ satisfy $ \int_\C \exp \( \frac1C
|\psi(z)| \) \E^{-|z|^2} \, \D z < \infty $, $ \int_\C \psi(z) \E^{-|z|^2} \, \D z = 0 $, and
\[
\log \int_\C \exp \bigg| \psi \Big( \frac{x+h+y}{2} \Big) - \psi \Big(
\frac{x+h-y}{2} \Big) \bigg| \cdot \frac1\pi \E^{-|x|^2} \, \D x \le C|y|^\de
\]
for all $ y,h \in \C $. Then the following random field is splittable:
\[
X = \psi(\phi*\wC) \, , \quad \text{that is,} \quad X_t = \psi \bigg( \int_{\R^2}
\phi(t-s) \wC(s) \, \D s \bigg) \, .
\]
\end{theorem}

\begin{proof}
\textsc{The real-valued case} (a) According to Lemma~\ref{5.9} and
Remark \ref{5.13}, it is sufficient to prove Conditions ($2_1$), \dots,
($2_4$) of Lemma~\ref{5.9}
with ``$<\infty$'' instead of ``$\le2$'', and waiving ``$16$'' in (3). Given $
g(y) \le C|y|^\de $, we check that
\[ \tag{$2'_1$}
\int_0^1 \Ex \exp C \sum_{k=0}^\infty |
G^{(1)}_{\al(\theta+k)M,0} |^\de \, \D\theta < \infty \, ,
\]
and similarly, ($2'_2$), ($2'_3$), ($2'_4$).

Applying Lemma \ref{3.2.6} to the set $ \{ (\al(\theta+k)M,0) : k=0,1,\dots \} $
treated as a countable measure space with the counting measure (according to
Remark \ref{3.2.7}) we reduce ($2'_1$) to
\[
\tag{$2''_1$} \esssup_{\theta\in(0,1)} \sum_{k=0}^\infty \( \Ex (
G^{(1)}_{\al(\theta+k)M,0} )^2 \)^{\de/2} < \infty \, .
\]
Likewise, ($2'_3$) reduces to
\[
\tag{$2''_3$} \esssup_{\theta\in(0,1)^2} \sum_{t\in A^{\al,\be}_\theta} \( \Ex (
G^{(1)}_t )^2 \)^{\de/2} < \infty \, ;
\]
and ($2'_2$), ($2'_4$) reduce similarly.

We proceed as in (the end of) the proof of Prop.~\ref{4.16}. By
Lemma \ref{4.10}(a), $ G^{(1)} $ is distributed like $ \phi^{(1)} \star w $;
thus $ \Ex (G_{t_1,t_2}^{(1)})^2 = \iint_{\R^2} \( \phi^{(1)}
(t_1,s_1,t_2-s_2) \)^2 \, \D s_1 \D s_2 = 2 \iint_{\sgn s_1 \ne \sgn
t_1} \phi^2(t_1-s_1,t_2-s_2) \, \D s_1 \D s_2 $, which reduces $2''_3$ to
\[
\tag{$2'''_3$} \esssup_{\theta\in(0,1)^2} \sum_{t\in A^{\al,\be}_\theta} \Big(
\iint_{\sgn s_1 \ne \sgn t_1} \!\!\!\!\!\!\!\!
\( 1 + \sqrt{ (t_1-s_1)^2 + (t_2-s_2)^2 } \)^{-\bal} \, \D s_1 \D
s_2 \Big)^{\de/2} < \infty \, ;
\]
this inequality is ensured by Remark~\ref{4.13c} (to Lemma~\ref{4.13}) applied to
$ K_1 = \{ (t_1,t_2) \in \R^2 : |t_1| \ge |t_2|, \al t_1 \ge 0, \be t_2 \ge 0 \} $,
$ K_2 = \{ (s_1,s_2) \in \R^2 : \al s_1 \le 0 \} $. The same argument gives
$2''_1$; just take $ K_1 = \{ (t_1,0) : \al t_1 \ge 0 \} $.

Thus, $2_1$ and $2_3$ are proved; for $2_2$ and $2_4$ swap the coordinates
$t_1,t_2$ (and the corresponding upper indices).

\textsc{The complex-valued case} (b) is quite similar.
\end{proof}

\subsection{More general than stationarity}
\label{5c}

It happens sometimes that a complex-valued random field $X$ is not stationary,
and nevertheless the field $ \psi(X) $ is stationary (since $\psi$ need not be
one-to-one).

Here is a generalization of Prop.~\ref{4.4}. Denote the circle $ \{ z \in \C :
|z|=1 \} $ by $ \T $.

\begin{proposition}\label{5.15}
Let $ X^{\doT,\doT} $ be a \splt{2} of a complex-valued random field $ X $ on $ \R^2 $. Then
the following $8=3+4+1$ conditions (together; on $X$ and $X^{(1)}$, $X^{(2)}$,
$X^{(1,2)}$ introduced before) are sufficient for \splittability{C} of $X$:
\begin{description}
\item[\rm(s1)] there exists a Borel measurable function $ \eta : \R^2 \times
\R^2 \to \T $ such that for every $ r \in \R^2 $ the random field $ ( \eta(r,t)
X_{t+r} )_{t\in\R^2} $ is distributed like $X$;
\item[\rm(s2)] there exists a Borel measurable function $ \eta_1 : \R \times
\R^2 \to \T $ such that for every $ r \in \R $ the random field
$ ( \eta_1(r;t_1,t_2 ) X^{(1)}_{t_1,t_2+r})_{t_1,t_2} $ is distributed like
$X^{(1)}$;
\item[\rm(s3)] there exists a Borel measurable function $ \eta_2 : \R \times
\R^2 \to \T $ such that for every $ r \in \R $ the random field
$ ( \eta_2(r;t_1,t_2 ) X^{(2)}_{t_1+r,t_2})_{t_1,t_2} $ is distributed like
$X^{(2)}$;
\end{description}

(a1)--(a4), (b) as in Prop.~\ref{4.4}.
\end{proposition}

\begin{lemma}\label{5.16}
(a) If $X$ satisfies \ref{5.15}(s1) and $Y$ is a leak for $X$, then for every $
r \in \R^2 $ the field $ \overline{\eta(r,\cdot)} Y
= \( \overline{\eta(r,t)} Y_t \)_t $ is a leak for $ (X_{t+r})_t $.

(b) If $X$ satisfies \ref{5.15}(s2) and $Y$ is a leak for
$(X^{(1)}_{t_2,t_1})_{t_1,t_2}$, then for every $ r \in \R $ the field
$ \( \overline{\eta_1(r;t_2,t_1)} Y_{t_1,t_2} \)_{t_1,t_2} $ is a
leak for $ (X^{(1)}_{t_2,t_1+r})_{t_1,t_2} $.

(c) If $X$ satisfies \ref{5.15}(s3) and $Y$ is a leak for
$(X^{(2)}_{t_1,t_2})_{t_1,t_2}$, then for every $ r \in \R $ the field
$ \overline{\eta_2(r,\cdot)} Y = \( \overline{\eta_2(r,t)} Y_t \)_t $ is a 
leak for $ (X^{(2)}_{t_1+r,t_2})_{t_1,t_2} $.
\end{lemma}

\begin{proof}
(a)
We have $ Y = \Leak_1(X^0,X^-,X^+) $ for some split $(X^0,X^-,X^+)$ of $X$. By
(s1), $ ( \eta(r,t) X_{t+r} )_{t\in\R^2} $ is distributed like $X$, thus, $
(X_{t+r})_t $ is distributed like $ \overline{\eta(r,\cdot)} X $. The split
$ \( \overline{\eta(r,\cdot)} X^0, \overline{\eta(r,\cdot)}
X^-, \overline{\eta(r,\cdot)} X^+ \) $ of $ \overline{\eta(r,\cdot)} X $ is also
a split of $ (X_{t+r})_t $. Its leak $ \overline{\eta(r,\cdot)} Y $ is a leak
for $ (X_{t+r})_t $.

(b), (c) are proved similarly.
\end{proof}

\begin{remark}\label{5.17}
As noted in Section \ref{sect2} (on page 6) for dimension one, shifting back the
leak of a shifted split we get the leak around a point. In dimension two,
defining a shift transformation $ F_r $ (for $ r \in \R^2 $) by $ F_r(X) =
(X_{t+r})_t $, we observe the following. If $Y$ is a leak for $ F_r(X) $, then $
F_{-r}(Y) $ is a leak for $X$ around the line $ t_1=r_1 $. More explicitly: $
(X^0,X^-,X^+) $ is a split of $X$ if and only if $ \( F_r(X^0), F_r(X^-),
F_r(X^+) \) $ is a split of $ F_r(X) $; and if it is, and $ Y_{t_1,t_2} = \(
F_r(X^{\sgn t_1})-F_r(X^0) \)_{t_1,t_2} =  X^{\sgn t_1}_{t+r}-X^0_{t+r} $ is the
corresponding leak for $ F_r(X) $, then $ \( F_{-r}(Y) \)_{t_1,t_2} =
Y_{t_1-r_1,t_2-r_2} = X^{\sgn(t_1-r_1)}_{t_1,t_2} - X^0_{t_1,t_2} $; this is what
we call the leak (for $X$) around $ t_1=r $.

\noindent\hspace{3mm} Thus, in \ref{5.16}(a), the field $ F_{-r} \( \overline{\eta(r,\cdot)} Y \)
= \( \overline{ \eta(r;t_1-r_1,t_2-r_2)} Y_{t_1-r_1,t_2-r_2})_{t_1,t_2} $ is a
leak for $X$ around $ t_1=r $. Also, defining the (involutive) swap
transformation $ \Fswap $ by $ \( \Fswap(X) \)_{t_1,t_2} = X_{t_2,t_1} $, we
observe that $ \Leak_1(\Fswap(\dots)) = \Fswap(\Leak_2(\dots)) $; that is,
swapping (back) the leak of the swapped split, we get the leak around the other
axis.

If $Y$ is a leak for $ F_r(X) $ around $ t_2=0 $, then $ F_{-r} (Y) $ is a leak
for $X$ around $ t_2=r_2 $.

If $X$ satisfies \ref{5.15}(s1) (with some $\eta$), then $ \ti X = \Fswap(X) $
satisfies it as well (with $ \ti\eta :
(r_1,r_2;t_1,t_2) \mapsto \eta(r_2,r_1;t_2,t_1) $). And if $Y$ is a leak for
$ \ti X $, then $ \overline{ \ti\eta(r,\cdot) } Y $ is a leak for $ F_r(\ti X) $
by \ref{5.16}(a). Thus, $ \Fswap \( \overline{ \ti\eta(r,\cdot) } Y \) $ is a
leak for $ \Fswap \( F_r(\ti X) \) $ around $ t_2=0 $. That is (renaming $
r_1,r_2 $ into $ r_2,r_1 $), $ \( \overline{ \eta(r_1,r_2;t_1,t_2) }
Y_{t_2,t_1} \)_{t_1,t_2} $ is a leak for $ (X_{t_1+r_1,t_2+r_2})_{t_1,t_2} $
around $ t_2=0 $. And so, $ \( \overline{ \eta(r_1,r_2;t_1-r_1,t_2-r_2) }
Y_{t_2-r_2,t_1-r_1} \)_{t_1,t_2} $ is a leak for $X$ around $ t_2=r $.
\end{remark}

\begin{proof}[Proof of Prop.~\ref{5.15}]
We follow the proof of Prop.~\ref{4.4}, note each argument based on
Conditions \ref{4.4}(s1,s2,s3) and replace it with an argument based on
Conditions \ref{5.15}(s1,s2,s3) and Lemma~\ref{5.16}.

``Item (a) there follows from (a2) and (s2) here.''
Only $ |X^{(1)}| = \( |X^{(1)}_t| \)_t $ is relevant; and $
\( |X^{(1)}_{t_1,t_2+r}| \)_{t_1,t_2} $ is distributed like $ |X^{(1)}|  $
by \ref{5.15}(s2).\footnote{%
 Note that $ X^{(1)}_{t_2,t_1} $ is now substituted for $ X_{t_1,t_2} $
 of \cite[Def.~1.3]{IV}.}

``WLOG, $r=0$ due to (s2).'' Lemma \ref{5.16}(b) converts the given leak $
\( X_{t_2,t_1}^{(1,2)} \)_{t_1,t_2} $ for $ \( X_{t_2,t_1}^{(1)} \)_{t_1,t_2} $
into a leak for $ \( X^{(1)}_{t_2,t_1+r} \)_{t_1,t_2} $ of the same absolute
value.

``Item (a) there follows from (a1) and (s1) here.'' Again, only $|X|$ is
relevant; and it is stationary by \ref{5.15}(s1).

``Case $i=1$ \dots\ WLOG, $r=0$ due to (s1).'' 
Again, Lemma \ref{5.16}(a) converts the given leak $ X^{(1)} $ for $ X $ into a
leak for $ \( X_{t_1+r,t_2} \)_{t_1,t_2} $ multiplying by a function
$ \R^2 \to \T $; such multiplication preserves \splittability{(1,1)}
by \cite[Prop.~2.1]{IV} (generalized to the complex-valued setup).

``Case $i=2$ \dots\ WLOG, $r=0$ due to (s1).'' 
The same as above.
\end{proof}

We turn to complex-valued Gaussian random fields. Let $ \phi
: \R^2 \times \R^2 \to \C $ satisfy \eqref{4.6}, and $ G = \phi \star \wC $.

\begin{lemma}\label{5.18}
Conditions \ref{5.15}(s1,s2,s3) are satisfied by $ X = G = \phi \star \wC $
whenever there exist Borel measurable $ \tau_1, \tau_2 : \R^2 \times \R^2 \to \T
$ such that\footnote{%
 Still, $ \T = \{ z\in\C : |z|=1 \} $.}
\begin{equation}\label{5.19}
 \begin{gathered}
  \text{for all } r \in \R^2, \text{ for almost all } s,t \in \R^2 \text{
   holds}\\
  \phi(t+r,s+r) = \tau_1(r,t) \phi(t,s) \tau_2(r,s) \, .
 \end{gathered}
\end{equation}
\end{lemma}

\begin{sloppypar}
\begin{proof}
Below ``$X_t\sim Y_t$'' means that $(X_t)_t$ and $(Y_t)_t$ are identically
distributed.

(s1)
$ \eta(r,t) X_{t+r} = \eta(r,t) \iint \phi(t+r,s) \wC(s) \, \D s \sim
\eta(r,t) \iint \phi(t+r,s+r) \wC(s) \, \D s =
\eta(r,t) \iint \tau_1(r,t) \phi(t,s) \tau_2(r,s) \wC(s) \, \D s \sim
\eta(r,t) \tau_1(r,t) \iint \phi(t,s) \wC(s) \, \D s =
\eta(r,t) \tau_1(r,t) X_t $; it remains to take $ \eta(r,t)
= \overline{\tau_1(r,t)} $.

(s2)
By Lemma~\ref{4.10}, $ G^{(1)} \sim \phi^{(1)} \star \wC $, and
$ \phi^{(1)}(t_1,t_2;s_1,s_2) $ is $ \phi(t_1,t_2;s_1,s_2) $ times a factor
dependent on $ t_1,s_1 $ only. Thus, \eqref{5.19} holds for $ \phi^{(1)} $
provided that $ r_1=0 $, which is enough for getting (s2) by the argument used
above for (s1).

(s3) is similar to (s2).
\end{proof}
\end{sloppypar}

It follows from \eqref{5.19} that $ |\phi(t+r,s+r)| = |\phi(t,s)| $ for almost
all $s,t$; thus the function $ t \mapsto \int_{\R^2} |\phi(t,s)|^2 \, \D s $ is
constant a.e.\ (and therefore locally integrable as required by
\eqref{4.6}). From now on we assume that \eqref{5.19} holds and the constant
function is $1$, that is,
\begin{equation}\label{5.20}
\int_{\R^2} |\phi(t,s)|^2 \, \D s = 1 \quad \text{for almost all } t \in \R^2 \,
;
\end{equation}
and in addition,
\begin{equation}\label{5.21}
\esssup_{(s,t)\in\R^2\times\R^2} (1+|t-s|)^\bal |\phi(t,s)|^2 < \infty \quad
\text{for a given } \bal \, .
\end{equation}

Inequality \eqref{5.8} for $ G = \phi \star \wC $, without stationarity, follows
from \eqref{5.20} and \eqref{5.21} (irrespective of \eqref{5.19}). Indeed,
Lemmas \ref{4.8}, \ref{4.10}, \ref{5.3}, \ref{5.4} and inequality \eqref{5.*} do
not assume stationarity, and Lemma \ref{5.5}(b) generalizes as follows.

\begin{lemma}\label{5.22}
Let $ \bal > 4 $. Then for every $ \eps > 0 $ there exists $ C $ (dependent only
on $ \eps $, $ \bal $ and the $ \esssup $) such that for all $ M \ge C $, all $
\theta_1,\theta_2\in[0,1] $ and all $n$ the array $ \(
(1+\eps)G_{(k+\theta_1)M,(\ell+\theta_2)M} \)_{k,\ell\in\{-n,-n+1,\dots,n\}} $
is \noisy{\gC}.
\end{lemma}

\begin{proof}
See the proofs of Lemma \ref{5.5} and Lemma \ref{3.11}; it is a matter of decay
of the correlation function, $ | \Ex G_u \overline{G_t} | = | \int_{\R^2}
\phi(u,s) \overline{\phi(t,s)} \, \D s | \le \const \cdot (1+|t-s|)^{-\bal/2} $,
as follows from \eqref{5.21}.
\end{proof}

Lemma \ref{5.9}(b) holds. Stationarity is mentioned in its proof (the second
paragraph), but all really needed is stipulated by Lemma \ref{5.22} (used
instead of Lemma \ref{5.5}). And, of course, Prop.~\ref{5.15} is used (in
concert with Lemma \ref{5.18}) instead of Prop.~\ref{4.4}. Remark \ref{5.13}
holds as before, and we arrive at a generalization of Theorem~\ref{5.14}(b).

\begin{theorem}\label{5.23}
Let positive numbers $ C, \bal, \de $ satisfy $ \de \le 1 $ and $ \bal >
2(1+\frac2\de) $. Let a Lebesgue measurable function $ \phi : \R^2\times\R^2 \to \C $
satisfy \eqref{5.19}, \eqref{5.20} and \eqref{5.21}. Let a Lebesgue
measurable function $ \psi : \C \to \R $ satisfy $ \int_\C \exp \( \frac1C
|\psi(z)| \) \E^{-|z|^2} \, \D z < \infty $, $ \int_\C \psi(z) \E^{-|z|^2} \, \D z = 0 $, and
\[
\log \int_\C \exp \bigg| \psi \Big( \frac{x+h+y}{2} \Big) - \psi \Big(
\frac{x+h-y}{2} \Big) \bigg| \cdot \frac1\pi \E^{-|x|^2} \, \D x \le C|y|^\de
\]
for all $ y,h \in \C $. Then the following random field $ \psi(\phi\star\wC) $ is
splittable:
\[
X = \psi(\phi\star\wC) \, , \quad \text{that is,} \quad X_t = \psi \bigg(
\int_{\R^2} \phi(t,s) \wC(\D s) \bigg) \, .
\]
\end{theorem}

Here is a counterpart of Prop.~\ref{3.19}.

\begin{proposition}\label{5.24}
Let a Lebesgue measurable $ \phi : \R^2\times\R^2 \to \C $
satisfy \eqref{5.19}, \eqref{5.20}, and \eqref{5.21} for some $ \bal > 6 $. Then
the following random field $X$ is splittable:
\[
X = \log | \phi \star \wC | + \gEuler/2 \, , \quad \text{that is,} \quad X_t =
\log \bigg| \int_{\R^2} \phi(t,s) \wC(s) \, \D s \bigg| + \frac12 \gEuler \, .
\]
\end{proposition}

\begin{proof}
Similar to the proof of Prop.~\ref{3.19}, using Theorem~\ref{5.23} instead of
Prop. \ref{3.155}(b).
\end{proof}

\subsection{Gaussian Entire Function}

We return to the case introduced in Sect.~\ref{sect1}: $ G_{t_1,t_2} = F^*(z) =
F(z) \E^{-|z|^2/2} $ for $ z = t_1+\I t_2 $, where $F(z)$ is given by
\eqref{1.1}, and $ X_t = \psi(G_t) $ for $ t \in \R^2 $, where $ \psi(z) = \log
|z| + \frac12 \gEuler $.

\begin{theorem}\label{5.25}
The following random field is splittable:
\[
X = \psi(G), \quad \text{that is,} \quad X_t = \log|G_t| + \frac12 \gEuler \;
\text{ for } t \in \R^2.
\]
\end{theorem}

\begin{proof}
The covariance function is (see \cite[Sect.~3.1]{NS-ICM})
\[
\Ex F^*(z) \overline{F^*(w)} = \exp \Big( z\overline w - \frac12 |z|^2 - \frac12
|w|^2 \Big) = \exp \Big( \I\Im(z\overline w) - \frac12|z-w|^2 \Big) \, ,
\]
that is,
\[
\Ex G_s \overline{G_t} = \exp \Big( \! -\I s \wedge t - \frac12 |s-t|^2 \Big)
\]
(here $ s \wedge t = (s_1,s_2) \wedge (t_1,t_2) = s_1 t_2 - s_2 t_1 $).

We define $ \phi : \R^2 \times \R^2 \to \C $ by
\[
\phi(t,s) = \frac1{\sqrt\pi} \exp \Big( \! - \I t \wedge s - \frac12 |s-t|^2 \Big)
\]
and note that the covariance function of $ \phi \star \wC $ is equal to the
covariance function of $G$, that is,
\[
\Ex G_u \overline{G_v} = \int_{\R^2} \phi(u,s) \overline{\phi(v,s)} \, \D s
\quad \text{for all } u,v \in \R^2 \, ,
\]
which is well-known, see for instance \cite[(7.49)]{KS} or \cite[Lemma
6.2]{2008}. Thus, the Gaussian random fields $G$ and $ \phi\star\wC $ are
identically distributed. Clearly, $\phi$ satisfies \eqref{5.20}, and
\eqref{5.21} for every $\bal$ (due to exponential decay). Also, $ \phi $
satisfies \eqref{5.19} for $ \tau_1(r,t) = \exp(\I r\wedge t) $, $ \tau_2(r,s) =
\exp(-\I r\wedge s) $, since
\begin{multline*}
\phi(t+r,s+r) = \frac1{\sqrt\pi} \exp \Big( -\I (t+r) \wedge (s+r) - \frac12
 |s-t|^2 \Big) = \\
\exp(-\I t\wedge r) \phi(t,s) \exp(-\I r\wedge s) \, .
\end{multline*}
By Prop.~\ref{5.24}, $ \psi(\phi\star\wC) $ is splittable, and $ \psi(G) $ as
well.
\end{proof}

\appendixtitleon
\appendixtitletocon
\begin{appendices}

\section[Reader's guide to parts I--V]
  {\raggedright Reader's guide to parts I--V}
\label{appendix}
Here is a guide for a reader interested in the proof of moderate deviations (and
asymptotic normality) for zeroes of the Gaussian Entire Function and not
interested in more general results (at least, for now). Hopefully, it may help
to get insight into the long proof. But do it on your own risk, taking the
burden of filling small gaps and resolving small discrepancies.

In Part I (that is, \cite{I}) read Lemma 2a8, Lemma 2d1, and proofs of
Corollaries 1.7, 1.8 (at the end of Sect.~3b), taking Theorem 1.6 for
granted. Throw away the rest of Part I. If you like, see also Lemma 6.2 of Part
III.

Interpret ``random field'' as ``random element of $L_{1,\text{loc}}(\R^2)$'',
that is, a random equivalence class of locally integrable functions $ \R^2 \to
\R $, for the real-valued case; for the complex-valued case, use $
L_{1,\text{loc}}(\R^2\to\C)$.

For the start, interpret splittability of a stationary random field according to
Proposition \ref{4.4} (no need to go beyond these sufficient conditions), but
first, read Section \ref{sect2} till Remark \ref{2.rem}, and Subsection \ref{4a}.

Here is a top-down sketch of main results of Parts II--V.

\vspace{2pt}

\textsc{Part V, Section 1. } Moderate deviations (and asymptotic normality as
well) for zeroes of the Gaussian Entire Function follow easily from these for a
well-known stationary random field on the complex plane.

\textsc{Part V, Sections 4-5. } The random field mentioned above is splittable
(Theorem \ref{5.25}).

\textsc{Part IV. } Linear response holds for every stationary splittable random
field (Theorem 1.5) and implies moderate deviations (Corollary 1.6), as well as
asymptotic normality (Corollary 1.7). The proof uses Parts II and III.

\textsc{Part III. } Linear response with boxes (rather than test functions)
holds for every stationary splittable random field (Theorem 1.1). The proof uses
Part II.

\textsc{Part II. } An upper bound on the linear response for boxes holds for
every (not just stationary) splittable random field (Theorem 1.5).

\vspace{2pt}

The transition from linear response to moderate deviations is shown in Part I,
Sect. 3b (proof of Corollary 1.7) and detailed in Part III, Lemma 6.2. The
transition from linear response to asymptotic normality is shown in Part I,
Sect. 3b (proof of Corollary 1.8).

The main result of Part IV requires stationarity, and nevertheless, its proof
involves non-stationary random fields (of the form $ \phi(t) X_t $ where $ X_t $
is stationary, see Sect.~5 there). The arguments of Part II are reused when
proving Proposition 3.12 of Part IV (similar to Proposition 3.6 of Part II), to
be applied in Section 4 (of Part IV) when estimating a supremum over random
fields that need not be stationary, see (4.1) there. This is why stationarity is
not assumed in Part II.

For better understanding splittability beyond stationarity (Definition 1.3 of
Part IV), read the six items before that definition, the two footnotes to that
definition, and (in this Part V) Definition \ref{2.spli}. As an exercise,
generalize Proposition \ref{4.4} to nonstationary random fields on the plane.

Interpret uniform splittability (of a family of random fields, in Part II) as
existence of $C$ such that each random field in the family is \splittable{C}.

I apologize for the fuss. Thank you for your interest. Good luck!

\end{appendices}

\bigskip
\filbreak
{
\small
\begin{sc}
\parindent=0pt\baselineskip=12pt
\parbox{4in}{
Boris Tsirelson\\
School of Mathematics\\
Tel Aviv University\\
Tel Aviv 69978, Israel
\smallskip
\par\quad\href{mailto:tsirel@post.tau.ac.il}{\tt
 mailto:tsirel@post.tau.ac.il}
\par\quad\href{http://www.tau.ac.il/~tsirel/}{\tt
 http://www.tau.ac.il/\textasciitilde tsirel/}
}

\end{sc}
}
\filbreak

\end{document}